\documentclass[3p,number,sort&compress]{elsarticle}
\journal{Journal de Math\'{e}matiques Pures et Appliqu\'{e}es}

\usepackage{amsmath,amssymb,nicefrac,bm,upgreek,mathtools,pdf14}
\usepackage{hyperref}
\usepackage[mathscr]{eucal}
\usepackage{dsfont}

\usepackage[T1]{fontenc}
\usepackage[french,english]{babel}
\newenvironment{abstracts}
 {\global\setbox\absbox=\vbox\bgroup
    \hsize=\textwidth
    \linespread{1}\selectfont}
 {\vspace{-\bigskipamount}\egroup}
\renewenvironment{abstract}[1][]
 {\if\relax\detokenize{#1}\relax\else\selectlanguage{#1}\fi
  \noindent\textbf{\abstractname}\par\medskip\noindent\ignorespaces}
 {\par\bigskip}
 
\usepackage{color}
\definecolor{dmagenta}{rgb}{.4,.1,.5}
\definecolor{dblue}{rgb}{.0,.0,.5}
\definecolor{mblue}{rgb}{.0,.0,.8}
\definecolor{ddblue}{rgb}{.0,.0,.4}
\definecolor{dred}{rgb}{.6,.0,.0}
\definecolor{dgreen}{rgb}{.0,.5,.0}
\definecolor{Eeom}{rgb}{.0,.0,.5}

\usepackage[normalem]{ulem}

\newtheorem{lemma}{Lemma}[section]
\newtheorem{theorem}{Theorem}[section]
\newtheorem{proposition}{Proposition}[section]
\newtheorem{corollary}{Corollary}[section]

\newdefinition{definition}{Definition}[section]
\newdefinition{assumption}{Assumption}[section]
\newdefinition{hypothesis}{Hypothesis}[section]
\newdefinition{notation}{Notation}[section]
\newdefinition{example}{Example}[section]
\newdefinition{remark}{Remark}[section]

\newproof{proof}{Proof}

\numberwithin{equation}{section}

\hypersetup{
  colorlinks=true,
  citecolor=dblue,
  linkcolor=dblue,
  frenchlinks=false,
  pdfborder={0 0 0},
  naturalnames=false,
  hypertexnames=false,
  breaklinks}
\usepackage[capitalize,nameinlink]{cleveref}

\crefname{section}{Section}{Sections}
\crefname{subsection}{Subsection}{Subsections}
\crefname{condition}{Condition}{Conditions}
\crefname{hypothesis}{Hypothesis}{Conditions}
\crefname{assumption}{Assumption}{Assumptions}
\crefname{lemma}{Lemma}{Lemmas}

\Crefname{figure}{Figure}{Figures}

\crefformat{equation}{\textup{#2(#1)#3}}
\crefrangeformat{equation}{\textup{#3(#1)#4--#5(#2)#6}}
\crefmultiformat{equation}{\textup{#2(#1)#3}}{ and \textup{#2(#1)#3}}
{, \textup{#2(#1)#3}}{, and \textup{#2(#1)#3}}
\crefrangemultiformat{equation}{\textup{#3(#1)#4--#5(#2)#6}}%
{ and \textup{#3(#1)#4--#5(#2)#6}}{, \textup{#3(#1)#4--#5(#2)#6}}%
{, and \textup{#3(#1)#4--#5(#2)#6}}

\Crefformat{equation}{#2Equation~\textup{(#1)}#3}
\Crefrangeformat{equation}{Equations~\textup{#3(#1)#4--#5(#2)#6}}
\Crefmultiformat{equation}{Equations~\textup{#2(#1)#3}}{ and \textup{#2(#1)#3}}
{, \textup{#2(#1)#3}}{, and \textup{#2(#1)#3}}
\Crefrangemultiformat{equation}{Equations~\textup{#3(#1)#4--#5(#2)#6}}%
{ and \textup{#3(#1)#4--#5(#2)#6}}{, \textup{#3(#1)#4--#5(#2)#6}}%
{, and \textup{#3(#1)#4--#5(#2)#6}}

\crefdefaultlabelformat{#2\textup{#1}#3}
\makeatletter
\DeclareRobustCommand\widecheck[1]{{\mathpalette\@widecheck{#1}}}
\def\@widecheck#1#2{%
    \setbox\z@\hbox{\m@th$#1#2$}%
    \setbox\tw@\hbox{\m@th$#1%
       \widehat{%
          \vrule\@width\z@\@height\ht\z@
          \vrule\@height\z@\@width\wd\z@}$}%
    \dp\tw@-\ht\z@
    \@tempdima\ht\z@ \advance\@tempdima2\ht\tw@ \divide\@tempdima\thr@@
    \setbox\tw@\hbox{%
       \raise\@tempdima\hbox{\scalebox{1}[-1]{\lower\@tempdima\box
\tw@}}}%
    {\ooalign{\box\tw@ \cr \box\z@}}}
\makeatother

\newcommand{\df}{\coloneqq}
\DeclareMathOperator{\Exp}{\mathbb{E}} 
\DeclareMathOperator{\Prob}{\mathbb{P}} 
\newcommand{\D}{\mathrm{d}}          
\newcommand{\E}{\mathrm{e}}          
\newcommand{\RR}{\mathbb{R}}         
\newcommand{\Rd}{{\mathbb{R}^d}}       
\newcommand{\NN}{\mathbb{N}}         
\newcommand{\Ind}{\mathds{1}}            
\newcommand{\Act}{\mathbb{U}}        
\newcommand{\Uadm}{\mathfrak{U}}     
\newcommand{\Usm}{\mathfrak{U}_{\mathrm{SM}}}  
\newcommand{\bUsm}{\overline{\mathfrak{U}}_{\mathrm{SM}}}  
\newcommand{\Ussm}{\mathfrak{U}_{\mathrm{SSM}}}  
\newcommand{\Ulb}{\widehat{\mathfrak{U}}_{\mathrm{SM}}}  
\newcommand{\Ulbs}{\widehat{\mathfrak{U}}_{\mathrm{SSM}}} 

\newcommand{\Sob}{\mathscr{W}}       
\newcommand{\Sobl}{\mathscr{W}_{\mathrm{loc}}}  
\newcommand{\Cc}{C}                  
\newcommand{\Ccl}{C_{\mathrm{loc}}}  
\newcommand{\Lp}{L}                  
\newcommand{\Lpl}{L_{\mathrm{loc}}}  

\newcommand{\str}{^{\!*}}          
\newcommand{\lamstr}{\lambda^{\!*}}

\newcommand{\transp}{^{\mathsf{T}}}  

\newcommand{\Lg}{\mathscr{L}}        
\newcommand{\cL}{\mathcal{L}}        

\newcommand{\uuptau}{\Breve{\uptau}}
\newcommand{\grad}{\nabla}

\newcommand{\cA}{\mathcal{A}}    
\newcommand{\sA}{\mathscr{A}}    
\newcommand{\cB}{\mathcal{B}}    
\newcommand{\sB}{\mathscr{B}}    
\newcommand{\cC}{\mathcal{C}}     
\newcommand{\cD}{\mathcal{D}}     
\newcommand{\sE}{\mathscr{E}}     
\newcommand{\sF}{\mathfrak{F}}    
\newcommand{\sJ}{\mathscr{J}}     
\newcommand{\sK}{\mathscr{K}}     
\newcommand{\sM}{\mathscr{M}}     
\newcommand{\cP}{\mathcal{P}}     
\newcommand{\sR}{\mathscr{R}}     
\newcommand{\Lyap}{\mathscr{V}}   

\newcommand{\abs}[1]{\lvert#1\rvert}
\newcommand{\norm}[1]{\lVert#1\rVert}
\newcommand{\babs}[1]{\bigl\lvert#1\bigr\rvert}

\newcommand{\babss}[1]{\biggl\lvert#1\biggr\rvert}
\newcommand{\bnorm}[1]{\bigl\lVert#1\bigr\rVert}

\DeclareMathOperator*{\argmin}{arg\,min}

\DeclareMathOperator*{\essinf}{ess\,inf}

\DeclareMathOperator*{\supp}{support}

\DeclareMathOperator{\trace}{trace}

\DeclareMathOperator{\sign}{sign}

\begin{document}
\begin{frontmatter}

\title
{Strict monotonicity of principal eigenvalues of elliptic operators
in \texorpdfstring{$\Rd$}{}\\ and risk-sensitive control}

\author[ut]{Ari Arapostathis\corref{cor1}}
\cortext[cor1]{Corresponding author}
\ead{ari@ece.utexas.edu}

\author[iiser]{Anup Biswas}
\ead{anup@iiserpune.ac.in}

\author[iitg]{Subhamay Saha}
\ead{saha.subhamay@iitg.ernet.in}

\address[ut]{Department of ECE,
The University of Texas at Austin, 2501 Speedway, EER 7.824,
Austin, TX~78712, USA}

\address[iiser]{Department of Mathematics,
Indian Institute of Science Education and Research,\\ 
Dr. Homi Bhabha Road, Pune 411008, India}

\address[iitg]{Department of Mathematics,
Indian Institute of Technology Guwahati, Assam 781039, India}

\begin{abstracts}
\begin{abstract}
This paper studies the eigenvalue problem on $\Rd$
for a class of second order, elliptic operators of the form
$\Lg^f = a^{ij}\partial_{x_i}\partial_{x_j} + b^{i}\partial_{x_i} + f$,
associated with non-degenerate diffusions.
We show that strict monotonicity of the principal eigenvalue 
of the operator with respect to the potential function $f$
fully characterizes the 
ergodic properties of the associated ground state diffusion, and the
unicity of the ground state, and we present a comprehensive study of the eigenvalue
problem from this point of view.
This allows us to extend or strengthen various results in the literature
for a class of viscous Hamilton--Jacobi
equations of ergodic type with smooth coefficients
to equations with
 measurable drift and potential.
In addition, we 
establish the strong duality for the equivalent infinite dimensional linear
programming formulation of these ergodic control problems.
We also apply these results to the study of the infinite horizon
risk-sensitive control problem for diffusions, and
establish existence of optimal Markov controls,
verification of optimality results,
and the continuity of the controlled principal eigenvalue
with respect to stationary Markov controls.
\end{abstract}
\end{abstracts}


\end{frontmatter}

\tableofcontents

\section{Introduction}\label{S1}
In this paper we study the eigenvalue problem on $\Rd$ for
non-degenerate, second order elliptic operators
$\Lg^f$ of the form
\begin{equation}\label{E-Lgf}
\Lg^f\varphi \;=\; \sum_{i,j=1}^{d}a^{ij}
\frac{\partial^{2}\varphi}{\partial{x}_{i}\partial{x}_{j}}
+\sum_{i=1}^{d} b^{i} \frac{\partial\varphi}{\partial{x}_{i}}+ f\, \varphi\,.
\end{equation}
Here $b,f\in \Lpl^\infty(\Rd)$, $a\in\Ccl^{0,1}(\Rd)$ and $a$, $b$ satisfy
a linear growth assumption in the outward radial direction (see (A2)
in \cref{S1.1}).
In other words, $a$ and $b$ satisfy the usual assumptions for
existence and uniqueness of a strong solution of the It\^o equation
\begin{equation}\label{E1.2}
X_{t} \;=\;x + \int_{0}^{t} b(X_{s})\,\D{s}
+ \int_{0}^{t} \upsigma(X_{s})\,\D{W}_{s}\,, \quad\text{with}\quad
a\df\tfrac{1}{2}\,\upsigma\upsigma\transp,
\end{equation}
where $W$ is a standard Brownian motion.

We focus on certain properties of 
the principal eigenvalue of the operator $\Lg^f$ which play a
key role in infinite horizon risk-sensitive control problems.
When $D$ is a smooth bounded domain,
and $a$, $b$, $f$ are regular enough,
existence of a principal eigenvalue and corresponding eigenfunction
under a Dirichlet boundary condition can be obtained by
an application of Krein-Rutman theory (see for instance \cite{Krein-Rutman, Pinsky}).
This eigenvalue is the bottom of the spectrum of $-\Lg^f$ with Dirichlet boundary
condition. 
For non-smooth domains, a generalized notion of a
principal eigenvalue was introduced in the seminal work of
Berestycki, Nirenberg and Varadhan \cite{Berestycki-94}.
An analogous theory for non-linear
elliptic operators has been developed by Quaas and Sirakov in \cite{Quaas-08a}. 
The principal eigenvalue plays a key role in the study of
non-homogeneous elliptic operators and the maximum principle
(see \cite{Berestycki-94, Berestycki-15,
Furusho-Ogura, Quaas-08a}). 
For some other definitions of the principal (or \emph{critical}) eigenvalue we
refer the reader to the works of
Pinchover \cite{Pinchover-88} and Pinsky \cite[Chapter~4]{Pinsky}.

For unbounded domains,
principal eigenvalue problems have been recently considered
by Berestycki and Rossi in \cite{Berestycki-06, Berestycki-15}.
Not surprisingly, certain properties of the principal eigenvalue which
hold in bounded domains may not be true for unbounded ones.
For instance, when $D$ is smooth and bounded it is well known that
for the Dirichlet boundary value problem, the principal eigenvalue is simple,
and the associated principal eigenfunction is positive.
Moreover, it is the unique eigenvalue with a positive eigenfunction.
But if $D$ is unbounded and smooth, then there exists a constant
$\lamstr=\lamstr(f)$ such that any
$\lambda\in[\lamstr, \infty)$ is an eigenvalue of $\Lg^f$
with a positive eigenfunction \cite[Theorem~1.4]{Berestycki-15}
(see also \cite{Furusho-Ogura} and \cite[Theorem~2.6]{Kaise-06}).
The lowest such value $\lamstr$ serves as a definition of
the principal eigenvalue when $D$ is not bounded.
The principal eigenvalue is known to be strictly monotone as a
function of the bounded domain $D$ (the latter ordered with respect
to set inclusion), and also strictly monotone in the coefficient $f$
when the domain is bounded (see \cite{Berestycki-15} and \cref{L2.1} below).
These properties fail to hold in
unbounded domains as remarked by Berestycki and Rossi 
\cite[Remark~2.4]{Berestycki-15}.
Strict monotonicity of
$f\mapsto\lamstr(f)$ and its implications
are a central theme in our study.
We adopt a probabilistic approach in our investigation.
One can view $\lamstr(f)$ as a risk-sensitive average of $f$ over the diffusion
in \cref{E1.2}.
More precisely, since \cref{E1.2} has a unique solution which exists for all
$t\in[0,\infty)$, then we can define
\begin{equation}\label{E-sE}
\sE_x(f)\;\df\; \limsup_{T\to\infty}\, \frac{1}{T}\;
\log\,\Exp_x\Bigl[\E^{\int_0^T f(X_s)\, \D{s}}\Bigr]\,,\quad x\in\Rd\,,
\end{equation}
with `$\log$' denoting the natural logarithm.
As shown in the proof of Lemma~2.3 in \cite{ari-anup} we have
$\lamstr(f)\le \sE_x(f)$, and equality is indeed the case in
many important situations, although strictly speaking it is only a heuristic.
This heuristic is based on the fact that for a bounded $f$,
the operator
$\Lg^f$ is the infinitesimal generator of a strongly
continuous, positive semigroup with potential $f$, see for instance
\cite[Chapter~IV]{Engel-Nagel}.
If $f$ is a bounded continuous function, and if the occupation measures of
$\{X_t\}$ obey a large deviation principle, then one can express $\sE_x(f)$ in
terms of the large deviation rate function.
This is known as the variational representation for the eigenvalue.
See for instance the article by Donsker and Varadhan \cite{Donsker-75} where
this representation is obtained for compact domains.
But large deviation
principles for $\{X_t\}$ are generally available only under strong hypotheses
on the process (see \cite{Donsker-76}).
In this paper we rely on the stochastic representation of the principal
eigenfunction which can be established under very mild assumptions.
This approach has been recently used by Arapostathis and Biswas in \cite{ari-anup}
to study the multiplicative
Poisson equation when $f$ is \textit{near-monotone}
(which includes the case of inf-compact $f$).
By an \emph{eigenpair} of $\Lg^f$ we mean a pair $(\Psi,\lambda)$, with
$\Psi$ a positive function in
 $\Sobl^{2,p}(\Rd)$, for all $p\in[1,\infty)$, and $\lambda\in\RR$, that satisfies 
\begin{equation}\label{E1.4}
\Lg^f \Psi\;=\; a^{ij} \partial_{ij}\Psi
+ b^{i} \partial_{i}\Psi+ f\, \Psi\;=\;\lambda\Psi\,.
\end{equation}
We refer to $\lambda$ as the \emph{eigenvalue}, and to $\Psi$ as the
\emph{eigenfunction}.
In \cref{E1.4}, and elsewhere in this paper, we adopt the notation
$\partial_{i}\df\tfrac{\partial~}{\partial{x}_{i}}$ and
$\partial_{ij}\df\tfrac{\partial^{2}~}{\partial{x}_{i}\partial{x}_{j}}$
for $i,j\in\NN$, and use the standard summation rule that
repeated subscripts and superscripts are summed from $1$ through $d$.

As mentioned earlier, such a pair $(\Psi,\lambda)$ exists only if
$\lambda\ge\lamstr(f)$ (see \cref{C2.1}).
Given an eigenpair $(\Psi,\lambda)$, the associated
\emph{twisted diffusion} $Y$ (a terminology used in
\cite{Kontoyiannis-02}) is an It\^{o} process as in \cref{E1.2}, but with
the drift $b$ replaced by $b+2a\grad(\log\Psi)$.
It is not generally the case that the twisted process has a strong solution which
exists for all time.
If $\lambda>\lamstr(f)$ the twisted diffusion is always transient (see \cref{L2.6}).
When $\lambda=\lamstr(f)$, the eigenfunction is denoted as $\Psi^*$ and is
called the \emph{ground state} \cite{Pinsky, Wu-94}. The corresponding twisted
diffusion, denoted by $Y^*$, is referred to as the \emph{ground-state diffusion}.

Let $\Cc_{\mathrm{o}}^+(\Rd)$ ($\Cc_{\mathrm{c}}^+(\Rd)$) denote
the class of non-zero, nonnegative real valued continuous functions on $\Rd$
which vanish at infinity (have compact support).
We say that $\lamstr(f)$ is \emph{strictly monotone at $f$} if
there exists $h\in\Cc_{\mathrm{o}}^+(\Rd)$ satisfying $\lamstr(f-h)<\lamstr(f)$.
We also say that $\lamstr(f)$ is \emph{strictly monotone at $f$ on the right} if
$\lamstr(f+h)>\lamstr(f)$ for all $h\in\Cc_{\mathrm{c}}^+(\Rd)$.
In \cref{T2.1} we show that
strict monotonicity at $f$ implies strict monotonicity at $f$ on the right.
Our main results provide sharp characterizations of the ground state $\Psi^*$ and
the ground state process $Y^*$ in terms of these monotonicity properties.
Assume that $f\colon\Rd\to\RR$ is a locally bounded, Borel measurable function,
satisfying $\essinf_\Rd f>-\infty$, and that $\lamstr(f)$ is finite.
We show that strict monotonicity of $\lamstr(f)$ at $f$ on the right implies
the simplicity of $\lamstr(f)$, i.e., the uniqueness of the ground state $\Psi^*$,
and that this is also a necessary and sufficient
condition for the ground state process to be recurrent (see \cref{L2.7,T2.3}).
Another important result is that the ground state diffusion is exponentially ergodic
(see \cref{D2.2}) if and only if $\lamstr(f)$ is strictly monotone at $f$.
These results are summarized in \cref{T2.1} in \cref{S2}. 
Other results in \cref{S2} provide a characterization of the eigenvalue in terms of the
long time behavior of the twisted process and stochastic representations
of the ground state (see \cref{L2.2,L2.3,L2.7}, and \cref{T2.4}).

In \cite{Pinsky}, Pinsky uses the existence of a Green's measure to define
the \emph{critical} eigenvalue of a non-degenerate elliptic  operator.
This critical eigenvalue coincides with the principal eigenvalue when
the boundary of the domain and the coefficients of $\Lg^f$ are smooth enough.
He shows that for any bounded domain, and provided that
the coefficients are in $\Cc^{1,\alpha}(\Rd)$, $\alpha>0$,
and bounded,
there exists a critical value $\lambda_c$
such that for any $\lambda>\lambda_c$ we can find a Green's measure corresponding
to the  operator $\Lg^{(f-\lambda)}$ \cite[Theorem~4.7.1]{Pinsky}.
The result in \cref{T2.3} in \cref{S2} extends this to $\Rd$ without assuming much
regularity on the coefficients.

Continuous dependence of $\lamstr$ on the coefficients
of $\Lg$ has also been a topic of interest.
It is not hard to see that $f\mapsto\lamstr(f)$ is lower-semicontinuous
in the $\Lpl^1(\Rd)$ topology for $f$.
Continuity of this map is also established in 
\cite[Proposition~9.2]{Berestycki-15} with respect to the
$L^\infty(\Rd)$ norm of $f$.
In \cref{T2.4,R4.1} we study the continuity of $\lamstr(f)$
for a class of functions $f$ under the $\Lpl^1(\Rd)$ topology. We also obtain
a pinned multiplicative ergodic theorem which is of independent interest,
and show that $\sE_x(f)=\lamstr(f)$ for a large class of problems.

We next discuss the connection of this problem with
a stochastic ergodic control problem.
Defining $\Breve\psi\df\log\Psi^*$ we obtain from \cref{E1.4} that
\begin{equation}\label{E1.5}
a^{ij} \partial_{ij}\Breve\psi + b^{i} \partial_{i}\Breve\psi
- \langle \grad\Breve\psi, a\grad\Breve\psi\rangle
\;=\; a^{ij} \partial_{ij}\Breve\psi + b^{i} \partial_{i}\Breve\psi
+\min_{u\in\Rd}\bigl[2\langle  a\,u, \grad\Breve\psi\rangle
+ \langle u, au\rangle\bigr]
\;=\;f-\lamstr(f)\,.
\end{equation}
It is easy to see that \cref{E1.5} is related to an ergodic
control problem with controlled drift $b+2a u$ and running cost
$\langle u, au\rangle-f(x)$.
The parameter $\lamstr(f)$
can be thought of as the optimal ergodic value;
see Ichihara \cite{Ichihara-13b}.
Note then that the twisted process defined above corresponds
to the optimally controlled diffusion.
We refer to Ichihara \cite{Ichihara-11,Ichihara-13b}
and Kaise and Sheu \cite{Kaise-06} for some important results in this direction.
For a potential $f$ that vanishes at infinity, Ichihara
\cite{Ichihara-13b, Ichihara-15} considers the ergodic
control problem in \cref{E1.5}, with a more general Hamiltonian and under
scaling of the potential.
When $f$ is nonnegative,
it is shown that the value of the ergodic problem with potential $\beta f$,
$\beta\in\RR$, equals the eigenvalue
$\lamstr(\beta f)$, and $\grad\psi^*$ is the
optimal control when the parameter $\beta$ exceeds a critical value
$\beta_c$,
while below that critical value a bifurcation occurs.
Analogous are the results in \cite{Barles-16} for viscous Hamilton--Jacobi
equations with $a$ the identity matrix and a Hamiltonian which is a power
of the gradient term.
Most of the above results are obtained for bounded,
and Lipschitz continuous $a$, $b$, and $f$.
In \cref{T2.5,T2.6} we extend these results
to measurable $b$ and $f$, and possibly unbounded $a$ and $b$.

Optimality for the ergodic problem is shown in \cite{Ichihara-13b, Ichihara-15}
via the study of the optimal finite horizon problem (Cauchy parabolic problem).
Inevitably, in doing so, optimality is shown in a certain class of controls.
To overcome this limitation,
we take a different approach to the ergodic control problem in \cref{E1.5}.
As well known, ergodic control problems can be cast as infinite dimensional
linear programs \cite{Bhatt-Bor-96,Stockbridge-90}.
Consider a controlled diffusion, with the control taking values in
a space $\Act$ with extended generator $\cA$, where the `action' $u\in\Act$ enters
implicitly as a parameter in $\cA$.
Let $\sR\colon\Act\to\RR$ denote the running cost.
The primal problem then can be written
\begin{equation*}
\alpha_* \;=\;\biggl\{\inf\; \int_{\Rd\times\Act}\sR(x,u)\,\uppi(\D{x},\D{u})\;\colon
\;\cA^*\uppi=0\,,
\ \ \uppi\in\cP(\Rd\times\Act)\biggr\}\,.
\end{equation*}
Here $\cP(\Rd\times\Act)$ denotes the class of probability measures
on the Borel $\sigma$-field of $\Rd\times\Act$. Its elements
are called \emph{ergodic occupation measures} (see \cite{Bhatt-Bor-96}).
The dual problem takes the form
\begin{equation*}
\alpha\;=\; \sup\;\Bigl\{c\in\RR\;\colon\; \inf_{u\in\Act}\,
\bigl[\cA g(x,u) + \sR(x,u)\bigr]\ge c\,,\ \ g\in\cD(\cA)\Bigr\}\,,
\end{equation*}
where $\cD(\cA)$ denotes the domain of $\cA$.
In other words the dual problem is a maximization over subsolutions of the
Hamilton--Jacobi--Bellman (HJB) equation.
For non-degenerate diffusions with a compact action space $\Act$,
under the hypothesis that $\sR$
is near-monotone, or under uniform ergodicity conditions, it is well known
that we have strong duality, i.e., $\alpha_*=\alpha$.
To the best of our knowledge, this has not been established for
problems with non-compact action spaces.
In \cref{T2.7} we establish strong duality for the ergodic problem in
\cref{E1.5}.
In this result, the coefficients $b$ and $f$ are bounded and measurable,
and $a$ is bounded, Lipschitz, and uniformly elliptic.
Moreover, we establish the unicity of the optimal ergodic occupation measure,
and as a result of this, the uniqueness of the optimal stationary Markov control.
The methodology is general enough that can be applied to various classes
of ergodic control problems that are characterized by viscous HJB equations.

The results in \cite{Ichihara-11,Kaise-06}
are obtained for smooth coefficients ($\Cc^{2, \alpha}$),
and under an assumption of exponential ergodicity
(see \cref{P3.1A} below).
We provide a sufficient condition in (H2) under which strict monotonicity
of the principal eigenvalue holds.
It is also shown that the exponential ergodicity condition of
\cite{Ichihara-11, Kaise-06} actually implies (H2);
thus (H2) is weaker.
Moreover, \cref{P3.1A} cannot hold for bounded
coefficients $a$ and $b$.
See \cref{R3.2} for details.
In \cref{T3.3} we cite a sufficient condition under which strict monotonicity of
$\lamstr(f)$ holds even when $a$ and $b$ are bounded.
Let us also remark that the method of proof \cite{Ichihara-11, Kaise-06}
utilizes the smoothness of the coefficients $a$, $b$ and $f$.
This is because a gradient estimate (Bernstein method) is required,
 which is not available under weaker regularity.
But this amount of regularity might not be available in many situations,
for instance in models with a measurable drift which are often encountered
in stochastic control problems.
Let us also mention the unpublished work of Kaise and Sheu in
\cite{Kaise-04} that contains some results similar to ours, in particular,
similar to the results in \cref{S3} and the pinned multiplicative ergodic theorems.
These results are also obtained under sufficient smoothness of
the coefficients $a$, $b$, and $f$.

In \cref{S4} we apply the above mentioned results to study the infinite
horizon risk-sensitive control problem.
We refer the reader to \cite{ari-anup} where the importance of these control problems
is discussed.
Unfortunately, the development of the infinite horizon risk-sensitive control problem
for controlled diffusions
has not been completely satisfactory, and the  same applies to
controlled Markov chain models.
Most of the available results have been obtained under restrictive settings,
and a full characterization such as uniqueness of the solution to the
risk-sensitive HJB equation,
and verification of optimality results is lacking.
Let us give a quick overview of the existing literature on risk-sensitive control
in the context of
controlled diffusions which is relevant to our problem.
Risk-sensitive control for models with a constant diffusion matrix
and \emph{asymptotically flat} drift is studied by Fleming and McEneaney
in \cite{Fleming-95}.
Another particular setting is considered by Nagai \cite{Nagai-96},
where the action space is the whole Euclidean space, and the running cost
has a specific structure.
Menaldi and Robin have considered models with periodic data \cite{Menaldi-05}.
Under the assumption of a near-monotone cost, the infinite
horizon risk-sensitive control problem
is studied in \cite{ari-anup,biswas-11a,Biswas-10}, whereas Biswas in \cite{biswas-11}
has considered this problem under the assumption of exponential ergodicity.
Differential games with risk-sensitive type costs have been studied by
Basu and Ghosh \cite{Basu-Ghosh}, Biswas and Saha \cite{BS-arxiv},
and Ghosh et. al. \cite{Ghosh-16}.
All the above studies, have obtained
existence of a pair $(V,\lamstr)$ that satisfies the risk-sensitive HJB equation,
with $\lamstr$ the optimal risk-sensitive
value, and show that any minimizing selector of the HJB is an optimal control.
The works in \cite{Nagai-96, Menaldi-05} address the existence
and uniqueness of a solution to the HJB equation, in their particular set up, but do not 
contain any verification of optimality results. 
Two main results that are missing
from the existing literature, with the exception of \cite{ari-anup},
are (a) uniqueness of the solution to the HJB equation,
and (b) verification for optimal control.
 
Following the ergodic control paradigm,
we can identify two classes of models:
(i) models with a near-monotone running cost and finite optimal value,
and no other hypotheses on the dynamics,
and (ii) models that enjoy a uniform exponential ergodicity.
Near-monotone running cost models are studied in
\cite{Nagai-96, ari-anup, biswas-11a, Biswas-10}; however, only \cite{ari-anup}
obtains a full characterization without
imposing a blanket ergodicity hypothesis.
Studies for models in class (ii) can be found in \cite{Fleming-95, Basu-Ghosh,
biswas-11, Ghosh-16}.

In this paper we study models in class (ii).
The results developed in \cref{S2,S3} enable us to
obtain a full characterization of the risk-sensitive control problem
in \cref{S4}.
The main hypotheses are \cref{A4.1,A4.2}.
Another  interesting result that we establish in \cref{S4} is the continuity of the
controlled principal eigenvalue with respect to (relaxed) stationary Markov
controls (see \cref{T4.3}).
This facilitates establishing the 
existence of an optimal stationary Markov control for risk-sensitive control
problems under risk-sensitive type constraints.
Let us also remark that this existence result is far from being obvious, since
the controlled risk-sensitive value is lower-semicontinuous with respect 
to Markov controls and 
the equality $\lamstr(f)=\sE(f)$ is not true in general.
Moreover, the usual technique of Lagrange multipliers does not work in this situation,
because of the non-convex nature of the optimization criterion.

To summarize the main contributions of the paper, we have
established several characterizations of the property of
strict monotonicity of the principal eigenvalue, and
extended several results in the literature on viscous HJB equations
with potentials $f$ vanishing at infinity, and smooth data, to measurable potential and
drift (\cref{T2.1,T2.2,T2.4,T2.5,T2.6,T2.7,T2.8,T3.2}).
We have also studied a general class of risk-sensitive control problems under a
uniform ergodicity hypothesis, and established the uniqueness of a solution
to the HJB equation and verification of optimality results
(\cref{T4.1,T4.2}).
Equally interesting are the continuity results of the controlled principal eigenvalue
with respect to stationary Markov controls (\cref{T4.3,T4.5}).

The paper is organized as follows.
\Cref{S1.1} states the assumptions on the coefficients of the operator $\Lg$,
and \cref{S1.2} summarizes the notation used in the paper.
The first three subsections of
\Cref{S2} contain the main results on the principal eigenvalue under
minimal assumptions,
while \cref{S2.4} is devoted to operators
with potential $f$ which vanishes at infinity.
\Cref{S3} improves on the results of \cref{S2}, under the assumption that
\cref{E1.2} is exponentially ergodic.
\Cref{S4} is dedicated to the infinite horizon, risk-sensitive
optimal control problem.

\subsection{Assumptions on the model}\label{S1.1}
The following assumptions on the coefficients of $\Lg$
are in effect throughout the paper
unless otherwise mentioned.
\begin{itemize}
\item[(A1)]
\emph{Local Lipschitz continuity:\/}
The function
$\upsigma\;=\;\bigl[\upsigma^{ij}\bigr]\,\colon\,\RR^{d}\to\RR^{d\times d}$
is locally Lipschitz in $x$ with a Lipschitz constant $C_{R}>0$
depending on $R>0$.
In other words, with $\norm{\upsigma}\df\sqrt{\trace(\upsigma\upsigma\transp)}$,
we have
\begin{equation*}
\norm{\upsigma(x) - \upsigma(y)}
\;\le\;C_{R}\,\abs{x-y}\qquad\forall\,x,y\in B_R\,.
\end{equation*}
We also assume that
$b\;=\;\bigl[b^{1},\dotsc,b^{d}\bigr]\transp\,\colon\,\RR^{d}\to\RR^{d}$
is locally bounded and measurable.

\item[(A2)]
\emph{Affine growth condition:\/}
$b$ and $\upsigma$ satisfy a global growth condition of the form

\begin{equation*}
\langle b(x),x\rangle^{+} + \norm{\upsigma(x)}^{2}\;\le\;C_0
\bigl(1 + \abs{x}^{2}\bigr) \qquad \forall\, x\in\RR^{d},
\end{equation*}
for some constant $C_0>0$.

\item[(A3)]
\emph{Nondegeneracy:\/}
For each $R>0$, it holds that
\begin{equation*}
\sum_{i,j=1}^{d} a^{ij}(x)\xi_{i}\xi_{j}
\;\ge\;C^{-1}_{R} \abs{\xi}^{2} \qquad\forall\, x\in B_{R}\,,
\end{equation*}
and for all $\xi=(\xi_{1},\dotsc,\xi_{d})\transp\in\RR^{d}$,
where, as defined earlier, $a= \frac{1}{2}\upsigma \upsigma\transp$.
\end{itemize}

Let us remark that the assumptions (A1)--(A3) are not optimal,
and can be weakened in many situations. 
For instance, if $\upsigma$ is 
continuous and its weak derivative lies in $L^{2(d+1)}_{\mathrm{loc}}(\Rd)$,
then \cref{E2.1} has a unique strong solution (see \cite{Zhang-05}).
The results in this paper can be extended to this setup as well.

\subsection{Notation}\label{S1.2}
The standard Euclidean norm in $\RR^{d}$ is denoted by $\abs{\,\cdot\,}$,
and $\langle\,\cdot\,,\cdot\,\rangle$ denotes the inner product.
The set of nonnegative real numbers is denoted by $\RR_{+}$,
$\NN$ stands for the set of natural numbers, and $\Ind$ denotes
the indicator function.
Given two real numbers $a$ and $b$, the minimum (maximum) is denoted by $a\wedge b$ 
($a\vee b$), respectively.
The closure, boundary, and the complement
of a set $A\subset\Rd$ are denoted
by $\Bar{A}$, $\partial{A}$, and $A^{c}$, respectively.
We denote by $\uptau(A)$ the \emph{first exit time} of the process
$\{X_{t}\}$ from the set $A\subset\RR^{d}$, defined by
\begin{equation*}
\uptau(A) \;\df\; \inf\;\{t>0\;\colon\, X_{t}\not\in A\}\,.
\end{equation*}
The open ball of radius $r$ in $\RR^{d}$, centered at the origin,
is denoted by $B_{r}$, and we let $\uptau_{r}\df \uptau(B_{r})$,
and $\uuptau_{r}\df \uptau(B^{c}_{r})$.

The term \emph{domain} in $\RR^{d}$
refers to a nonempty, connected open subset of the Euclidean space $\RR^{d}$. 
For a domain $D\subset\RR^{d}$,
the space $\Cc^{k}(D)$ ($\Cc^{\infty}(D)$), $k\ge 0$,
refers to the class of all real-valued functions on $D$ whose partial
derivatives up to order $k$ (of any order) exist and are continuous.
Also, $\Cc^{k}_{b}(D)$ ($\Cc_b^{\infty}(D)$) is the class of functions
whose partial derivatives up to order $k$ (of any order) are continuous and bounded
in $D$, and $\Cc_{\mathrm{c}}^k(D)$ denotes the subset of $\Cc^{k}(D)$,
$0\le k\le \infty$, consisting of functions that have compact support.
In addition, $\Cc_{\mathrm{o}}(\Rd)$ denotes the class of continuous
functions on $\Rd$ that vanish at infinity. By $\Cc_{\mathrm{c}}^+(\Rd)$
and $\Cc_{\mathrm{o}}^+(\Rd)$ we denote
the subsets of $\Cc_{\mathrm{c}}(\Rd)$ and
$\Cc_{\mathrm{o}}(\Rd)$, respectively, consisting of all non-trivial
nonnegative functions.
We use the term \emph{non-trivial} to refer to a function that is not
a.e.\ equal to $0$.
The space $\Lp^{p}(D)$, $p\in[1,\infty)$, stands for the Banach space
of (equivalence classes of) measurable functions $f$ satisfying
$\int_{D} \abs{f(x)}^{p}\,\D{x}<\infty$, and $\Lp^{\infty}(D)$ is the
Banach space of functions that are essentially bounded in $D$.
The standard Sobolev space of functions on $D$ whose generalized
derivatives up to order $k$ are in $\Lp^{p}(D)$, equipped with its natural
norm, is denoted by $\Sob^{k,p}(D)$, $k\ge0$, $p\ge1$.
For a probability  measure $\mu$ in $\cP(\Rd)$ and a real-valued function
$f$ which is integrable with respect to $\mu$ we use the notation
\begin{equation*}
\langle f,\mu\rangle\;=\; \mu(f)\;\df\; \int_{\Rd} f(x)\,\mu(\D{x})\,.
\end{equation*}

In general, if $\mathcal{X}$ is a space of real-valued functions on $Q$,
$\mathcal{X}_{\mathrm{loc}}$ consists of all functions $f$ such that
$f\varphi\in\mathcal{X}$ for every $\varphi\in\Cc_{\mathrm{c}}^{\infty}(Q)$.
In this manner we obtain for example the space $\Sobl^{2,p}(Q)$.

We often use Krylov's extension of the It\^{o} formula for functions
in $\Sobl^{2, d}(\Rd)$  \cite[p.~122]{Krylov}, which we refer to as the
It\^o--Krylov formula.

\section{General results}\label{S2}

Let $(\Omega, \sF, \{\sF_t\}, \Prob)$ be a given
filtered probability space with a complete, right continuous filtration $\{\sF_t\}$.
Let $W$ be a standard Brownian motion adapted to $\{\sF_t\}$.
Consider the stochastic differential equation
\begin{equation}\label{E2.1}
X_{t} \;=\;X_{0} + \int_{0}^{t} b(X_{s})\,\D{s}
+ \int_{0}^{t} \upsigma(X_{s})\,\D{W}_{s}\,.
\end{equation}
The third term on the right hand side of \cref{E2.1} is an
It\^o stochastic integral.
We say that a process $X=\{X_{t}(\omega)\}$ is a solution of \cref{E2.1},
if it is $\sF_{t}$-adapted, continuous in $t$, defined for all
$\omega\in\Omega$ and $t\in[0,\infty)$, and satisfies \cref{E2.1} for
all $t\in[0,\infty)$ a.s.
It is well known that under (A1)--(A3), there exists a unique solution
of \cref{E2.1} \cite[Theorem~2.2.4]{book}.
We let $\Exp_x$ denote the expectation operator on the canonical space of
the process
with $X_0=x$, and $\Prob_x$ the corresponding probability measure.
Recall that $\uptau(D)$ denotes the first exit time of the process $X$ from a domain
$D$. 
The process $X$ is said be \emph{recurrent} if for any bounded domain $D$
we have $\Prob_x(\uptau(D^c)<\infty)=1$ for all $x\in \Bar{D}^c$.
Otherwise the process is called transient. A recurrent process is said to be
\emph{positive recurrent} if $\Exp_x[\uptau(D^c)]<\infty$ for all $x\in \Bar{D}^c$.
It is known that for a non-degenerate diffusion the property of
recurrence (or positive recurrence) is independent of $D$ and $x$, i.e.,
if it holds for
some domain $D$ and $x\in\Bar{D}^c$, then it also holds for every domain $D$, and
all points $x\in \Bar{D}^c$
(see \cite[Lemma~2.6.12 and Theorem~2.6.10]{book}).
We define the extended operator $\Lg\colon\Cc^{2}(\RR^{d})\mapsto\Lpl^{\infty}(\Rd)$
associated to \cref{E2.1} by
\begin{equation}\label{E-Lg}
\Lg g(x) \;\df\; a^{ij}(x)\,\partial_{ij} g(x)
+ b^{i}(x)\, \partial_{i} g(x)\,.
\end{equation}

Let $f\colon\Rd\to\RR$
be a locally bounded, Borel measurable function, which is bounded from below in
$\Rd$, i.e., $\inf_\Rd f>-\infty$.
We refer to a function $f$ with these properties as a \emph{potential},
and let $\Lg^f\df \Lg + f$.

\subsection{Risk-sensitive value and Dirichlet eigenvalues}

The following lemma summarizes some results from \cite{Berestycki-94,Berestycki-15,
Quaas-08a} on the eigenvalues of the Dirichlet problem for the operator $\Lg^f$.
For simplicity, we state it for balls $B_r$, instead of more general domains.

\begin{lemma}\label{L2.1}
For each $r\in(0,\infty)$ there exists a unique pair
$(\widehat\Psi_{r},\Hat\lambda_{r})
\in\bigl(\Sobl^{2,p}(B_{r})\cap\Cc(\Bar{B}_{r})\bigr)\times\RR$,
for any $p\in[1,\infty)$, satisfying
$\widehat\Psi_{r}>0$ on $B_{r}$, $\widehat\Psi_{r}=0$ on
$\partial B_{r}$, and $\widehat\Psi_{r}(0)=1$,
which solves
\begin{equation}\label{EL2.1A}
\Lg \widehat\Psi_{r}(x) + f(x)\,\widehat\Psi_{r}(x)
\;=\; \Hat\lambda_{r}\,\widehat\Psi_{r}(x)
\qquad\text{a.e.\ }x\in B_{r}\,,
\end{equation}
with $\Lg$ as defined in \cref{E-Lg}.
Moreover, $\Hat\lambda_{r}$ has the following properties:
\begin{enumerate}[(a)]
\item
The map $r\mapsto\Hat\lambda_{r}$ is continuous and strictly increasing.

\item
In its dependence on the function $f$,
$\Hat\lambda_r$ is nondecreasing, convex, and Lipschitz continuous
\textup{(}with respect to the $\Lp^{\infty}$ norm\/\textup{)},
 with Lipschitz constant $1$.
In addition, if $f\lneqq f'$, then $\Hat\lambda_r(f)<\Hat\lambda_r(f')$.
\end{enumerate}
\end{lemma}

\begin{proof}
Existence and uniqueness of the solution follow
by \cite[Theorem~1.1]{Quaas-08a} (see also \cite{Berestycki-94}).
Part (a) follows by \cite[Theorem~1.10]{Berestycki-15},
and (iii)--(iv) of \cite[Proposition~2.3]{Berestycki-15},
while part (b) follows by \cite[Proposition~2.1]{Berestycki-94}.
\qed\end{proof}

We refer to $(\widehat\Psi_{r},\Hat\lambda_{r})$ as the \emph{eigensolution}
of the Dirichlet problem, or the
\emph{Dirichlet eigensolution} of $\Lg^f$ on $B_r$.
Correspondingly, $\Hat\lambda_{r}$ and $\widehat\Psi_{r}$ are referred to
as the \emph{Dirichlet eigenvalue} and \emph{Dirichlet eigenfunction},
respectively.

\cref{L2.1}\,(a) motivates the following definition.

\begin{definition}\label{D2.1}
Let $f$ be a potential.
The principal eigenvalue $\lamstr(f)$ on $\Rd$ of the operator $\Lg^f$ given
in \cref{E-Lgf} 
is defined as
$\lamstr(f)\df\lim_{r\to\infty}\,\Hat\lambda_r(f)$.
\end{definition}

For a potential $f$ we also define
\begin{equation}\label{E2.3}
\sE_x(f)\;\df\;\limsup_{T\to\infty}\, \frac{1}{T}\,
\log\Exp_x\Bigl[\E^{\int_0^T f(X_s)\, \D{s}}\Bigr]\,,
\quad \text{and}\quad \sE(f)\;\df\;\inf_{x\in\Rd}\;\sE_x(f).
\end{equation}
We refer to $\sE(f)$ as the \textit{risk-sensitive} average of $f$.
This quantity plays a key role in our analysis.

We also compare \cref{D2.1} with the following definition of the principal
eigenvalue, commonly used in the pde literature \cite{Berestycki-15}.
\begin{equation}\label{D2.1A}
\Hat\Lambda(f)\;=\;\inf\,\bigl\{\lambda\in\RR\;
\colon\; \exists\, \varphi\in\Sobl^{2,d}(\Rd),
\, \varphi>0, \, \Lg\varphi + (f-\lambda)\varphi\le 0, \; \text{a.e. in}\; \Rd
\bigr\}\,.
\end{equation}

The following hypothesis is enforced throughout \cref{S2} without further mention,
and it is repeated only for emphasis.
\begin{itemize}
\item[\textbf{(H1)}]
$f$ is a potential, and $\lamstr(f)$ is finite.
\end{itemize}

\begin{lemma}\label{L2.2}
The following hold
\begin{enumerate}[(i)]
\item
For any $r>0$, the Dirichlet eigensolutions 
$(\widehat\Psi_{n},\Hat\lambda_{n})$ in \cref{EL2.1A} have the following stochastic
representation
\begin{equation}\label{EL2.2A}
\widehat\Psi_{n}(x)\;=\; \Exp_{x} \Bigl[\E^{\int_{0}^{\uuptau_{r}}
[f(X_{t})-\Hat\lambda_{n}]\,\D{t}}\, \widehat\Psi_{n}(X_{\uuptau_{r}})\,
\Ind_{\{\uuptau_{r}<\uptau_{n}\}}\Bigr]\qquad\forall\,
x\in B_{n}\setminus\overline{B}_r\,,
\end{equation}
for all large enough $n\in\NN$.

\item
It holds that $\lamstr(f)=\Hat{\Lambda}(f)$.

\item
Let $\Psi^*$ be any limit point of the Dirichlet eigensolutions
$(\widehat\Psi_{n},\Hat\lambda_{n})$ as $n\to\infty$, and $\sB$ be an open ball centered
at $0$ such that
$\lamstr(f-h) + \sup_{\sB^c}|h|<\lamstr(f)<\infty$ for some
bounded function $h$.
Then with $\uuptau$ denoting the first hitting time of $\sB$ we have
\begin{equation}\label{EL2.2B}
\Psi^*(x) \;=\; \Exp_{x}
\Bigl[\E^{\int_{0}^{\uuptau}
[f(X_{t})-\lamstr(f)]\,\D{t}}\,\Psi^* (X_{\uuptau})\,\Ind_{\{\uuptau<\infty\}}\Bigr]
\qquad\forall\, x\in\sB^c\,.
\end{equation}
\end{enumerate}
\end{lemma}

\begin{proof}
Part (i) follows from \cite[Lemma~2.10\,(i)]{ari-anup}.

Turning to part (ii),
suppose that $\lamstr(f)$ is finite. Then it is standard to show that
there exists
a positive $\Psi\in\Sobl^{2,d}(\Rd)$ which satisfies
\begin{equation}\label{EL2.2B0}
\Lg\Psi + f\,\Psi \,=\,\lamstr(f)\,\Psi \quad\text{a.e.~on\ }\Rd\,.
\end{equation}
See \cite{ari-anup, biswas-11a} for instance.
It is then clear that $\lamstr(f)\ge \Hat\Lambda(f)$.

To show the converse inequality, suppose that a pair
$(\varphi, \lambda)\in\Sobl^{2,d}\times\RR$,
with $\varphi>0$, satisfies
\begin{equation}\label{EL2.2C}
\Lg\varphi + (f-\lambda)\varphi\,\le\, 0\,, \qquad\text{and}\quad
\lambda\ge\Hat\Lambda(f)\,.
\end{equation}
We claim that $\lamstr(f)\le \lambda$.
If not, then  we can find a pair
$(\widehat\Psi_r, \Hat\lambda_r)$ as in by \cref{L2.1},
satisfying \cref{EL2.1A} and
$\Hat\lambda_r>\lambda$.
By the It\^o--Krylov formula \cite[p.~122]{Krylov}
we have
\begin{equation}\label{EL2.2D}
\varphi(x)\;\ge \;\Exp_{x}\Bigl[\E^{\int_{0}^{\uuptau_r}
[f(X_{t})-\lambda]\,\D{t}}\,
\varphi(X_{\uuptau_r})\,\Ind_{\{\uuptau_r<\infty\}}\Bigr]\,.
\end{equation}
Since $\varphi$ is positive, \cref{EL2.2A,EL2.2D}
imply that we can scale it by multiplying with a
constant $\kappa>0$ so that
$\kappa\varphi-\widehat{\Psi}_r$ attains
it minimum in $\Bar{B}_r$ and this minimum value is $0$.
Combining \cref{EL2.1A,EL2.2C}, we obtain
\begin{equation*}
\Lg(\kappa\varphi-\widehat{\Psi}_r) - (f-\Hat\lambda_r)^{-}
(\kappa\varphi-\widehat{\Psi}_r)
\;\le\; -(f-\Hat\lambda_r)^{+}(\kappa\varphi-\widehat{\Psi}_r)
+ (-\Hat\lambda_r + \lambda)\kappa\varphi\;\le\; 0\quad\text{in\ } B_r\,.
\end{equation*}
It then follows by the strong maximum principle \cite[Theorem~9.6]{GilTru}
that $\kappa\varphi-\widehat{\Psi}_r=0$ in $\Bar{B}_r$, which is 
not possible since $\varphi>0$ on $\Rd$.
This proves the claim.
Since $\lambda$ was arbitrary, this implies that
$\Hat{\Lambda}(f)\ge \lamstr(f)$, and thus we have equality.

It remains  to prove \cref{EL2.2B}.
We follow the same argument as in \cite[Lemma~2.10]{ari-anup}. We fix $\sB=B_r$.
Letting $n\to\infty$ in \cref{EL2.2A} and applying Fatou's lemma we obtain
\begin{equation}\label{EL2.2E}
\Psi^*(x)\;\ge \;\Exp_{x}\Bigl[\E^{\int_{0}^{\uuptau}
[f(X_{t})-\lamstr(f)]\,\D{t}}\,
\Psi^*(X_{\uuptau})\,\Ind_{\{\uuptau<\infty\}}\Bigr]\,.
\end{equation}
Thus, with $\Tilde\Psi^*$ denoting a solution of \cref{EL2.2B0},
with $f$ replaced by $f-h$ and $\lambda=\lamstr(f-h)$,
we also have
\begin{equation*}
\Tilde\Psi^*(x)\;\ge \;\Exp_{x}\Bigl[\E^{\int_{0}^{\uuptau}
[f(X_{t})-h(X_t)-\lamstr(f-h)]\,\D{t}}\,\Tilde\Psi^*(X_{\uuptau})
\,\Ind_{\{\uuptau<\infty\}}\Bigr]\,,
\end{equation*}
which implies that
\begin{equation}\label{EL2.2F}
\Exp_{x}\Bigl[\E^{\int_{0}^{\uuptau}
[f(X_{t})-h(X_t)-\lamstr(f-h)]\,\D{t}}\,\Ind_{\{\uuptau<\infty\}}\Bigr]
\;<\;\infty \qquad \forall \; x\in\sB^c\,,
\end{equation}
since $\Tilde\Psi^*>0$ in $\Rd$.
We write \cref{EL2.2A} as
\begin{equation}\label{EL2.2G}
\widehat\Psi_{n}(x)\;\le\;
\Exp_{x}\Bigl[\E^{\int_{0}^{\uuptau}
[f(X_{t})-\Hat\lambda_{n}]\,\D{t}}\,
\Psi^*(X_{\uuptau})\,\Ind_{\{\uuptau<\uptau_{n}\}}\Bigr]
\;+\; \biggl(\sup_{\sB}\,\babs{\Psi^*-\widehat\Psi_n}\biggr)\;
\Exp_{x}\Bigl[\E^{\int_{0}^{\uuptau}
[f(X_{t})-\Hat\lambda_{n}]\,\D{t}}\,\Ind_{\{\uuptau<\uptau_{n}\}}\Bigr]\,.
\end{equation}
Note that since $\Hat\lambda_{n}\nearrow\lamstr(f)$, the first
term on the right hand side of \cref{EL2.2G} is finite by \cref{EL2.2F} for
all large enough $n$.
Let
\begin{equation*}
\kappa_n\;\df\; \Bigl(\inf_{\sB}\,\widehat\Psi_{n}\Bigr)^{-1}
\sup_{\sB}\,\babs{\Psi^*-\widehat\Psi_n}\,.
\end{equation*}
The second term on the right hand side of \cref{EL2.2G} has the bound
\begin{equation*}
\biggl(\sup_{\sB}\,\babs{\Psi^*-\widehat\Psi_n}\biggr)\;
\Exp_{x}\Bigl[\E^{\int_{0}^{\uuptau}
[f(X_{t})-\Hat\lambda_{n}]\,\D{t}}\,\Ind_{\{\uuptau<\uptau_{n}\}}\Bigr]\;\le\;
\kappa_n\;
\Exp_{x}\Bigl[\E^{\int_{0}^{\uuptau}
[f(X_{t})-\Hat\lambda_{n}]\,\D{t}}\,
\widehat\Psi_{n}(X_{\uuptau})\,\Ind_{\{\uuptau<\uptau_{n}\}}\Bigr]
\;=\; \kappa_n\;\widehat\Psi_{n}(x)\,.
\end{equation*}
By the convergence of $\widehat\Psi_{n}\to\Psi^*$
as $n\to\infty$, uniformly on compact sets,
and since $\widehat\Psi_{n}$ is bounded away from $0$ in $\sB$,
uniformly in $n\in\NN$, by Harnack's inequality,
we have $\kappa_n\to0$ as $n\to\infty$.
Therefore, the second term on the right hand side of \cref{EL2.2G}
vanishes as $n\to\infty$. 
Also, since $\Hat\lambda_{n}$ is nondecreasing in $n$, and
$\Hat\lambda_{n}\nearrow\lamstr(f)$, we obtain
\begin{equation}\label{EL2.2H}
\Exp_{x}\Bigl[\E^{\int_{0}^{\uuptau} [f(X_{t})-\Hat\lambda_{n}]\,\D{t}}\,
\Psi^*(X_{\uuptau})\,\Ind_{\{\uuptau<\uptau_{n}\}}\Bigr]
\;\xrightarrow[n\to\infty]{}\;
\Exp_{x}\Bigl[\E^{\int_{0}^{\uuptau} [f(X_{t})-\lamstr(f)]\,\D{t}}\,
\Psi^*(X_{\uuptau})\,\Ind_{\{\uuptau<\infty\}}\Bigr]\,,
\end{equation}
by \cref{EL2.2F} and dominated convergence.
Thus taking limits in \cref{EL2.2G} as $n\to\infty$,
and using \cref{EL2.2H,EL2.2E}, we obtain \cref{EL2.2B}.
This completes the proof.
\qed\end{proof}

Combining \cref{L2.2}\,(ii) and
\cite[Theorem~1.4]{Berestycki-15} we have the following result.

\begin{corollary}\label{C2.1}
There exists a positive $\Psi\in\Sobl^{2,p}(\Rd)$, $p\ge 1$, satisfying
\begin{equation}\label{E-eigen}
\Lg\Psi + f\,\Psi \,=\, \lambda\Psi \quad\text{a.e.~on\ }\Rd\,,
\end{equation}
if and only if $\lambda\ge\lamstr(f)$.
\end{corollary}

As also mentioned in the introduction,
throughout the rest of the paper, by an eigenpair $(\Psi,\lambda)$ of
$\Lg^f$ we mean a positive function $\Psi\in\Sobl^{2,d}(\Rd)$
and a scalar $\lambda\in\RR$ that satisfy \cref{E-eigen}.
In addition, the eigenfunction
$\Psi$ is assumed to be normalized as $\Psi(0)=1$, unless indicated otherwise.
When $\lambda$ is the principal eigenvalue, we refer to
$(\Psi,\lambda)$ as a \emph{principal eigenpair}.
Note, that in view of the assumptions on the coefficients, any
$\Psi\in\Sobl^{2,d}(\Rd)$ which satisfies \cref{E-eigen} belongs
to $\Sobl^{2,p}(\Rd)$, for all $p\in[1,\infty)$.
Therefore, in the interest of notational economy,
we refrain from mentioning
the function space of solutions $\Psi$ of equations of the form \cref{E-eigen},
and any such solution is meant to be in $\Sobl^{2,d}(\Rd)$.
Moreover, since these are always strong solutions, we often suppress the qualifier
`a.e.', and unless a different domain is specified, such
equations or inequalities are meant to hold on $\Rd$.

\subsection{Summary of results}

A major objective in this paper is to relate the
properties of the eigenvalues $\lambda$ in \cref{E-eigen}
to the recurrence properties of the \emph{twisted process} which is
defined as follows.
For an eigenfunction $\Psi$ satisfying \cref{E-eigen}
we let $\psi\df\log\Psi$.
Then we can write \cref{E-eigen} as
\begin{equation}\label{E2.7}
\Lg\psi + \langle \grad\psi, a\grad\psi\rangle + f \,=\, \lambda\,.
\end{equation}
The twisted process corresponding to an eigenpair $(\Psi,\lambda)$
of $\Lg^f$ is defined by the SDE
\begin{equation}\label{E-twisted}
\D{Y_s} \,=\, b(Y_s)\D{s} + 2a(Y_s)\grad\psi(Y_s)\, \D{s}
+ \upsigma(Y_s)\,\D{W_s}\,.
\end{equation}
Since $\psi\in\Sobl^{2,p}(\Rd)$, $p>d$,
 it follows that $\grad\psi$ is locally bounded
(in fact it is locally H\"older continuous), and 
therefore \cref{E-twisted} has a unique strong solution up to its 
explosion time.
We let $\widetilde{\Lg}^\psi_{\phantom{u}}$ denote the extended generator of
\cref{E-twisted},
and $\widetilde{\Exp}^\psi_x$ the associated expectation operator.
The reader might have observed that the twisted process corresponds
to Doob's $h$-transformation of the operator $\Lg^{(f-\lambda)}$
with  $h=\Psi$.

With $\Psi^*$ denoting a principal eigenfunction, i.e., an eigenfunction
associated with $\lamstr(f)$, we let $\psi^*\df\log \Psi^*$, and
denote by $Y^*$ the corresponding twisted process.
A twisted process corresponding to a principal eigenpair is called
a \emph{ground state process}, and the eigenfunction $\Psi^*$ is called
a \emph{ground state}.

Recall that $\Cc_{\mathrm{o}}^+(\Rd)$ denotes the collection of all non-trivial,
nonnegative, continuous
functions which vanish at infinity.
We consider the following two properties of $\lamstr(f)$.

\begin{itemize}
\item[\textbf{(P1)}]
\underline{Strict monotonicity at $f$}.\ \ 
For some $h\in\Cc_{\mathrm{o}}^+(\Rd)$ we have $\lamstr(f-h)<\lamstr(f)$.

\item[\textbf{(P2)}]
\underline{Strict monotonicity at $f$ on the right}.\ \ 
For all $h\in\Cc_{\mathrm{o}}^+(\Rd)$ we have $\lamstr(f)<\lamstr(f+h)$.
\end{itemize}

It follows by the convexity of $f\mapsto\lamstr(f)$ that (P1) implies (P2).

Later, in \cref{S3}, we provide sufficient conditions under which (P1) holds.
Also, the finiteness of $\lamstr(f)$ and $\lamstr(f-h)$ is implicit in (P1).
Indeed, since for every positive $\varphi\in\Sobl^{2,d}(\Rd)$,
and $\lambda\in\RR$ we have
\begin{equation*}
\Lg\varphi + (f-\lambda-\norm{h}_\infty)\varphi \;\le\;\Lg\varphi
+ (f-h-\lambda)\varphi
\;\le\; \Lg\varphi + (f-\lambda)\varphi\,,
\end{equation*}
it follows that $\lamstr(f-h)$  and $\lamstr(f)$ are either both
finite, or both equal to $\pm\infty$.
It is also clear that $\lamstr(f-h)\le \lamstr(f)$ always hold.
As shown in \cref{T2.2}, (P1) implies
that  $\lamstr(f-h)<\lamstr(f)$ for all $h\in\Cc_{\mathrm{o}}^+(\Rd)$.

We introduce the following
definition of exponential ergodicity which we often use.

\begin{definition}[exponential ergodicity]\label{D2.2}
The process $X$ governed by
\cref{E1.2} is said to be exponentially ergodic if for some compact set
$\sB$ and $\delta>0$ we have
$\Exp_x\bigl[\E^{\delta\, \uptau(\sB^c)}\bigr]< \infty$,
 for all $x\in\sB^c$.
\end{definition}

The main results of this section center around the following theorem.

\begin{theorem}\label{T2.1}
Under \textup{(H1)}, the following hold:
\begin{enumerate}[(a)]
\item
A ground state process is recurrent if and only if
$\lamstr(f)$ is strictly monotone at $f$ on the right,
in which case the principal eigenvalue $\lamstr(f)$ is also simple,
and the ground state $\Psi^*$ satisfies
\begin{equation}\label{ET2.1A}
\Psi^*(x)\;=\;\Exp_x\Bigl[\E^{\int_0^{\uuptau_r}[f(X_s)-\lamstr(f)]\, \D{s}}\,
\Psi^*(X_{\uuptau})\,\Ind_{\{\uuptau_r<\infty\}}\Bigr]\qquad
\forall\, x\in \Bar{B}_r^c\,,\quad\forall\,r>0\,.
\end{equation}

\item
The ground state process is exponentially ergodic if and only if
$\lamstr(f)$ is strictly monotone at $f$.

\item
If $\lambda>\lamstr(f)$, the twisted process \cref{E-twisted}
corresponding to any solution $\psi$ of \cref{E2.7} is transient.
\end{enumerate}
\end{theorem}

\begin{proof}
Part (a) follows by \cref{L2.7,T2.3,C2.3}.
Part (b) is the statement of \cref{T2.2}, while part (c) is shown
in \cref{L2.6}.
\qed\end{proof}

\cref{T2.1} should be compared with the results in
\cite[Theorem~2.2]{Ichihara-11} and
\cite[Theorem~3.2 and 3.7]{Kaise-06}.
The results in \cite{Ichihara-11, Kaise-06} are obtained under a stronger hypothesis
(same as \cref{P3.1A} below) and for sufficiently regular coefficients. 
For a similar result in a bounded domain we refer the reader
to \cite[Theorem~4.2.4]{Pinsky}, where results are
obtained for a certain class of operators with regular coefficients.

We remark that (P1) does not imply that the underlying process in
\cref{E2.1} is recurrent.
Indeed consider a one-dimensional diffusion with
$b(x)=\frac{3}{2}x$ and $\upsigma=1$, and let $f(x)=x^2$.
Then \cref{E-eigen} holds with $\Psi(x) = \E^{-x^2}$ and
$\lambda=-1$.
But $b(x)+2a \grad{\psi}=-\frac{1}{2}x$, so the twisted
process is exponentially ergodic, while the original diffusion
is transient.


The proof of \cref{T2.1} is divided in several lemmas which also
contain results of independent interest.
These occupy the next section.

\subsection{Proof of \texorpdfstring{\cref{T2.1}}{T2.1} and other results}

In the sequel, we often use the following finite time representation.
This also appears in 
\cite[Lemma~2.4]{ari-anup} but in a slightly different form.
Let $\uptau_\infty\df\lim_{n\to\infty}\,\uptau_n$ where $\uptau_n$
denotes the exit time from the ball $B_n$.
Recall that if $(\Psi, \lambda)$ is an eigenpair of $\Lg^f$, and
$\psi=\log\Psi$, then $\widetilde\Exp_x^\psi$ denotes the expectation operator
associated with the twisted process $Y$ in \cref{E-twisted}.

\begin{lemma}\label{L2.3}
If $(\Psi, \lambda)$ is an eigenpair of $\Lg^f$, then
\begin{equation}\label{EL2.3A}
\Psi(x)\,\widetilde\Exp_x^\psi\bigl[g(Y_{T}) \,
\Psi^{-1}(Y_{T})\,\Ind_{\{T<\uptau_\infty\}}\bigr] \;=\;
\Exp_x\Bigl[\E^{\int_{0}^{T}[f(X_{t})-\lambda]\,\D{t}}\,g(X_{T})\Bigr]
\qquad\forall\,T>0\,,\quad\forall\,x\in\Rd\,,
\end{equation}
and for any function $g\in\Cc_{\mathrm{c}}(\Rd)$,
where $Y$ is the corresponding twisted process defined by \cref{E-twisted}.
\end{lemma}

\begin{proof}
The equation in \cref{EL2.3A} can be obtained by applying
the Cameron--Martin--Girsanov theorem \cite[p.~225]{LiSh-I}.
Since $\psi$ and $f$ are not bounded, we need to localize the martingale.
We use the first exit times $\uptau_n$ from $B_n$
as localization times.
It is well-known that assumption (A2) implies that
$\uptau_n\to \infty$ as $n\to\infty$ $\Prob_x$-a.s.
Applying the It\^{o}--Krylov formula  and using \cref{E2.7},
we obtain
\begin{align}\label{EL2.3B}
\psi(X_{T\wedge\uptau_n}) -\psi(x) &\;=\;
\int_0^{T\wedge\uptau_n}\Lg\psi(X_{s})\,\D{s}
+ \int_0^{T\wedge\uptau_n} \langle \grad\psi(X_s), \upsigma(X_s)\,\D{W_s}\rangle
\nonumber\\[3pt]
&\;=\; \int_0^{T\wedge\uptau_n}
\Bigl( \lambda- f(X_s)
- \langle\grad\psi,a\grad \psi\rangle(X_s)\Bigr)\,\D{s}
+ \int_0^{T\wedge\uptau_n} \langle \grad\psi(X_s), \upsigma(X_s) \D{W_s}\rangle\,.
\end{align}
Let $g$ be any nonnegative, continuous function with compact support.
Then from  \cref{EL2.3B} we obtain
\begin{align}\label{EL2.3C}
\Exp_x\Bigl[\E^{\int_0^{T\wedge \uptau_n}
[f(X_s)-\lambda]\,\D{s}}\,g(X_{T\wedge\uptau_n})\Bigr]
&\;=\; \Exp_x\biggl[g(X_{T\wedge\uptau_n}) \exp\biggl(
-\psi(X_{T\wedge\uptau_n}) + \psi(x) \nonumber\\[3pt]
&\mspace{30mu}
+\int_0^{T\wedge\uptau_n} \langle \grad\psi(X_s), \upsigma(X_s) \D{W_s}\rangle
-\int_0^{T\wedge\uptau_n} \langle\grad\psi,a\grad \psi\rangle(X_s)\,
\D{s}\biggr)\biggr]\nonumber\\[5pt]
&\;=\; \Psi(x)\,\widetilde\Exp_x^\psi \bigl[g(Y_{T\wedge\uptau_n}) \,
\Psi^{-1}(Y_{T\wedge\uptau_n})\bigr]\,,
\end{align}
where in the last line we use Girsanov's theorem.
Given any bounded ball $\sB$,
by It\^o's formula and Fatou's lemma, we obtain from \cref{E-eigen} that
\begin{equation}\label{EL2.3D}
\Exp_{x}\,\Bigl[\E^{\int_{0}^{T}
[f(X_{t})-\lambda]\,\D{t}}\,\Ind_{\sB}(X_T)\Bigr]\;\le\;
\Bigl(\inf_{\sB}\,\Psi\Bigr)^{-1}\,
\Psi(x)\qquad\forall\,T>0\,,\quad\forall\,x\in\Rd\,.
\end{equation}
Therefore, if we write
\begin{multline*}
\Exp_x\Bigl[\E^{\int_0^{T\wedge \uptau_n}
[f(X_s)-\lambda]\,\D{s}}\,g(X_{T\wedge\uptau_n})\Bigr]
\;=\; \Exp_x\Bigl[\E^{\int_0^{\uptau_n}
[f(X_s)-\lambda]\,\D{s}}\,g(X_{\uptau_n})\,\Ind_{\{T\ge\uptau_n\}}\Bigr]\\[3pt]
+ \Exp_x\Bigl[\E^{\int_0^{T}
[f(X_s)-\lambda]\,\D{s}}\,g(X_{T})\,\Ind_{\{T<\uptau_n\}}\Bigr]\,,
\end{multline*}
we deduce that the first term on the right hand side is equal to $0$ for all
$n$ sufficiently large since $g$ is compactly supported,
while the second term converges as $n\to\infty$ to the right hand side
of \cref{EL2.3A} by \cref{EL2.3D} and dominated convergence.
In addition, since $g$ has compact support,
the term inside the expectation in the right hand side of \cref{EL2.3C} is bounded
uniformly in $n$.
Since also $\widetilde\Exp_x^\psi\bigl[g(Y_{\uptau_n}) \,
\Psi^{-1}(Y_{\uptau_n})\bigr]=0$ for all sufficiently large $n$,
letting $n\to\infty$ in \cref{EL2.3C}, we obtain
\begin{align*}
\Exp_x\Bigl[\E^{\int_0^{T} [f(X_s)-\lambda]\,\D{s}}\,g(X_{T})\Bigr]
&\;=\; \lim_{n\to\infty}\;
\Psi(x)\,\widetilde\Exp_x^\psi \Bigl[g(Y_{T}) \,
\Psi^{-1}(Y_{T})\,\Ind_{\{T<\uptau_n\}}\Bigr]\\[3pt]
&\;=\;
\Psi(x)\,\widetilde\Exp_x^\psi \Bigl[g(Y_{T}) \,
\Psi^{-1}(Y_{T})\,\Ind_{\{T<\uptau_\infty\}}\Bigr]\qquad\forall\,T>0\,.
\end{align*}
This proves \cref{EL2.3A}.
\qed\end{proof}

Recall that $\uptau_\infty\df \lim_{n\to\infty}\,\uptau_n$.
An immediate corollary to \cref{L2.3} is the following.

\begin{corollary}\label{C2.2}
With $(\Psi,\lambda)$ as in \cref{L2.3},
we have
\begin{equation*}
\Psi(x)\,\widetilde\Prob_x^\psi(T<\uptau_{\infty}) \;=\;
\Exp_x\Bigl[\E^{\int_{0}^{T}[f(X_{t})-\lambda]\,\D{t}}\,\Psi(X_{T})\Bigr]
\qquad\forall\,T>0\,,\quad\forall\,x\in\Rd\,.
\end{equation*}
\end{corollary}

\begin{proof}
Choose a sequence of cut-off functions $g_n$ that approximates unity from below.
Then \cref{EL2.3A} holds
with $g$ replaced by $g_n\Psi$.
Thus the result follows by letting $n\to\infty$ and
applying the monotone convergence theorem.
\qed\end{proof}

We are now ready to prove uniqueness of the principal eigenfunction.

\begin{lemma}\label{L2.4}
Under \textup{(P1)}
there exists a unique ground state $\Psi^*$ for $\Lg^f$,
i.e., a positive $\Psi^*\in\Sobl^{2,d}(\Rd)$, $\Psi^*(0)=1$, which solves
\begin{equation}\label{EL2.4A}
\Lg\Psi^* + f\,\Psi^*\;=\;\lamstr(f)\,\Psi^*\,.
\end{equation}
\end{lemma}

\begin{proof}
Let $\Psi^*$ be a solution of \cref{EL2.4A} obtained as a limit of
$\widehat{\Psi}_r$ (see \cref{L2.2}).
Thus by \cref{L2.2}\,(iii) we can find an open ball $\sB$ such that
\begin{equation*}
\Psi^*(x)\;=\;\Exp_x\Bigl[\E^{\int_0^{\uuptau} [f(X_s)-\lamstr(f)]\, \D{s}}\,
\Psi^*(X_{\uuptau})\,\Ind_{\{\uuptau<\infty\}}\Bigr]\,,
\quad x\in\sB^c\,,
\end{equation*}
with $\uuptau=\uptau(\sB^c)$.
Suppose that $\Tilde\Psi$
is another principal eigenfunction of $\Lg^f$.
By the It\^o--Krylov formula and Fatou's lemma, and since $\Psi^*$
is positive on $\Bar\sB$,
we obtain
\begin{align}\label{EL2.4C}
\Tilde\Psi(x)&\;\ge\;\Exp_x\Bigl[\E^{\int_0^{\uuptau} [f(X_s)-\lamstr(f)]\, \D{s}}\,
\Tilde\Psi(X_{\uuptau})\,\Ind_{\{\uuptau<\infty\}}\Bigr]\nonumber\\
&\;\ge\; \Bigl(\min_{\Bar\sB}\,\tfrac{\Tilde\Psi}{\Psi^*}\Bigr)\,\Psi^*(x)
\qquad \forall\,x\in\sB^c\,.
\end{align}
It is clear by \cref{EL2.4C} that if $\Tilde\Psi>\Psi^*$ on $\Bar\sB$,
then $\Tilde\Psi-\Psi^*>0$ on $\Rd$.
Therefore, we can scale $\Psi^*$ by multiplying it with
$\min_{\Bar\sB}\tfrac{\Tilde\Psi}{\Psi^*}$ so that $\Tilde\Psi$ touches $\Psi^*$
from above in $\Bar\sB$ at the points $\argmin_{\Bar\sB}\tfrac{\Tilde\Psi}{\Psi^*}$.
Denoting this scaled $\Psi^*$ also as $\Psi^*$, it follows 
from \cref{EL2.4C} that $\Tilde\Psi-\Psi^*$ is nonnegative in $\Rd$, and its
minimum is $0$ and attained in $\Bar\sB$.
On the other hand, we have
\begin{equation*}
\Lg(\Tilde\Psi-\Psi^*) -\bigl(f-\lamstr(f)\bigr)^{-} (\Tilde\Psi-\Psi^*)
\;=\; -\bigl(f-\lamstr(f)\bigr)^{+} (\Tilde\Psi-\Psi^*)\;\le\; 0\,.
\end{equation*}
Thus  $\Tilde\Psi-\Psi^*=0$
by the strong maximum principle \cite[Theorem~9.6]{GilTru}, and this proves the result.
\qed\end{proof}

We next show that (P1) implies the exponential ergodicity of $Y^*$.

\begin{lemma}\label{L2.5}
Assume \textup{(P1)}.
Let $\Psi^*$ be the ground state of $\Lg^f$,
and $\psi^*=\log\Psi^*$. Then the ground state process $Y^*$ governed by
\begin{equation}\label{EL2.5A}
\D{Y^*_s} \;=\; b(Y^*_s)\,\D{s} + 2a(Y^*_s)\grad\psi^*(Y^*_s)\, \D{s}
+ \upsigma(Y^*_s)\,\D{W_s}\,,
\end{equation}
is exponentially ergodic. In particular, $Y^*$ is positive recurrent.
\end{lemma}
\begin{proof}

We first show that the finite time representation of $\Psi^*$ holds. 
Let $\Tilde\lambda\str\df\lamstr(f-h)$, 
and $\sB$ be a ball as in \cref{L2.2}\,(iii). 
Recall that $\Hat\lambda_n\to\lamstr(f)$ as $n\to\infty$,
and therefore, we have 
$\Hat\lambda_n>\Tilde\lambda\str + \sup_{\sB^c}\abs{h}$
for all sufficiently large $n$.
Consider the following equations
\begin{align*}
\Lg\widehat\Psi_n + f\widehat\Psi_n &\;=\;\Hat\lambda_r\widehat\Psi_n\,,\quad
\widehat\Psi_n>0,\ \ \widehat\Psi_n(0)=1,\ \ \widehat\Psi_n=0 
\text{\ \ on\ } \partial B_n\,,
\\[5pt]
\Lg\Tilde\Psi^* + (f-h)\Tilde\Psi^* &\;=\;\Tilde\lambda\str\Tilde\Psi^*,
\quad \Tilde\Psi(0)=1\,.
\end{align*}
Choose $n$ large enough so that $\sB\subset B_n$.
We can scale $\Tilde\Psi^*$, by multiplying it with a positive constant,
so that $\Tilde\Psi^*$ touches $\widehat\Psi_n$ from above.
Next we show that it can only touch $\widehat\Psi_n$ in $\sB$.
Note that in $B_n\setminus\sB$ we have
\begin{equation*}
\Lg(\Tilde\Psi^*-\widehat\Psi_n) -(f-h-\Tilde\lambda\str)^-(\Tilde\Psi^*-\widehat\Psi_n)
\;=\;
-(f-h-\Tilde\lambda\str)^+(\Tilde\Psi^*-\widehat\Psi_n)
-(\hat\lambda_n-\Tilde\lambda\str -h)\,
\widehat\Psi_n\;\le\; 0\,.
\end{equation*}
Therefore, by the strong maximum principle, if $(\Tilde\Psi^*-\widehat\Psi_n)$
 attains its minimum in $B_n\setminus\sB$, then 
$(\Tilde\Psi^*-\widehat\Psi_n)=0$ in $B_n$, which is not possible.
Thus $\Tilde\Psi^*$ touches $\widehat\Psi_n$ in $\sB$. 
Thus, applying
Harnack's inequality we can find a constant $\kappa_1$ such that
$\kappa_1\Tilde\Psi^*\ge\widehat\Psi_n$ for all sufficiently large $n$.
On the other hand, by the It\^o--Krylov formula and Fatou's lemma we know that
\begin{equation}\label{EL2.5B}
\Exp_x\Bigl[\E^{\int_0^{T}
[f(X_s)-h(X_s)-\Tilde\lambda\str]\,\D{s}}\,\Tilde\Psi^*(X_{T})\Bigr]\;\le\;
\Tilde\Psi^*(x)\qquad \forall\,T>0\,.
\end{equation}
Applying the It\^o--Krylov formula to \cref{EL2.1A} we have
\begin{equation*}
\widehat\Psi_n(x) \;=\; \Exp_x\Bigl[\E^{\int_0^{T}
[f(X_s)-\Hat\lambda_n]\,\D{s}}\, \widehat\Psi_n(X_{T})\,
\Ind_{\{T<\uptau_n\}}\Bigr]\,,
\end{equation*}
and letting $n\to\infty$, using \cref{EL2.5B} and
the dominated convergence theorem, we obtain 
\begin{equation*}
\Psi^*(x) \;=\; \Exp_x\Bigl[\E^{\int_0^{T}
[f(X_s)-\lamstr]\,\D{s}}\,\Psi^*(X_{T})\Bigr]\,,
\end{equation*}
where $\Psi^*$ is the unique solution of \cref{EL2.4A}.
This proves the finite time representation.
Thus it follows from \cref{C2.2} that \cref{EL2.5A} is
regular, i.e.,
$\widetilde\Prob_x^{\psi^*}(\uptau_\infty<\infty)=0$.

If we define
$\Phi\df\frac{\Tilde\Psi^*}{\Psi^*}$,
a straightforward calculation shows that
\begin{equation}\label{EL2.5C}
\widetilde\Lg^{\psi^*}\Phi\;=\;
\Lg\Phi + 2\langle a\grad \psi^*,\grad\Phi\rangle \;=\;
(\lamstr(f-h)-\lamstr(f) + h)\Phi
\;\le\; C\Ind_{\sB} - \epsilon \Phi
\end{equation}
for some positive constants $C$ and $\epsilon$.
Recall $\sB$ from \cref{L2.2}\,(iii).
It is easy to see from \cref{EL2.2B} that
\begin{equation*}
\Phi(x)\;\ge\; \frac{\min_{\sB}\Tilde\Psi^*}{\max_{\sB}\Psi^*}\;
\frac{\Exp_{x}\Bigl[\E^{\int_{0}^{\uuptau}
[f(X_{t})-h(X_t)-\lamstr(f-h)]\,\D{t}}\,\Ind_{\{\uuptau<\infty\}}\Bigr]}
{\Exp_{x}\Bigl[\E^{\int_{0}^{\uuptau}
[f(X_{t})-\lamstr(f)]\,\D{t}}\,\Ind_{\{\uuptau<\infty\}}\Bigr]}
\;\ge\; \frac{\min_{\sB}\Tilde\Psi^*}{\max_{\sB}\Psi^*}\quad
\forall\,x\in\sB^c\,.
\end{equation*}
Thus $\Phi$ is uniformly bounded from below by a positive constant.
Since $Y^*$ in \cref{EL2.5A} is regular,
the Foster--Lyapunov inequality in
\cref{EL2.5C} implies that $Y^*$ is exponentially ergodic.
\qed\end{proof}

We denote the invariant measure of \cref{EL2.5A} by $\mu^*$.
The following
lemma shows that the twisted process is transient for any $\lambda>\lamstr(f)$.

\begin{lemma}\label{L2.6}
Let $\Psi$ be an eigenfunction of $\Lg^f$ for an eigenvalue $\lambda>\lamstr(f)$.
Then the corresponding twisted process $Y$ is transient.
\end{lemma}

\begin{proof}
Let $\psi=\log\Psi$.
If $\widetilde\Prob_x^\psi(\uptau_\infty<\infty)>0$, then there is nothing to prove.
So we assume the contrary.
Hence from \cref{L2.3} we have
\begin{align}
\Psi(x)\,\widetilde\Exp_x^\psi\bigl[g(Y_T)\,\Psi^{-1}(Y_{T})\bigr]
&\;=\;
\Exp_x\Bigl[\E^{\int_{0}^{T}[f(X_{t})-\lambda]\,\D{t}}\,g(X_{T})\Bigr]
\qquad\forall\,T>0\,,\label{EL2.6B}
\end{align}
for any continuous $g$ with compact support.
Let $g\in\Cc_{\mathrm{c}}^+(\Rd)$.
By the It\^o--Krylov formula and Fatou's lemma, we have
\begin{equation*}
\Exp_x\Bigl[\E^{\int_{0}^{T}[f(X_{t})-\lamstr(f)]\,\D{t}}\,g(X_{T})\Bigr]
\;\le\; \biggl(\sup_\Rd\, \frac{g}{\Psi^*}\biggr)\,
\Exp_x\Bigl[\E^{\int_{0}^{T}[f(X_{t})-\lamstr(f)]\,\D{t}}\,\Psi^*(X_{T})\Bigr]
\;\le\; \biggl(\sup_\Rd\, \frac{g}{\Psi^*}\biggr)\,\Psi^*(x)
\end{equation*}
Thus, for $\delta=\lambda-\lamstr(f)>0$, we obtain
\begin{equation}\label{EL2.6D}
\Exp_x\Bigl[\E^{\int_{0}^{T}[f(X_{t})-\lambda]\,\D{t}}\,g(X_{T})\Bigr]
\;\le\; \biggl(\sup_\Rd\, \frac{g}{\Psi^*}\biggr)\, \E^{-\delta T}\,\Psi^*(x)\,,
\quad T>0\,.
\end{equation}
Combining \cref{EL2.6B,EL2.6D}, we have
\begin{equation*}
\biggl(\sup_{\supp (g)}\Psi\biggr)^{-1}
\int_0^\infty \widetilde\Exp_x^\psi\bigl[g(Y_t)\bigr]\, \D{t}\;\le\;
\int_0^\infty \widetilde\Exp_x^\psi\Bigl[\tfrac{g(Y_t)}{\Psi(Y_t)}\Bigr]\, \D{t}
\;\le\;\frac{1}{\delta}\,\biggl(\sup_\Rd\, \frac{g}{\Psi^*}\biggr)\,\Psi^*(x)\,
\Psi^{-1}(x)\,.
\end{equation*}
Therefore, $Y$ is transient.
\qed\end{proof}

\begin{theorem}\label{T2.2}
The following are equivalent.
\begin{enumerate}[(i)]
\item
The process $Y^*$, defined in \cref{EL2.5A}, corresponding to some
principal eigenpair $\bigl(\Psi^*,\lamstr(f)\bigr)$ is exponentially ergodic.

\item 
It holds that $\lamstr(f-h)<\lamstr(f)$ for all $h\in\Cc_{\mathrm{o}}^+(\Rd)$.

\item
It holds that $\lamstr(f-h)<\lamstr(f)$ for some $h\in\Cc_{\mathrm{o}}^+(\Rd)$.
\end{enumerate}
\end{theorem}

\begin{proof}
(iii)$\,\Rightarrow\,$(i) follows from \cref{L2.5},
and (ii)$\,\Rightarrow\,$(iii) is obvious.

We show that (i)$\,\Rightarrow\,$(ii).
If $Y^*$ is exponentially ergodic, then there exists a ball
$\sB$ and $\delta>0$ such that
\begin{equation*}
\widetilde\Exp_x^{\psi^*}\bigl[e^{\delta\uuptau}\bigr]\;<\; \infty\,,
\qquad \uuptau=\uuptau(\sB^c)\,.
\end{equation*}
Mimicking the calculations in the proof of \cref{L2.3}, we obtain that 
\begin{align*}
\Exp_x\Bigl[\E^{\int_0^{T\wedge\uuptau} [f(X_s)-\lamstr(f) + \delta]\,\D{s}}\,
g(X_{T\wedge\uuptau})
\Psi^*(X_{T\wedge\uuptau})\,\Ind_{\{\uuptau<\infty\}}\Bigr]
&\;=\; \Psi^*(x)\,\widetilde\Exp_x^{\psi^*}\bigl[e^{\delta (T\wedge\uuptau)}
g(Y^*_{T\wedge\uuptau})\bigr]\qquad\forall\,T>0\,,
\end{align*}
for $g\in\Cc_{\mathrm{c}}(\Rd)$.
We apply this equation to an increasing sequence
$\{g_m\}\subset\Cc_{\mathrm{c}}(\Rd)$ which converges to $1$,
and let first $m\to\infty$, and then $T\to\infty$, using Fatou's lemma
and the exponential ergodicity of $Y^*$, to obtain
\begin{equation}\label{ET2.2A}
\Exp_x\Bigl[\E^{\int_0^{\uuptau} [f(X_s)-\lamstr(f) + \delta]\,\D{s}}
\,\Ind_{\{\uuptau<\infty\}}\Bigr]\; <\; \infty\,,
\quad x\in\sB^c.
\end{equation}
Let $h\in\Cc_{\mathrm{o}}^+(\Rd)$.
Since $h$ is bounded, it is easy to see that $\lamstr(f-h)$ is finite.
Let $\Tilde{f}\df f-h$, and
$\bigl(\Tilde\Psi^*,\lamstr(\Tilde{f})\bigr)$ be a solution of 
\begin{equation}\label{ET2.2B}
\Lg\Tilde\Psi^* + (f-h)\Tilde\Psi^*\;=\;\lamstr(\Tilde{f})\Tilde\Psi^*, 
\quad \Tilde\Psi^*>0\,,
\end{equation}
which is obtained as a limit of Dirichlet eigensolutions as in \cref{L2.2}.
If $\lamstr(\Tilde{f})=\lamstr(f)$, then in view of \cref{ET2.2A} and
the calculations in the proof of \cref{L2.2}\,(iii), we have
\begin{equation}\label{ET2.2C}
\Tilde\Psi^*(x) \;=\; \Exp_{x}
\Bigl[\E^{\int_{0}^{\uuptau}
[f(X_{t})-h(X_{t})-\lamstr(\Tilde{f})]\,\D{t}}\,\Tilde\Psi^* (X_{\uuptau})
\,\Ind_{\{\uuptau<\infty\}}\Bigr]
\qquad\forall\, x\in\sB^c\,.
\end{equation}
Applying the It\^o--Krylov formula and Fatou's lemma to \cref{EL2.4A}, we obtain
\begin{equation}\label{ET2.2D}
\Psi^*(x) \;\ge\; \Exp_x\Bigl[\E^{\int_0^{\uuptau} [f(X_s)-\lamstr(f)]\, \D{s}}\,
\Psi^*(X_{\uuptau})\,\Ind_{\{\uuptau<\infty\}}\Bigr]\,,
\quad x\in\sB^c.
\end{equation}
It follows by \cref{ET2.2C,ET2.2D}
that we can multiply $\Psi^*$ with a suitable positive constant so that
$\Psi^*-\Tilde\Psi^*$ attains a minimum of $0$ in $\sB$.
On the other hand, from \cref{EL2.4A,ET2.2B} we have
\begin{equation}\label{ET2.2E}
\Lg(\Psi^*-\Tilde\Psi^*) - \bigl(f-\lamstr(f)\bigr)^-(\Psi^*-\Tilde\Psi^*)
\;=\;- \bigl(f-\lamstr(f)\bigr)^+(\Psi^*-\Tilde\Psi^*)-h\,\Tilde\Psi^*\;\le\; 0\,.
\end{equation}
Thus by strong maximum principle we have $\Psi^*=\Tilde\Psi^*$.
This, in turn, implies that $h\,\Tilde\Psi^*=0$ by \cref{ET2.2E}.
But this is not possible. Hence we have $\lamstr(\Tilde{f})<\lamstr(f)$,
and the proof is complete.
\qed\end{proof}


We define the Green's measure $G_\lambda$, $\lambda\in\RR$, by
\begin{equation*}
G_\lambda (g) \;\df\;\Exp_0\biggl[\int_0^\infty \E^{\int_0^t[f(X_s)-\lambda]\, \D{s}}
\,g(X_t)\, \D{t}\biggr]\quad \text{for all\ } g\in\Cc_{\mathrm{c}}^+(\Rd)\,.
\end{equation*}

The density of the Green's measure with respect to the
Lebesgue measure is called the Green's function.
Existence of a Green's function (and Green's measure) is used by Pinsky
\cite[Chapter~4.3]{Pinsky} in his definition of
the generalized principal eigenvalue of $\Lg^f$.
A number $\lambda\in\RR$ is said to be \emph{subcritical} if
$G_\lambda$ possesses a density,
\emph{critical} if it is not subcritical and $\Lg^{f-\lambda}V=0$
has a positive solution $V$,
and \emph{supercritical} if it is neither
subcritical nor critical.

The lemma which follows is an extension of \cite[Theorem~4.3.4]{Pinsky}
where, under a regularity assumption on the coefficients,
it is shown that a critical eigenvalue $\lambda$ is always simple.
This result establishes several equivalences of
the notion of criticality of $\lambda$.

\begin{lemma}\label{L2.7}
The following are equivalent.
\begin{enumerate}[(i)]
\item
The twisted process $Y$ corresponding to the eigenpair $(\Psi, \lambda)$ is recurrent.

\item
$G_\lambda(g)$ is infinite for some $g\in\Cc_{\mathrm{c}}^+(\Rd)$.

\item
For some open ball $\sB$, and with $\uuptau=\uuptau(\sB)$,
we have
\begin{equation*}
\Psi(x)\;=\;\Exp_x\Bigl[\E^{\int_0^{\uuptau}[f(X_s)-\lambda]\, \D{s}}\,\Psi(X_{\uuptau})
\,\Ind_{\{\uuptau<\infty\}}\Bigr]\,,
\quad x\in \Bar\sB^c\,,
\end{equation*}
where $\Psi$ is an eigenfunction corresponding to the eigenvalue $\lambda$.
\end{enumerate}
In addition, in \textup{(}ii\textup{)}--\textup{(}iii\textup{)}
``some'' may be replaced by ``all'',
and if any one of \textup{(}i\textup{)}--\textup{(}iii\textup{)} holds,
then $\lambda$ is a simple eigenvalue.
\end{lemma}

\begin{proof}
The argument of this proof is inspired from \cite[Theorem~2.8]{ari-anup}.
By \cref{C2.1} we have $\lambda\ge\lamstr(f)$.
Assume that (i) holds for some $\lambda\ge \lamstr(f)$.
Let $(\Psi, \lambda)$  be an eigenpair of $\Lg^f$.
Then for any $g\in\Cc_{\mathrm{c}}^+(\Rd)$ we have from \cref{L2.3} that
\begin{align}
\Psi(x)\,\widetilde\Exp_x^\psi\bigl[g(Y_{T}) \,
\Psi^{-1}(Y_{T})\,\bigr] &\;=\;
\Exp_x\Bigl[\E^{\int_{0}^{T}[f(X_{t})-\lambda]\,\D{t}}\,g(X_{T})\Bigr]
\qquad\forall\,T>0\,.\label{EL2.7B}
\end{align}
On the other hand, if $Y$ is recurrent, then 
\begin{equation*}
\int_0^\infty \widetilde\Exp_x^\psi\bigl[g(Y_{t}) \,
\Psi^{-1}(Y_{t})\,\bigr] \D{t}\;=\;\infty\,.
\end{equation*}
Combining this with \cref{EL2.7B} we have $G_\lambda(g)=\infty$. Hence (ii) follows.

Next suppose that (ii) holds, i.e., $G_\lambda(g)=\infty$
for some $g\in\Cc_{\mathrm{c}}^+(\Rd)$ and $\lambda\ge\lamstr(f)$.
Applying the It\^{o}--Krylov formula to $\Lg\Psi + (f-\lambda)\Psi=0$, we have
\begin{equation}\label{EL2.7C}
\Exp_x\Bigl[\E^{\int_0^{t} [f(X_s)-\lambda]\, \D{s}}\, \Ind_{\sB} (X_t) \Bigr]
\;\le\;  \frac{1}{\min_{\sB}\Psi}\,
\Exp_x\Bigl[\E^{\int_0^{t} [f(X_s)-\lambda]\, \D{s}}\, \Psi (X_t) \Bigr]
\;\le\; \frac{1}{\min_{\sB}\Psi}\,\Psi(x)
\end{equation}
for all $t\ge0$, and
for any bounded ball $\sB$.
Define $F_\alpha(x) \df f(x)-\lambda -\alpha$, and 
\begin{equation*}
\Gamma_\alpha \;\df\;
\Exp_0\biggl[\int_0^\infty \E^{\int_0^t F_\alpha (X_s)\, \D{s}}\,
g(X_t)\, \D{t}\biggr]\,.
\end{equation*}
for $\alpha>0$, and some $g\in\Cc_{\mathrm{c}}^+(\Rd)$.
From \cref{EL2.7C} we have $\Gamma_\alpha< \infty$ for all $\alpha>0$.
Moreover, $\Gamma_\alpha\to\infty$ as $\alpha\searrow 0$ by hypothesis.
Choose $n_0$ large enough so that $\supp(g)\subset B_{n_0}$.
Following \cite[Theorem~2.8]{ari-anup} we consider the positive solution 
$\varphi_{\alpha, n}\in\Sobl^{2,p}(B_n)\cap\Cc(\Bar{B}_n)$ of 
\begin{equation}\label{EL2.7D}
\Lg\varphi_{\alpha, n} + F_\alpha\, \varphi_{\alpha, n}
\;=\;-\Gamma_\alpha^{-1} g \quad \text{in\ } B_n\,,
\quad \varphi_{\alpha, n}=0\quad\text{on\ } \partial B_n\,,
\end{equation}
for $n\ge n_0$.
Since for every fixed $n$ we have
\begin{equation*}
\Exp_x\Bigl[\E^{\int_0^T F_\alpha(X_s)\, \D{s}}\,\varphi_{\alpha, n}(X_T)\,
\Ind_{\{T\le\uptau_n\}}\Bigr]
\;\le\; \Bigl(\max_{B_n}\,\varphi_{\alpha, n}\Bigr)\,\E^{-\alpha T}\,
\Exp_x\Bigl[\E^{\int_0^T F_0(X_s)\, \D{s}}
\,\Ind_{B_n}(X_T)\Ind_{\{T\le\uptau_n\}}\Bigr]
\;\xrightarrow[T\to\infty]{}\; 0
\end{equation*}
by \cref{EL2.7C}, applying the It\^o--Krylov formula to \cref{EL2.7D},
we obtain by \cite[Theorem~2.8]{ari-anup} that
\begin{equation}\label{EL2.7E}
\varphi_{\alpha, n}(0) \;=\; \Gamma^{-1}_\alpha
\Exp_0\Bigl[\int_0^{\uptau_n} \E^{\int_0^{t} F_\alpha(X_s)\,\D{s}}\,g(X_t)\,\D{t}\Bigr]
\;\le\; \Gamma^{-1}_\alpha \Gamma_\alpha=1\,.
\end{equation}
Since $\Gamma^{-1}_\alpha$ is bounded
uniformly on $\alpha\in(0,1)$ by hypothesis, we can apply
Harnack's inequality for a class of superharmonic functions
\cite[Corollary~2.2]{AA-Harnack}
to conclude that $\{\varphi_{\alpha, n}\,, n\in\NN\}$ is locally  bounded,
and therefore also uniformly bounded in $\Sobl^{2,p}(B_R)$, $p>d$,
for any $R>0$.
Thus, we have that
$\varphi_{\alpha, n}\to\varphi_\alpha$ weakly in 
$\Sobl^{2,p}(\Rd)$ along some subsequence, and that $\varphi_{\alpha}$ satisfies
\begin{equation}\label{EL2.7F}
\Lg\varphi_{\alpha} + F_\alpha\, \varphi_{\alpha}
\;=\;-\Gamma_\alpha^{-1} g \quad \text{in\ } \Rd
\end{equation}
by \cref{EL2.7D}.
Let $\sB$ be an open ball centered at $0$ such that $\supp(g)\subset\sB$.
Applying the It\^{o}--Krylov formula to \cref{EL2.7D} we obtain
\begin{equation*}
\varphi_{\alpha, n}(x)\;=\;
\Exp_x\Bigl[\E^{\int_0^{\uuptau\wedge T}F_\alpha(X_s)\, \D{s}}\,
\varphi_{\alpha, n}(X_{\uuptau\wedge T})\,\Ind_{\{\uuptau\wedge T<\uptau_n\}}\Bigr]\,,
\quad x\in B_n\setminus\Bar\sB, \quad \forall\; T>0\,,
\end{equation*}
with $\uuptau=\uuptau(\sB^c)$.
As in the derivation of \cref{EL2.7E}, using \cref{EL2.7C} and
a similar argument we obtain
\begin{equation}\label{EL2.7H}
\varphi_{\alpha, n}(x)\;=\;\Exp_x\Bigl[\E^{\int_0^{\uuptau}F_\alpha(X_s)\, \D{s}}\,
\varphi_{\alpha, n}(X_{\uuptau})\,\Ind_{\{\uuptau<\uptau_n\}}\Bigr]\,,
\quad x\in B_n\setminus\Bar\sB\,.
\end{equation}
Letting $n\to\infty$ along some subsequence, and arguing as above,
we obtain a function $\varphi_\alpha$
which satisfies \cref{EL2.7F} and
\begin{equation}\label{EL2.7I}
\varphi_{\alpha}(x)\;=\;\Exp_x\Bigl[\E^{\int_0^{\uuptau}F_\alpha(X_s)\, \D{s}}\,
\varphi_{\alpha}(X_{\uuptau})\,\Ind_{\{\uuptau<\infty\}}\Bigr]\,,
\quad x\in \Bar\sB^c,
\end{equation}
where \cref{EL2.7I} follows from \cref{EL2.7H}.
From \cref{EL2.7E} we have $\varphi_\alpha(0)=1$ for all $\alpha\in(0, 1)$.
Now applying Harnack's inequality once again and letting $\alpha\searrow 0$,
we deduce that $\varphi_\alpha$ converges weakly in $\Sobl^{2,p}(\Rd)$, $p>d$,
to some positive function $\Psi$ which satisfies
$\Lg\Psi + F_0\, \Psi =0$ in $\Rd$,
and
\begin{equation*}
\Psi(x)\;=\;\Exp_x\Bigl[\E^{\int_0^{\uuptau}F_0(X_s)\, \D{s}}\,\Psi(X_{\uuptau})
\,\Ind_{\{\uuptau<\infty\}}\Bigr]\,,
\quad x\in \Bar\sB^c.
\end{equation*}
This implies (iii). 

Lastly, suppose that (iii) holds.
In other words, there exists an eigenpair $(\Psi, \lambda)$ and an open ball
$\sB$ such that
\begin{equation}\label{EL2.7J}
\Psi(x)\;=\;\Exp_x\Bigl[\E^{\int_0^{\uuptau}F_0(X_s)\, \D{s}}\,\Psi(X_{\uuptau})
\,\Ind_{\{\uuptau<\infty\}}\Bigr]\,,
\quad x\in \sB^c.
\end{equation}
We first show that $\lambda$ is a simple eigenvalue, which implies that
there is a unique twisted process $Y$
corresponding to $\lambda$.
To establish the simplicity of $\lambda$ consider another eigenpair
$(\Tilde\Psi,\lambda)$ of $\Lg^f$.
By the It\^{o}--Krylov formula we obtain
\begin{equation*}
\Tilde\Psi(x)\;\ge\;\Exp_x\Bigl[\E^{\int_0^{\uuptau}F_0(X_s)\, \D{s}}\,
\Tilde\Psi(X_{\uuptau})
\,\Ind_{\{\uuptau<\infty\}}\Bigr]\,,
\quad x\in \sB^c\,.
\end{equation*}
Thus using \cref{EL2.7J} and an argument similar to \cref{L2.4} we can show that
$\Psi=\Tilde\Psi$.
Then (iii)$\,\Rightarrow\,$(i) follows from \cite[Lemma~2.6]{ari-anup}.

Uniqueness of the eigenfunction $\Psi$ follows from the stochastic representation in
\cref{EL2.7J} and the proof of (iii)$\,\Rightarrow\,$(i).
\qed\end{proof}

As an immediate corollary to \cref{L2.6,L2.7} we have the following.

\begin{corollary}\label{C2.3}
Let $(\Psi, \lambda)$ be an eigenpair of $\Lg^f$ which satisfies
\begin{equation*}
\Psi(x)\;=\;\Exp_x\Bigl[\E^{\int_0^{\uuptau}[f(X_s)-\lambda]\, \D{s}}\,\Psi(X_{\uuptau})
\,\Ind_{\{\uuptau<\infty\}}\Bigr]
\quad \forall\,x\in \sB^c\,.
\end{equation*}
for some bounded open ball $\sB$ in $\Rd$.
Then $\lambda=\lamstr(f)$, and it is a simple eigenvalue.
\end{corollary}

\cref{T2.3} below is a generalization of
\cite[Theorem~4.7.1]{Pinsky} in $\Rd$, which is stated in bounded domains, and for
bounded and smooth coefficients.
It is shown in \cite{Pinsky} that for smooth bounded domains, the Green's measure
is not defined at the critical value $\lamstr$ \cite[Theorem~3.2]{Pinsky}.
But by \cref{T2.3} below we see that this is not the case on $\Rd$.
In fact, \cite[Theorem~4.3.2]{Pinsky} shows that $\lambda^*$ could be either
subcritical or critical in the sense of Pinsky.
We show that the criticality of $\lambda^*$ is equivalent to the strict
monotonicity of $\lamstr(f)$ on the right, i.e.,
$\lamstr(f)<\lamstr(f+h)$ for all $h\in\Cc_{\mathrm{o}}^+(\Rd)$.

\begin{theorem}\label{T2.3}
A ground state process is recurrent if and only if 
$\lamstr(f)<\lamstr(f+h)$ for all $h\in\Cc_{\mathrm{o}}^+(\Rd)$.
\end{theorem}

\begin{proof}
Suppose first that a ground state process corresponding to $\lamstr(f)$ is recurrent.
Then $G_{\lamstr}(g)=\infty$ for all
$g\in\Cc_{\mathrm{c}}^+(\Rd)$ by \cref{L2.7}.
Let $\Tilde{f}=f+h$ and $\Tilde\lambda\str\df\lamstr(f+h)$. 
Suppose that $\lamstr=\Tilde\lambda\str$.
Let $\Tilde\Psi$ be a principal eigenfunction of $\Lg^{\Tilde{f}}$, i.e.,
\begin{equation}\label{ET2.3B}
\Lg\Tilde\Psi + \Tilde{f}\,\Tilde\Psi\;=\;\Tilde\lambda\str\Tilde\Psi.
\end{equation}
Writing \cref{ET2.3B} as
$\Lg\Tilde\Psi + (f-\lamstr)\Tilde\Psi=-h \Tilde\Psi$,
and applying the It\^{o}--Krylov formula, followed by Fatou's lemma, we obtain
\begin{equation*}
\Exp_x\Bigl[\E^{\int_0^T [f(X_s)-\lamstr]\, \D{s}}\, \Tilde\Psi(X_T)\Bigr]
+ \int_0^T\Exp_x\Bigl[\E^{\int_0^t [f(X_s)-\lamstr]\, \D{s}}\,
h(X_t)\Tilde\Psi(X_t)\Bigr]\, \D{t}\;\le\; \Tilde\Psi(x)\,,
\end{equation*}
which contradicts the property that $G_{\lamstr}(g)=\infty$ for all
$g\in\Cc_{\mathrm{c}}^+(\Rd)$.
Therefore,  $\lamstr(f)<\lamstr(f+h)$
for all $h\in\Cc_{\mathrm{o}}^+(\Rd)$.

To prove the  converse, suppose that $Y^*$ is transient. 
Then for $g\in\Cc_{\mathrm{c}}^+(\Rd)$ with $B_1\subset\supp(g)$ we have
$G_{\lamstr}(g)<\infty$.
Following the arguments in the proof of
(ii)$\,\Rightarrow\,$(iii) in \cref{L2.7}, we obtain a positive $\Phi$ satisfying
\begin{equation}\label{ET2.3D}
\Lg\Phi + (f-\lamstr)\, \Phi\;=\;-\Gamma_0^{-1} g\,.
\end{equation}
Let $\varepsilon=\Gamma_0^{-1}\min_{B_1} \frac{g}{\Phi}$.
Then from \cref{ET2.3D} we have
\begin{equation*}
\Lg\Phi + (f+\varepsilon\Ind_{B_1}-\lamstr)\, \Phi\;\le\; 0\,.
\end{equation*}
This implies that
$\lamstr(f+\varepsilon\Ind_{B_1})\le \lamstr(f)$
by \cref{L2.2}\,(ii).
Thus $\lamstr(f+\varepsilon\Ind_{B_1})= \lamstr(f)$.
Therefore, if $\lamstr(f)<\lamstr(f+h)$ for all $h\in\Cc_{\mathrm{o}}^+(\Rd)$,
then $Y^*$ has to be recurrent.
This completes the proof.
\qed\end{proof}

It is well known that a (null) recurrent diffusion
$\{X_t\}$ with locally uniformly elliptic
and Lipschitz continuous $a$, and locally bounded measurable drift,
admits a $\sigma$-finite invariant probability measure $\nu$ which is a Radon measure
on the Borel $\sigma$-field of $\Rd$ \cite{Hasminskii}.
This measure is equivalent to the Lebesgue measure and is unique up to a multiplicative
constant.
Theorem~8.1 in \cite{Hasminskii} states that if $g$ and $h$ are real-valued functions
which are integrable with respect to the measure $\nu$ then
\begin{equation}\label{EC2.4A}
\Prob_x\,\left( \lim_{T\to\infty}\;\frac{\int_0^T g(X_t)\,\D{t}}{\int_0^T h(X_t)\,\D{t}}
\;=\; \frac{\int_\Rd g(x)\,\nu(\D{x})}{\int_\Rd h(x)\,\nu(\D{x})}\right)\;=\;1\,.
\end{equation}
Suppose $g\colon\Rd\to\RR_+$ is a non-trivial function.
Select $h$ as the indicator function of some open ball.
Then it is well known that the expectation
of $Y^h_t \df \int_0^t h(X_t)\,\D{t}$ tends to $\infty$ as $t\to\infty$.
Adopt the analogous notation $Y^g_t$, and let $\alpha=\frac{\nu(g)}{2\nu(h)}$.
Let $M>0$ be arbitrary, and select $t_0$ large enough such that
$\Exp\bigl[Y^h_{t_0}\bigr]\ge 2M$.
Then of course we may find a positive constant $\kappa$ such
$\Exp\bigl[Y^h_{t_0}\,\Ind_{\{Y^h_{t_0}\le\kappa\}}\bigr]\ge M$.
Since $Y^h_t$ and $Y^g_t$ are nondecreasing in $t$, it follows by
\cref{EC2.4A} that
\begin{equation*}
\Prob_x\,\Bigl(\lim_{t\to\infty}\,\bigl(Y^g_t-\alpha Y^h_{t_0}\bigr)^-
\,\Ind_{\{Y^h_{t_0}\le\kappa\}}\,\ge 0\Bigr)\;=\;0\,.
\end{equation*}
This of course implies, using dominated convergence, that
$\liminf_{t\to\infty}\,\Exp\bigl[Y^g_{t}\,\Ind_{\{Y^h_{t_0}\le\kappa\}}\bigr]
\ge \alpha M$. 
Since $M$ was arbitrary, this shows
that $\Exp\bigl[Y^g_{t}\bigr]\to\infty$ as $t\to\infty$, or equivalently that
$\int_0^\infty \Exp_x[g(X_t)]\,\D{t}=\infty$.
Using this property in the proof of \cref{T2.3} we obtain the following corollary.

\begin{corollary}\label{C2.4} 
For $\lamstr(f)$ to be strictly monotone at $f$ on the right it is sufficient
that there exists some non-trivial
Borel measurable bounded function $g\colon\Rd\to\RR_+$ with compact support
satisfying $\lamstr(f+\epsilon\, g)>\lamstr(f)$ for all $\epsilon>0$.
\end{corollary}

\subsubsection{Minimal growth at infinity}

We next discuss the property known as
\textit{minimal growth at infinity} \cite[Definition~8.2]{Berestycki-15}.
As shown in  \cite[Proposition~8.4]{Berestycki-15}, minimal growth at infinity
implies that the eigenspace corresponding  
to the eigenvalue $\lamstr(f)$ is one dimensional, i.e., $\lamstr(f)$ is simple.
We start with the following definition, which is a variation
of \cite[Definition~8.2]{Berestycki-15}.

\begin{definition}
A positive function $\varphi\in\Sobl^{2,d}(\Rd)$
is said to be a solution of minimal growth at infinity of
$\Lg^f\varphi - \lambda\varphi=0$,
if for any $r>0$ and any positive function
$v\in\Sobl^{2,d}(\Rd\setminus B_r)$ satisfying
$\Lg^f v - \lambda v\le 0$ a.e., in $B_r^c$, there exists 
$R>r$ and $k>0$ such that $k\varphi\le v$ in $B^c_R$.
\end{definition}

Define the generalized principal
eigenvalue of $\Lg^f$ in the domain $D$ by
\begin{equation*}
\lambda_1(f,D)\;\df\;\inf\;\bigl\{\lambda\; :\;
\exists\, \varphi\in\Sobl^{2,d}(D),\;
\varphi>0, \; \Lg\varphi + (f-\lambda)\varphi\;\le\; 0 \text{\ a.e. in}\; D
\bigr\}\,.
\end{equation*}
Note that $\lambda_1(f,\Rd)=\Hat\Lambda(f)=\lamstr(f)$.
It is also clear from this definition that for $D_1\subset D_2$ we have
$\lambda_1(f,D_1)\le \lambda_1(f,D_2)$.

It is shown in \cite[Theorem~8.5]{Berestycki-15} that the hypothesis
\begin{enumerate}
\item[(A1)]
$\lim_{r\to\infty}\;\lambda_1(f,B_r^c)\;<\; \lamstr(f)$
\end{enumerate}
implies that the ground state $\Psi^*$ of $\Lg^f$ is a solution of minimal growth at
infinity.

On the other hand, the following result has been established
in \cite[Theorem~2.1]{ABG-arxiv}.

\begin{theorem}
The ground state $\Psi^*$ of $\Lg^f$ is a solution of minimal growth at infinity
of $\Lg^f\Psi^* - \lamstr(f)\Psi^*=0$
if and only if $\lamstr(f)$ is strictly monotone at $f$ on the right.
\end{theorem}

It thus follows by the above results that (A1) is a sufficient condition
for strict monotonicity of $\lamstr(f)$ on the right.
It turns out that (A1) is equivalent to strict monotonicity
and, moreover, the map $r\mapsto\lambda_1(f,B_r^c)- \lamstr(f)$ is either
negative on $(0,\infty)$ or identically equal to $0$.
This is the subject of the following theorem.

\begin{theorem}
The following are equivalent.
\begin{enumerate}[(a)]
\item
$\exists\, r>0\,\colon\;\lambda_1(f,B_r^c)<\lambda_1(f,\Rd)$.

\item
$\lamstr(f)$ is strictly monotone at $f$.

\item
$\lambda_1(f,B_r^c)<\lambda_1(f,\Rd) \quad\forall\,r>0$.
\end{enumerate}
\end{theorem}

\begin{proof}
It easily follows by \cref{L2.5,T2.2} and the definition of $\lambda_1$ that
(b)$\,\Rightarrow\,$(c).
Thus it remains to prove that (a)$\,\Rightarrow\,$(b).
Suppose that $\lambda\equiv\lambda_1(f,B_{\Bar{r}}^c)<\lambda_1(f,\Rd)=\lamstr(f)$
for some $\Bar{r}>0$.
Using the Dirichlet eigenvalues for the annulus
$\sB_{r}\setminus\Bar{B}_{\Bar{r}}$, for $r>\Bar{r}$,
and letting $r\to\infty$, we can construct a solution
$\psi\in\Sobl^{2,d}(\Bar{B}_{\Bar{r}}^c)$ of
$\Lg \psi + f\psi = \lambda\psi$ on $B_{\Bar{r}}^c$, with $\psi>0$ on
$\Bar{B}_{\Bar{r}}^c$, and $\psi=0$ on $\partial B_{\Bar{r}}$.
Then $\psi$ is bounded away from $0$ on $\partial B_{r'}$ for all $r'>r$.
Using any $r'>r$, we extend $\psi$ smoothly inside $B_{r'}$ to obtain some function
$\varphi\in\Sobl^{2,d}(\Rd)$ which is strictly positive on $\Rd$ and agrees with
$\psi$ on $B_{r'}^c$.
Let $h\df \lambda \varphi-\Lg \varphi - f\varphi$,
and $\Tilde{f} \df f+ \frac{h}{\varphi}$.
Then $\Lg \varphi + \Tilde{f}\varphi \;=\; \lambda \varphi$,
and therefore, we have
\begin{equation*}
\lamstr\bigl(f+ \tfrac{h^-}{\varphi}\bigr)
\;\le\; \lamstr(\Tilde{f})\;\le\;\lambda\;<\;\lamstr(f)\,,
\end{equation*}
which implies strict monotonicity at $f$, and completes the  proof.
\qed\end{proof}

\subsection{Potentials \texorpdfstring{$f$}{f} vanishing at infinity}
\label{S2.4}

Let $\cB_{\mathrm{o}}(\Rd)$ denote the class of bounded
Borel measurable functions which are vanishing at infinity, i.e.,
satisfying $\lim_{R\to\infty}\;\sup_{B^c_R}\,\abs{f}=0$,
and $\cB^+_{\mathrm{o}}(\Rd)$ the class of nonnegative functions
in $\cB_{\mathrm{o}}(\Rd)$ which are not a.e.\ equal to $0$.

\cref{T2.4} which follows is a (pinned) multiplicative ergodic theorem
(compare with \cite[Theorem~7.1]{Kaise-04}).
Note that the continuity result in this theorem  is stronger than that
of \cite[Proposition~9.2]{Berestycki-15}.
See also \cref{R4.1} on the continuity of $\lamstr(f)$ for a larger class of $f$.
We introduce the eigenvalue $\lambda^{\prime\prime}(f)$ defined by
\begin{equation}\label{E-l''}
\lambda^{\prime\prime}(f)\;\df\;\inf\;\Bigl\{\lambda\, \colon\;
\exists\, \varphi\in\Sobl^{2,d}(\Rd), \; \inf_{\Rd}\varphi>0, \;
\Lg\varphi+ (f-\lambda)\varphi\;\le\; 0\text{\ a.e. in\ }\Rd\Bigr\}\,.
\end{equation}

\begin{theorem}\label{T2.4}
Let $f\in\cB_{\mathrm{o}}(\Rd)$.
If the solution of \cref{E2.1} is recurrent,
then $\lamstr(f)=\lambda^{\prime\prime}(f)=\sE(f)$.
In addition, if the solution of \cref{E2.1} is positive recurrent
with invariant measure $\mu$, and $\int_{\Rd} f\, \D{\mu}>0$,
the following hold:
\begin{itemize}
\item[\textup{(}a\textup{)}]
for any measurable $g$ with compact support we have
\begin{equation}\label{ET2.4A}
\Exp_x\Bigl[\E^{\int_0^T [f(X_s)-\lamstr(f)]\, \D{s}}\, g(X_T)\Bigr]
\;\xrightarrow[T\to\infty]{}\; C_g \Psi^*(x)\,,
\end{equation}
for some positive constant $C_g$.
Moreover, the corresponding twisted process $Y^*$ is exponentially ergodic.

\item[\textup{(}b\textup{)}]
If $f_n$ is a sequence of functions in $\cB_{\mathrm{o}}(\Rd)$
satisfying $\sup_{n}\norm{f_n}_\infty<\infty$,
and converging to $f$ in $\Lpl^1(\Rd)$, and also uniformly outside some
compact set $K\subset\Rd$, then
$\lamstr(f_n)\to\lamstr(f)$.
\end{itemize}
\end{theorem}

\begin{proof}
Applying the It\^{o}--Krylov formula to
$\Lg\varphi+ (f-\lambda)\varphi\le 0$, it is easy to see that
$\sE(f)\le \lambda^{\prime\prime}(f)$.
Also, from \cite[Lemma~2.3]{ari-anup} we have $\lamstr(f)\le \sE(f)$.
Thus we obtain
$\lamstr(f)\le \sE(f)\le  \lambda^{\prime\prime}(f)$.
If $\lamstr(f)\ge \lim_{\abs{x}\to\infty}\, f(x)$, then by
\cite[Theorem~1.9\,(iii)]{Berestycki-15}
we have $\lamstr(f)=\lambda^{\prime\prime}(f)$ which in turn implies
that $\lamstr(f)= \sE(f)=  \lambda^{\prime\prime}(f)$.
On the other hand, if $\lamstr(f)<\lim_{\abs{x}\to\infty}\, f(x)$,
then $f$ is near-monotone, relative to $\lamstr(f)$, in the sense of \cite{ari-anup}.
Applying \cite[Lemma~2.1]{ari-anup} we again obtain
$\lamstr(f)= \sE(f)=  \lambda^{\prime\prime}(f)$. 

We now turn to part (a).
Applying Jensen's inequality it is easy to see that $\sE(f)\ge \int f\,\D{\mu}>0$.
Therefore, $\lamstr(f-f^{+})\le 0< \lamstr(f)$. Taking $h=f^+$ and mimicking
the arguments of \cref{T2.1} we see that
$Y^*$ is exponentially ergodic.
Let $\mu^*$ be the unique invariant measure of $Y^*$.
Then
\cref{ET2.4A} follows from \cref{EL2.6B} and \cite[Theorem~1.3.10]{Kunita} with
$C_g\;=\;\int \frac{g}{\Psi^*}\, \D{\mu^*}$.

Next we prove part (b).
By the first part of the theorem we have $\lamstr(f_n)=\sE(f_n)$ for all $n$,
and by the lower-semicontinuity property of $\lamstr$ it holds that
$\liminf_{n\to\infty}\, \lamstr(f_n)\ge \lamstr(f)$.
Let $h\in\Cc_{\mathrm{c}}^+(\Rd)$ and
$\Tilde{f}=f-h$. Then by \cref{T2.2} we have $2\delta:=\lamstr(f)-\lamstr(\Tilde{f})>0$.
Choose a open ball $\sB$, containing $K$,
such that $\sup_{x\in \sB^c} |f_n-f|<\delta$ and $\lamstr(f_n)>\lamstr(f)-\delta$
for all sufficiently large $n$.
Let $(\Psi^*_n, \lamstr(f_n))$ denote
the principal eigenpair.
Then
\begin{equation}\label{ET2.4B}
\Lg \Psi^*_n + f_n\, \Psi^*_n \;=\;\lamstr(f_n)\, \Psi^*_n\,.
\end{equation}
We can choose $\sB$ large enough
 such that
\begin{equation}\label{ET2.4C}
\Psi^*_n(x)\; =\; \Exp_x\Bigl[\E^{\int_0^{\uuptau}[f_n(X_t)-\lamstr(f_n)]\, \D{t}} \,
\Psi^*_n(X_{\uuptau})\Bigr]\,, \quad x\in\sB^c, \quad
\forall \; n\in\NN\,,
\end{equation}
where $\uuptau=\uuptau(\sB)$.
Suppose $\limsup_{n\to\infty}\lamstr(f_n)=\Lambda$.
It is standard to show that for some positive $\Psi$, it holds that
$\Psi^*_n\to\Psi$ weakly in $\Sobl^{2,p}(\Rd)$, $p>d$, as $n\to\infty$,
and therefore, from \cref{ET2.4B} we have
\begin{equation*}
\Lg \Psi + f\, \Psi \;=\;\Lambda\, \Psi\,.
\end{equation*}
Therefore, $\Lambda\ge \lamstr(f)$. Note that on $\sB^c$ we have
\begin{equation*}
f_n-\lamstr(f_n)\;\le\; f+\delta -\lamstr(f_n)\;\le\; f-\lamstr(f)+2\delta
\;=\; f-\lamstr(\Tilde{f})
\end{equation*}
for all $n$ sufficiently large.
Since
$\Exp_x\bigl[\E^{\int_0^{\uuptau}[f(X_t)-\lamstr(\Tilde{f})]\, \D{t}} \, \bigr]
<\infty$,
passing to the limit in \cref{ET2.4C},
and using the dominated convergence theorem, we obtain that
\begin{equation*}
\Psi(x)\; =\; \Exp_x\Bigl[\E^{\int_0^{\uuptau}[f(X_t)-\Lambda]\, \D{t}} \,
\Psi(X_{\uuptau})\Bigr]\,, \quad x\in\sB^c\, .
\end{equation*}
Therefore, $\Lambda=\lamstr(f)$ by \cref{C2.3}.
This completes the proof.
\qed\end{proof}

We pause for a moment to provide an example where (P2) holds but
(P1) fails.

\begin{example}
Let $d=2$ and $\Lg=\Delta$.
If $f=0$, then the ground state is a constant function,
and in turn, the ground state diffusion is a two dimensional Brownian
motion, hence recurrent.
It follows that $\lamstr$ is strictly monotone on the right at $0$.
Now let $f$ a non-trivial non-negative continuous function with compact support.
It is clear that $\lamstr(\beta f)\le 0$ for $\beta\le0$.
On the other hand, by \cref{T2.4},
we have $\lamstr(\beta f)=\sE(\beta f)$ for all $\beta\in\RR$.
Therefore, for $\beta\le0$, we have
\begin{align*}
0\;\ge\;\lamstr(\beta f)\;=\;\sE(\beta f)
&\;=\; \limsup_{T\to\infty}\, \frac{1}{T}\;
\log\,\Exp_x\Bigl[\E^{\int_0^T \beta f(X_s)\, \D{s}}\Bigr]\\
&\;\ge\;
\limsup_{T\to\infty}\; \frac{1}{T} \Exp\biggl[\int_0^T \beta f(X_s)\,\D{s}\biggr]
\;=\;0\,.
\end{align*}
Thus $\lamstr(\beta f)=0$ for all $\beta\leq 0$,
which implies that $\lamstr$ it is not strictly monotone at $0$.
\end{example}

In the rest of this section
we show how the previous development can be used to obtain
results analogous to those reported in \cite{Ichihara-13b},
without imposing any smoothness assumptions on the coefficients.
With $\Breve\psi=-\log\Psi^*=-\psi^*$, we have
\begin{equation}\label{ABEx01}
-a^{ij}\partial_{ij}\Breve\psi - b^i\, \partial_i\Breve\psi
+ \langle \Breve\psi, a\Breve\psi\rangle + f\;=\;\lamstr(f)\,.
\end{equation}
Note that \cref{ABEx01} is a particular form of a more general class of
quasilinear pdes of the form
\begin{equation}\label{ABEx02}
-a^{ij}\partial_{ij}\Breve\psi + H(x, \grad\Breve\psi) + f\;=\;\lamstr(f)\,,
\end{equation}
where the function $H(x,p)$, with $(x,p)\in\Rd\times\Rd$, serves
as a Hamiltonian.
Let $f$ be a non-constant, nonnegative
continuous function satisfying $\lim_{\abs{x}\to\infty} f(x)=0$, and define
$\Lambda_\beta\df\lamstr(\beta f)$, $\beta\in\RR$.
Then by \cite[Proposition~2.3\,(vii)]{Berestycki-15} we know that
$\beta\mapsto\Lambda_\beta$ is non-decreasing and convex.
For the diffusion matrix $a$ equal the identity,
Ichihara studies some qualitative properties of
$\Lambda_\beta$ in \cite{Ichihara-13b} associated to the pde \cref{ABEx02},
and their relation to the recurrence and transience 
behavior of the process with generator
\begin{equation*}
\sA^{\Breve\psi} g\;=\; \Delta g
- \langle \grad_pH(x, \grad\Breve\psi),\, \grad g\rangle, \quad g\in
\Cc^2_{c}(\Rd)\,.
\end{equation*}
It is clear that if
$H(x, p)=-\langle b(x), p\rangle + \langle p, a(x) p\rangle$,
then $\sA^{\Breve\psi}$
is the generator of the twisted process $Y^*$ corresponding to $\Psi^*$.
One of the key assumptions in \cite[Assumption~(H1)\,(i)]{Ichihara-13b} is that
$H(x, p)\ge H(x, 0)=0$ for all
$x$ and $p$. Note that this forces $b$ to be $0$. 

Let $$\beta_{c}\;\df\;\inf\,
\Bigl\{\beta\in\RR\; \colon \Lambda_\beta
>\lim_{\beta\to -\infty} \Lambda_\beta\Bigr\}\,.$$
It is easy to see that $\beta_c\in[-\infty, \infty]$.
The following result is an extension
of \cite[Theorems~2.2 and ~2.3]{Ichihara-13b}
to measurable drifts $b$ and potentials $f$.

\begin{theorem}\label{T2.5}
Let $f\in\cB^+_{\mathrm{o}}(\Rd)$.
Then the twisted process $Y^*=Y^*(\beta)$ corresponding to the eigenpair
$(\Psi^*_\beta, \Lambda_\beta)$
is transient for  $\beta<\beta_c$, exponentially ergodic for $\beta>\beta_c$,
and, provided $f=0$ a.e.\ outside some compact set, it is recurrent
for $\beta=\beta_c$.
In addition, the following hold.
\begin{enumerate}[(i)]
\item
If $\Lg^{0}$ is self-adjoint \textup{(}i.e.,
$\Lg^0=\partial_i(a^{ij}\partial_j)$\/\textup{)},
with the matrix $a$ bounded, uniformly elliptic and
radially symmetric in $\Rd$, and the solution of \cref{E2.1} is transient,
then $\beta_c\ge 0$.
Also $\Lambda_\beta\ge 0$ for all  $\beta\in\RR$.

\item
Provided that the solution of \cref{E2.1} is recurrent,
then $\beta_c<0$ if it is exponentially ergodic,
and $\beta_c=0$ otherwise.

\item
Assume that $\beta>\beta_c$, and that \cref{E2.1} is recurrent in the case that
$\Lambda_\beta\le0$.
Let $\Psi^*_\beta$ and $\mu^*_\beta$ denote the ground
state and the invariant probability measure
of the ground state diffusion, respectively, corresponding to $\Lambda_\beta$.
Then it holds that
\begin{equation}\label{ET2.5A}
\Lambda_\beta\;=\;
\mu^*_\beta\bigl(\beta f-\langle \grad\psi^*_\beta, a \grad\psi^*_\beta\rangle\bigr)\,,
\end{equation}
where, as usual, $\psi^*_\beta=\log\Psi^*_\beta$.
\end{enumerate}
\end{theorem}

\begin{proof}
The first part of the proof follows from \cref{T2.1,T2.3}, and \cref{C2.4}.
Next we proceed to prove (i).
Suppose $\beta_c<0$. Then $Y^*=Y^*(0)$, i.e., the twisted process corresponding to
$\Lambda_0$, is exponentially ergodic.
By \cite[Theorem~1.9\,(i)--(ii)]{Berestycki-15} we have $\Lambda_0=\sE(0)=0$.
Moreover, $\Psi^*_0=1$ is a ground state.
Therefore, the twisted process must be given by \cref{E2.1}, which is transient 
by hypothesis.
This is a contradiction.
Hence $\beta_c\ge 0$.
Since $\beta\mapsto\Lambda_\beta$ is convex, it follows that $\Lambda_\beta$
is constant in $(-\infty, \beta_c]\ni 0$.
Hence $\Lambda_\beta=\Lambda_0=0$ for $\beta\le\beta_c$.
This proves (i).

We now turn to part (ii).
By \cref{T2.4} we have $\Lambda_\beta=\sE(\beta f)$.
We claim that if the solution of \cref{E2.1} is recurrent then $\lambda^*(\beta f)>0$,
whenever $\beta>0$.
Indeed, arguing by contradiction, if $\lambda^*(\beta f)=0$ for some
$\beta>0$, then
$\Lg \Psi^*_\beta =  - \beta f \Psi^*_\beta$ on $\Rd$, which implies that
that $\Psi^*_\beta(X_t)$ is a nonnegative supermartingale, and
since it is integrable, it converges a.s.
Since the process $X$ is recurrent, this implies that $\Psi^*_\beta$
must equal to a constant, which, in turn,
necessitates that $f=0$, a contradiction.
This proves the claim, which in turn implies that 
if the solution of \cref{E2.1} is recurrent then $\beta_c\le0$.
Now suppose that $\beta_c$ is negative.
Then the twisted process corresponding to
$\beta=0$ is exponentially ergodic by \cref{T2.1}.
Since $\Psi^*_0=1$,
the ground state diffusion for $\beta=0$ agrees with \cref{E2.1}, which implies
that the latter is exponentially ergodic.

Next, suppose that $X$, and therefore also $Y^*(0)$ is exponentially ergodic.
It then follows from \cref{T2.2} that $\beta\mapsto \Lambda_\beta$ is strictly
monotone at $0$.
This of course implies that $\beta_c<0$.
The proof of part (ii) is now complete.

Next we prove part (iii).
We distinguish two cases.

\noindent\emph{Case 1.} Suppose $\Lambda_\beta>0$.
Let $\widecheck\Psi=\widecheck\Psi_\beta\df(\Psi^*_\beta)^{-1}$
and $\Breve\psi\df\log\widecheck\Psi$.
Then $\widecheck\Psi$ satisfies
\begin{equation}\label{ET2.5B}
\widetilde\Lg^{\psi^*_\beta}\, \widecheck\Psi \;=\;\bigl(\beta f - \Lambda_\beta
\bigr) \widecheck\Psi
\end{equation}
Since $\beta f\in\cB_{\mathrm{o}}(\Rd)$, there exists
$\epsilon_\circ>0$ and a ball $\sB$
such that $\beta f-\Lambda_\beta<-\epsilon_\circ$ for all $x\in\sB^c$.
Applying the Feynman--Kac formula, it follows from \cite[Lemma~2.1]{ari-anup}
that  $\inf_\Rd\,\widecheck\Psi =\min_{\Bar\sB}\,\widecheck\Psi$.
Thus $\widecheck\Psi$ is bounded away from $0$ on $\Rd$.
Let $Y^*$ denote the ground state process corresponding to the
eigenvalue $\Lambda_\beta$.
Simplifying the notation we let $\widetilde\Exp_x^*\df \widetilde\Exp_x^{\psi^*_\beta}$.
By the exponential Foster--Lyapunov equation \cref{ET2.5B} we have
that (see \cite[Lemma~2.5.5]{book})
\begin{equation}\label{ET2.5C}
\widetilde\Exp_x^*\bigl[\widecheck\Psi(Y^*_t)\bigr]
\;\le\; C_0 + \widecheck\Psi(x)\,\E^{-\epsilon_\circ t}\qquad\forall\,t\ge0\,.
\end{equation}
Using this estimate together with the fact that $\inf_{\Rd}\widecheck\Psi>0$,
we obtain
\begin{equation}\label{ET2.5D}
\lim_{t\to\infty}\; \frac{1}{t}\,\widetilde\Exp_x^*\bigl[\Breve\psi(Y^*_t)\bigr]
\;=\;0\,.
\end{equation}
Next, we show that
\begin{equation}\label{ET2.5E}
\lim_{R\to\infty}\;
\widetilde\Exp_x^*\bigl[\Breve\psi(Y^*_{t\wedge\uptau_R})\bigr]
\;=\;\widetilde\Exp_x^*\bigl[\Breve\psi(Y^*_{t})\bigr]\,,
\end{equation}
where $\uptau_R$ denotes the exit time from the ball $B_R$.
First, there exists some constant $k_0$ such that
$\bigl(\beta f - \Lambda_\beta
\bigr) \widecheck\Psi\le k_0$ on $\Rd$.
Thus $\widetilde\Exp_x^*\bigl[\widecheck\Psi(Y^*_t)\bigr] \le k_0 t + \widecheck\Psi(x)$
by \cref{ET2.5B}, and of course also
$\widetilde\Exp_x^*\bigl[\widecheck\Psi(Y^*_{t\wedge\uptau_R})\bigr]
\le k_0 t + \widecheck\Psi(x)$ for all $R>0$.
Let $\Gamma(R,m) \df \{x\in\partial B_R\colon \abs{\Breve{\psi}(x)}\ge m\}$
for $m\ge 1$.
Then
\begin{align*}
\widetilde\Exp_x^*\bigl[\Breve\psi(Y^*_{\uptau_R})\,\Ind_{\{t\ge\uptau_R\}}\bigr]
&\;\le\; m\,\widetilde\Prob_x^*(t\ge\uptau_R)
+ \widetilde\Exp_x^*\bigl[\Breve\psi(Y^*_{\uptau_R})\,
\Ind_{\Gamma(R,m)}(Y^*_{\uptau_R})\,\Ind_{\{t\ge\uptau_R\}}\bigr]\\[5pt]
&\;\le\;m\,\widetilde\Prob_x^*(t\ge\uptau_R)
+\bigl(k_0 t + \widecheck\Psi(x)\bigr)\,\sup_{\Gamma(R,m)}\,
\frac{\Breve\psi}{\widecheck\Psi}\\[5pt]
&\;\le\;m\,\widetilde\Prob_x^*(t\ge\uptau_R)
+\frac{m}{\E^m}\,\bigl(k_0 t + \widecheck\Psi(x)\bigr)\,.
\end{align*}
Taking limits as $R\to\infty$, and since $m\in\RR_+$ is arbitrary, it follows
that
\begin{equation}\label{ET2.5F}
\lim_{R\to\infty}\;
\widetilde\Exp_x^*\bigl[\Breve\psi(Y^*_{\uptau_R})\,\Ind_{\{t\ge\uptau_R\}}\bigr]
\;=\;0\,.
\end{equation}
Write
\begin{equation}\label{ET2.5G}
\widetilde\Exp_x^*\bigl[\Breve\psi(Y^*_{t\wedge\uptau_R})\bigr]
\;=\; \widetilde\Exp_x^*\bigl[\Breve\psi(Y^*_t)\,\Ind_{\{t<\uptau_R\}}\bigr]
+\widetilde\Exp_x^*\bigl[\Breve\psi(Y^*_{\uptau_R})\,\Ind_{\{t\ge\uptau_R\}}\bigr]\,.
\end{equation}
Without loss of generality we assume $\widecheck\Psi\ge1$.
Since $\abs{\Breve\psi}\le \widecheck\Psi$, an application of Fatou's lemma shows that
\begin{align*}
\widetilde\Exp_x^*\bigl[\Breve\psi(Y^*_t)\bigr]&\;\le\;
\liminf_{R\to\infty}\;
\widetilde\Exp_x^*\bigl[\Breve\psi(Y^*_t)\,\Ind_{\{t<\uptau_R\}}\bigr]\\[5pt]
&\;\le\;
\limsup_{R\to\infty}\;
\widetilde\Exp_x^*\bigl[\Breve\psi(Y^*_t)\,\Ind_{\{t<\uptau_R\}}\bigr]
\;\le\; \widetilde\Exp_x^*\bigl[\Breve\psi(Y^*_t)\bigr]\,.
\end{align*}
We use this together with \cref{ET2.5F,ET2.5G} to obtain \cref{ET2.5E}.

We write \cref{ABEx01} as
\begin{align}\label{ET2.5H}
0 &\;=\;   a^{ij} \partial_{ij}\Breve\psi + b^{i} \partial_{i}\Breve\psi
+\min_{u\in\Rd}\bigl[2\langle  a\,u, \grad\Breve\psi\rangle
+ \langle u, au\rangle\bigr] -\beta f
+\Lambda_\beta\nonumber\\[5pt]
&\;=\; a^{ij} \partial_{ij}\Breve\psi + b^{i} \partial_{i}\Breve\psi
-2\langle  \grad\Breve\psi, a \grad\Breve\psi\rangle
+ \langle \grad\Breve\psi, a \grad\Breve\psi\rangle -\beta f
+\Lambda_\beta\nonumber\\[5pt]
&\;=\;\widetilde\Lg^{\psi^*_\beta} \Breve\psi
+ \langle \grad\Breve\psi, a \grad\Breve\psi\rangle -\beta f
+\Lambda_\beta\,.
\end{align}
Let $F\df \langle \grad\Breve\psi, a \grad\Breve\psi\rangle -\beta f
=\langle \grad\psi^*_\beta, a \grad\psi^*_\beta\rangle -\beta f$.
Applying the It\^o--Krylov formula to \cref{ET2.5H}, we obtain
\begin{equation}\label{ET2.5I}
\widetilde\Exp_x^*\bigl[\Breve\psi(Y^*_{t\wedge\uptau_R})\bigr] - \Breve\psi(x)
+\widetilde\Exp_x^*\biggl[\int_0^{t\wedge\uptau_R} F(Y^*_s)\,\D{s}\biggr]
+\Lambda_\beta\widetilde\Exp_x^*\bigl[t\wedge\uptau_R\bigr]
\;=\;0\,.
\end{equation}
Letting $R\to\infty$ in \cref{ET2.5I}, using \cref{ET2.5E}, then
dividing by $t$ and letting $t\to\infty$, using \cref{ET2.5D} and
Birkhoff's ergodic theorem, we obtain
\begin{equation*}
\mu^*_\beta\bigl(\langle \grad\psi^*_\beta, a \grad\psi^*_\beta\rangle -\beta f\bigr)
+ \Lambda_\beta\;=\;0\,,
\end{equation*}
which is the assertion in part (iii).

\smallskip
\noindent\emph{Case 2.} Suppose $\Lambda_\beta\le 0$ and \cref{E2.1} is recurrent.
The case $\Lambda_\beta= 0$ is then trivial, since $\grad\psi^*_0=0$,
so we assume that $\Lambda_\beta< 0$.
Then \cref{E2.1} is exponentially ergodic by part (ii),
and thus $\Psi^*_\beta$ is bounded below in $\Rd$ by \cite[Lemma~2.1]{ari-anup}.
With $\psi^*=\psi^*_\beta=\log\Psi^*_\beta$, in analogy to \cref{ET2.5H} we have
\begin{equation}\label{ET2.5J}
\widetilde\Lg^{\psi^*} \psi^*
- \langle \grad\psi^*, a \grad\psi^*\rangle +\beta f -\Lambda_\beta\;=\;0\,.
\end{equation}
We claim that
\begin{equation}\label{ET2.5K}
\lim_{t\to\infty}\; \frac{1}{t}\,\widetilde\Exp_x^*\bigl[\psi^*(Y^*_t)\bigr]
\;=\;0\,,\quad\text{and}\quad
\lim_{R\to\infty}\;
\widetilde\Exp_x^*\bigl[\psi^*(Y^*_{t\wedge\uptau_R})\bigr]
\;=\;\widetilde\Exp_x^*\bigl[\psi^*(Y^*_{t})\bigr]\,,
\end{equation}
where as defined earlier, $\widetilde\Exp_x^*= \widetilde\Exp_x^{\psi^*}$,
and $Y^*$ denotes the ground state process.
Assuming \cref{ET2.5K} is true, we first apply
the It\^o--Krylov formula to \cref{ET2.5J} to obtain
the analogous equation to \cref{ET2.5I}, and then take limits
and use Birkhoff's ergodic theorem to establish \cref{ET2.5A}.

It remains to prove \cref{ET2.5K}.
Choose $\epsilon>0$ so that $\beta>\beta-\epsilon>\beta_c$, and
let $\Psi^*_{\beta-\epsilon}$ denote the ground state corresponding to
$\Lambda_{\beta-\epsilon}$.
We choose a ball $\sB$ such that 
\begin{equation}\label{ET2.5L}
\epsilon\, f (x) \;<\; \frac{1}{2}(\Lambda_\beta-\Lambda_{\beta-\epsilon})
\qquad \forall\,x\in\sB^c\,.
\end{equation}
Since $f$ vanishes at infinity, and $\Lambda_\beta>\Lambda_{\beta-\epsilon}$,
there exists a constant $\alpha>1$ and a ball also denoted as $\sB$,
such that
\begin{equation}\label{ET2.5M}
\alpha\bigl(\beta f(x)-\Lambda_\beta\bigr) \;<\; (\beta-\epsilon) f(x)
-\Lambda_{\beta-\epsilon}\qquad \forall\,x\in\sB^c\,.
\end{equation}
Since the ground state processes corresponding to the principal eigenvalues
$\Lambda_\beta$ and $\Lambda_{\beta-\epsilon}$ are
ergodic we have from \cref{L2.7} that
\begin{equation}\label{ET2.5N}
\begin{split}
\Psi^*_\beta(x) &\;=\; \Exp_x\Bigl[\E^{\int_0^{\uuptau}[\beta f(X_s)-\Lambda_\beta]\,
\D{s}}\, \Psi^*_\beta(X_{\uuptau})\,\Ind_{\{\uuptau<\infty\}}\Bigr]\,,\\
\Psi^*_{\beta-\epsilon}(x) &\;=\;
\Exp_x\Bigl[\E^{\int_0^{\uuptau}[(\beta-\epsilon)f(X_s)-\Lambda_{\beta-\epsilon}]\,
\D{s}}\, \Psi^*_{\beta-\epsilon}(X_{\uuptau})\,\Ind_{\{\uuptau<\infty\}}\Bigr]\,,
\end{split}
\end{equation}
for all $x\in\sB^c$ where $\uuptau=\uptau(\sB^c)$.
By \cref{EL2.5C}, the function
$\widetilde\Psi_\epsilon\df\frac{\Psi^*_{\beta-\epsilon}}{\Psi^*_\beta}$
satisfies
\begin{equation}\label{ET2.5O}
\widetilde\Lg^{\psi^*}\widetilde\Psi_\epsilon\;=\;
\bigl(\Lambda_{\beta-\epsilon}
-\Lambda_\beta + \epsilon f\bigr)\,\widetilde\Psi_\epsilon\,.
\end{equation}
Applying the Feynman--Kac formula to \cref{ET2.5O},
using \cref{ET2.5L}, it follows as in \cite[Lemma~2.1]{ari-anup}
that  $\inf_\Rd\,\widetilde\Psi_\epsilon =\min_{\Bar\sB}\,\widetilde\Psi_\epsilon$.
Thus $\widetilde\Psi_\epsilon$ is bounded away from $0$ on $\Rd$.

Let $\kappa\df\min_{\sB}\frac{\Psi^*_{\beta-\epsilon}}{(\Psi^*_\beta)^\alpha}$.
Then by \cref{ET2.5M,ET2.5N} we obtain
\begin{align*}
\widetilde\Psi_\epsilon(x) 
&\;\ge\; \frac{\kappa}{\Psi^*_\beta(x)}\,
\Exp_x\Bigl[\E^{\int_0^{\uuptau}\alpha [\beta f(X_s)-\Lambda_\beta]\,\D{s}}\,
\bigl(\Psi^*_\beta(X_{\uuptau})\bigr)^\alpha\,\,\Ind_{\{\uuptau<\infty\}}\Bigr]\\[5pt]
&\;\ge\;\frac{\kappa}{\Psi^*_\beta(x)}\,
\Bigl(\Exp_x\Bigl[\E^{\int_0^{\uuptau}[\beta f(X_s)-\Lambda_\beta]\, \D{s}}\,
\Psi^*_\beta(X_{\uuptau})\,\Ind_{\{\uuptau<\infty\}}\Bigr]\Bigr)^\alpha
\;\ge\; \kappa\,\bigl(\Psi^*_\beta(x)\bigr)^{\alpha-1}\qquad\forall\,x\in\sB^c\,.
\end{align*}
Therefore, for some constant $\kappa_1$ we have
\begin{equation}\label{ET2.5Q}
\psi^*\;\le\; \kappa_1 + \frac{1}{\alpha-1}\,
\log\widetilde\Psi_\epsilon\quad\text{on\ }\Rd\,.
\end{equation}

Let $\epsilon_\circ\df\frac{1}{2}(\Lambda_\beta-\Lambda_{\beta-\epsilon})$.
From \cref{ET2.5L} and exponential Foster--Lyapunov equation \cref{ET2.5O}
we deduce that \cref{ET2.5C} holds for $\widetilde\Psi_\epsilon$.
Thus the first equation in \cref{ET2.5K} follows directly from
\cref{ET2.5C,ET2.5Q} and the fact that $\inf_{\Rd}\psi^*>-\infty$,
while the second one follows by repeating the argument leading to \cref{ET2.5F}.
This completes the proof.
\qed\end{proof}

\begin{remark}
The assumption that \cref{E2.1} is recurrent
in the case that $\Lambda_\beta<0$ in \cref{T2.5}\,(iii) is equivalent
to the statement that $\lamstr(0)=0$.
Note that as shown in \cite[Theorem~2.1]{Ichihara-15},
unless $\lamstr(0)=0$, then \cref{ET2.5A} does not
hold if $\Lambda_\beta<0$.

If \cref{E2.1} is not recurrent, then it is possible that
$\beta_c<0$ and also that $\Lambda_\beta<0$ for $\beta\ge0$.
Consider a diffusion with $d=1$, $b(x)=2x$, and $\upsigma(x)=\sqrt{2}$.
Then, we have $\Lg\varphi=-\varphi$ for
$\varphi(x)=\frac{1}{2}e^{-x^2}$.
Thus $\Hat\Lambda_0\le -1$, where $\Hat\Lambda_0$ denotes the eigenvalue
in \cref{D2.1A} for $f=0$.
Thus  $\lamstr(0)\le-1$ by \cref{L2.2}\,(b).
Since the twisted process corresponding to $\varphi$ is exponentially ergodic,
we must have $\lamstr(0)=-1$ by \cref{T2.1}\,(c),
and thus $\varphi$ is the ground state.
\Cref{T2.1}\,(b) then asserts that
$\beta\mapsto\Lambda_\beta$ is strictly increasing at $\beta=0$.
Thus $\beta_c<0$.
Observe that the ground state diffusion is an Ornstein--Uhlenbeck process having a
Gaussian  stationary distribution of mean $0$ and variance $\nicefrac{1}{2}$.
An easy computation reveals that
$\mu^*\bigl(-\langle \grad\psi^*, a \grad\psi^*\rangle\bigr)=-2$
which is smaller than $\lamstr(0)$.
\end{remark}

The conclusion of \cref{T2.5}\,(iii) can be sharpened.
Consider the controlled diffusion
\begin{equation}\label{E-twist1}
\D{Z}_{t} \;=\;\bigl(b(Z_{t})+ 2 a(Z_{t}) v(Z_t)\bigr)\,\D{t}
+ \upsigma(Z_{t})\,\D{W}_{t}\,.
\end{equation}
Here $v\colon\Rd\to\Rd$ is a locally bounded Borel measurable map.
Let $\Ulb$ denote the class of such maps.
These are identified with the class of locally bounded stationary Markov controls.
Let $\Ulbs\subset\Ulb$ be the collection of those
$v$ under which the diffusion in \cref{E-twist1} is ergodic,
and denote by $\widehat\mu_v$ the associated invariant probability measure.
We let $\sA_v\df\Lg + 2\langle a v,\grad\rangle$,
and use the symbol $\widehat\Exp^v_x$ to denote the 
expectation operator associated with \cref{E-twist1}.

In order to simplify the
notation, we use the norm $\norm{v}_a \df \sqrt{\langle v,av\rangle}$.
For $v\in\Ulb$ we define
\begin{align*}
F_v(z) &\;\df\; \norm{v(z)}^2_{a(z)} - \beta f(z)\,,\\[5pt]
\sJ_x(v) &\;\df\; \limsup_{T\to\infty}\; \frac{1}{T}\,\widehat\Exp^v_x \biggl[\int_0^T
F_v(Z_s)\,\D s\biggr]\,,
\end{align*}
and $\overline\sJ_x \df \inf_{v\in\Ulb}\;\sJ_x(v)$.

\begin{theorem}\label{T2.6}
Assume that $f\in\cB^+_{\mathrm{o}}(\Rd)$ and $\beta>\beta_c$.
Then the following hold
\begin{enumerate}[(a)]
\item
If $\Lambda_\beta>0$, then we have
\begin{equation}\label{ET2.6A}
\overline\sJ_x \;=\;\sJ_x(\grad\psi^*_{\beta}) \;=\; -\Lambda_\beta
\qquad\forall\,x\in\Rd\,.
\end{equation}
In addition,
if $v\in\Ulb$ satisfies $\sJ_x(v)=\overline\sJ_x$, then $v=\grad\psi^*_\beta$ a.e.

\item
If $\Lambda_\beta\le0$ and \cref{E2.1} is recurrent then
\cref{ET2.6A} holds, and $v=\grad\psi^*_\beta$ is the a.e.\ unique
control in $\Ulbs$ which satisfies $\sJ_x(v)=\overline\sJ_x$.

\item
If $\Lambda_\beta<0$ and \cref{E2.1} is not recurrent,
then $\overline\sJ_x=0$ for all $x\in\Rd$.
\end{enumerate}
\end{theorem}

\begin{proof}
We start with part (a).
By \cref{T2.5}\,(iii), we have $\sJ_x(\grad\psi^*_\beta) = -\Lambda_\beta$ in both
of cases (a) and (b).
It suffices then to show that if $\sJ_x(v)\le -\Lambda_\beta$
for some $v\in\Ulb$, then $v=\grad\psi^*_\beta$ a.e.\ in $\Rd$.
Let such a control $v$ be given.
Then \cref{E-twist1} must be positive recurrent under $v$, for otherwise
we must have
$\sJ_x(v)\ge \limsup_{T\to\infty}\; \frac{1}{T}\,\widehat\Exp^v_x \bigl[\int_0^T
-\beta f(Z_s)\,\D s\bigr]\ge0$.
Therefore,
\begin{equation}\label{ET2.6B}
\sJ_x(v) \;=\; \int_\Rd F_{v}(z)\,\widehat\mu_v(\D z) \;\le\; -\Lambda_\beta \;<\;0\,,
\end{equation}
where $\widehat\mu_v$, as defined earlier, denotes the invariant probability measure
associated with $\sA_v\df\Lg + 2 \langle v, a\grad\rangle$.
Thus, $\sJ_x(v)$ does not depend on $x$, and dropping this dependence
in the notation we  let $\sJ(v)=\sJ_x(v)$.
Since $f\in\cB_{\mathrm{o}}(\Rd)$, it follows by
\cref{ET2.6B} and the definition of $F_{v}$ that there exists
a ball $\sB$ such that
\begin{equation}\label{ET2.6Bb}
\sJ(v)-F_{v}(x)\;\le\; -\frac{1}{2}\,\Lambda_\beta \;<\;0
\qquad\forall x\in\sB^c\,.
\end{equation}
By \cref{ET2.6Bb}, and
since $v$ is locally bounded,
and $F_{v}$ is integrable with respect to $\widehat\mu_v$,
we can assert the existence of a solution
$\Breve\varphi\in\Sobl^{2,d}(\Rd)$ to the Poisson equation
\begin{equation}\label{ET2.6C}
\Lg\Breve\varphi(x) + 2\bigl\langle a(x)v(x),\grad\Breve\varphi(x)\bigr\rangle + F_{v}(x)
\;=\; \sJ(v)\,,
\end{equation}
which is bounded below in $\Rd$ (see Lemma~3.7.8\,(d) in \cite{book}).
It follows by \cref{ET2.6C}
that $\Phi\df \E^{\varphi}$, $\varphi=-\breve\varphi$, satisfies
\begin{align}\label{ss1}
\Lg\Phi + \bigl(\beta f- \norm{v-\grad\varphi}^2_a\bigr)\,\Phi \;=\;-\sJ(v)\,\Phi\,.
\end{align}
This shows that $\bigl(\Phi, -\sJ(v)\bigr)$ is an eigenpair for
$\Lg^{\Breve F}$, with
$\Breve F\df \beta f- \norm{v-\grad\varphi}^2_a$.
The corresponding twisted process with
generator $\Tilde\Lg = \Lg + 2\langle a\grad\varphi,\grad\rangle$ then satisfies
\begin{equation}\label{ET2.6E}
\Tilde\Lg\Breve\varphi
+ \norm{v-\grad\varphi}^2_a
 - \beta f \;=\; \sJ(v)\,.
\end{equation}
Since $\Breve\varphi$ is bounded below in $\Rd$ and $\sJ(v)<0$,
\cref{ET2.6E} shows that the twisted process is positive recurrent. 
We claim that $-\sJ(v)$ is the principal eigenvalue of $\Lg^{\Breve F}$.
Indeed, if $\lamstr(\Breve F)<-\sJ(v)$ then by the proof of \cref{L2.3}
and for any 
$g\in\Cc^+_{\mathrm{c}}(\Rd)$ we obtain
\begin{align*}
\Phi(x)\, \widetilde\Exp^{\varphi}_x\bigl[g(Y_T)\Phi^{-1}(Y_T)
\Ind_{\{T\le \uptau_n\}}\bigr] &\;= \;
\Psi^*(x)\, \widetilde\Exp^{\psi^*}_x
\Bigl[e^{[\lamstr(\Breve F)+\sJ(v)] T}g(Y^*_T)(\Psi^*)^{-1}(Y^*_T)
\Ind_{\{T\le \uptau_n\}}\Bigr]\\[3pt]
&\;\le\; \Psi^*(x)\,\Bigl(\sup_{\Rd}\frac{g}{\Psi^*}\Bigr)
e^{[\lamstr(\Breve F)+\sJ(v)] T}
\end{align*}
for all sufficiently large $n$,
where $\uptau_n$ denotes the first exit time from $B_n$.
By first letting $n\to\infty$, and then integrating with respect to $T$ we obtain 
\begin{equation*}
\int_0^\infty \widetilde\Exp^{\varphi}\bigl[g(Y_t)\Phi^{-1}(Y_t)\bigr]\, \D{t}
\;<\; \infty\,.
\end{equation*}
But this contradicts the positive recurrence
of the twisted process corresponding to $\Tilde\Lg$.
Therefore, $-\sJ(v)$ must be the principal eigenvalue of $\Lg^{\Breve F}$,
which implies that
\begin{equation}\label{ET2.6F}
-\sJ(v) \;=\;
\lamstr\bigl(\beta f- \norm{v-\grad\varphi}^2_a\bigr)
\;\le\; \Lambda_\beta\,.
\end{equation}
Thus we have shown that $\sJ_x(v)=\sJ(v)= -\Lambda_\beta$.
The strict monotonicity of $\lamstr$ at $\beta f$ together
with \cref{ET2.6F} imply that $v=\grad\varphi$ a.e.\ in $\Rd$.
In turn, \cref{ss1} and the uniqueness of the ground state
imply that $\Phi=\Psi^*_\beta$, up to a multiplication by a positive constant. 
Therefore, we have $v=\grad\psi^*_\beta$ a.e.\ in $\Rd$,
and this completes the proof of part (a).

We continue with part (b).
The case $\Lambda_\beta=0$ is trivial, so assume that $\Lambda_\beta<0$.
Then \cref{E2.1} is exponentially ergodic by \cref{T2.5}(ii).
Thus $\Psi^*_\beta$ is bounded away from $0$ in $\Rd$ by \cite[Lemma~2.1]{ari-anup}.
Let $v\in\Ulb$, and $\Breve\psi= -\psi^*_\beta$.
We have
\begin{equation}\label{ET2.6G}
\Lg\Breve\psi + 2 \langle v, a\grad\Breve\psi\rangle
-\norm{v+\grad\Breve\psi}^2_a
 +F_v \;=\; -\Lambda_\beta\,.
\end{equation}
Since $\Breve\psi$ is bounded above in $\Rd$, it follows from \cref{ET2.6G}
by a standard argument that
\begin{equation*}
\limsup_{T\to\infty}\;
\frac{1}{T}\,\widehat\Exp^{v}_x \biggl[\int_0^T
F_{v}(Z_s)\,\D s\biggr]\;\ge\;-\Lambda_\beta\,.
\end{equation*}

We next show uniqueness in $\Ulbs$ of the optimal control $\grad\psi^*_\beta$.
Let $v\in\Ulbs$ and suppose $\sJ_x(v)= -\Lambda_\beta$.
In other words, $\widehat\mu_v(F_v)=-\Lambda_\beta$.
By the It\^o--Krylov formula and Fatou's lemma
and since $\Breve\psi$ is bounded above, we obtain from \cref{ET2.6G} that
\begin{equation}\label{ET2.6H}
\widehat\Exp^v_x\bigl[\Breve\psi(Z_t)\bigr]-\Breve\psi(x)
- \widehat\Exp^v_x\biggl[\int_0^t G_v(Z_s)\, \D{s}\biggr] 
+ \widehat\Exp^v_x\biggl[\int_0^t F_v(Z_s)\, \D{s}\biggr]
\;\ge\; -t\Lambda_\beta\,,
\end{equation}
with
$$G_v(z)\;\df\;\bigl\langle \bigl(v(z)+\grad\Breve\psi(z)\bigr),
a(z) \bigl(v(z)+\grad\Breve\psi(z)\bigr)\bigr\rangle\,.$$
Dividing \cref{ET2.6H} by $t$ and taking limits as $t\to\infty$, we obtain
$- \widehat\mu_v(G_v)+\sJ_x(v)\ge -\Lambda_\beta$.
Therefore,
$\widehat\mu_v(G_v)=0$, since $G_v$ is nonnegative.
Thus $G_v=0$ a.e.\ in $\Rd$, and since $\widehat\mu_v$ has a density, this implies that
$v=-\grad\Breve\psi=\grad\psi^*_\beta$ a.e.\ in $\Rd$.

We now turn to part (c).
It is evident that under the control $v=0$, since the diffusion in
\cref{E-twist1} is transient and $f$ vanishes at infinity, we have
$\lim_{t\to\infty}\,\frac{1}{t}\,\widehat\Exp^v_x\bigl[F_0(Z_t)\bigr]=0$.
It is also clear that under any control $v\in\Ulb\setminus\Ulbs$
we have $\lim_{t\to\infty}\,\frac{1}{t}\,\widehat\Exp^v_x\bigl[F_v(Z_t)\bigr]\ge0$.
Suppose that under some $v\in\Ulbs$, we have
$$\liminf_{t\to\infty}\;\frac{1}{t}\;\widehat\Exp^v_x\bigl[F_v(Z_t)\bigr]
\;=\;\sJ_x(v)\;<\;0\,.$$
Then there exists a solution $\Breve\phi$ to the Poisson equation \cref{ET2.6C}
which is bounded below in $\Rd$.
Thus following the proof of Case~1 in part (a)
we obtain by \cref{ET2.6F} that $\sJ_x(v)\ge -\Lambda_\beta$ which is a contradiction.
We have therefore shown that $\sJ_x(v)\ge0$ for all $v\in\Ulb$, which
implies that $0$ is the optimal value in the class of controls $\Ulb$.
\qed\end{proof}

\begin{remark}\label{R2.2}
The assumption that $f$ is nonnegative can be weakened to $f\in\cB_{\mathrm{o}}(\Rd)$.
From the proof of \cref{T2.2} we note that if
$\lamstr(f+h)<\lamstr(f)$ for some $h\in\cB_{\mathrm{o}}(\Rd)$,
then the ground state diffusion corresponding to
$\lamstr(f)$ is geometrically ergodic.
Moreover, due to \cite[Proposition~2.3\,(vii)]{Berestycki-15} the function 
$\beta\mapsto\lamstr(\beta f)$ is convex for every $f\in \cB_{\mathrm{o}}(\Rd)$.
Instead of the critical value $\beta_c$, we can define
a critical value $\lambda_c$ by
$\lambda_c\df \inf_{\beta\in\RR}\,\Lambda_\beta$.
Then if we replace the condition $\beta>\beta_c$ by $\Lambda_\beta>\lambda_c$
as done in \cite{Ichihara-15}, it is evident that
$\lamstr(\beta f)$ is strictly monotone at $\beta f$
and the results in \cref{T2.5}\,(iii)
and \cref{T2.6} still hold, provided $\Lambda_\beta\ne 0$, and the proofs are the same.
\end{remark}

The results in \cref{T2.6}\,(b) can be also stated for nonstationary controls.
Consider the controlled diffusion
\begin{equation}\label{E-twist2}
\D{Z}_{t} \;=\;\bigl(b(Z_{t})+ 2 a(Z_{t}) U_t\bigr)\,\D{t}
+ \upsigma(Z_{t})\,\D{W}_{t}\,.
\end{equation}
Here $U=\{U_t\}$ is an $\Rd-$valued control process which is jointly measurable
in $(t,\omega)\in [0,\infty)\times\Omega$, and is \emph{nonanticipative}: for
$t > s$, $W_{t} - W_{s}$ is independent of
\begin{equation*}
\sF_s\;\df\; \text{the completion of\ }
\cap_{y>s}
\sigma(X_{0}, W_{r}, U_{r}\,\colon r\le y) \text{\ relative to\ } (\sF,\Prob)\,.
\end{equation*}
Let $f\in\cB_{\mathrm{o}}(\Rd)$, not necessarily nonnegative.
Assume that $\Lambda_\beta>\lambda_c$,
$\Lambda_\beta\le0$, and \cref{E2.1} is recurrent (see \cref{R2.2}).
Suppose that under $U$, the diffusion in \cref{E-twist2} has a unique
weak solution.
We claim that
\begin{equation*}
\sJ_x(U) \;\df\; \limsup_{T\to\infty}\; \frac{1}{T}\;\widehat\Exp^U_x \biggl[\int_0^T
\bigl[\langle U_s,a(Z_s)U_s\rangle
 - \beta f(Z_s)\bigr]\,\D s\biggr] \;\ge\; - \Lambda_\beta\,.
\end{equation*}
We can prove this as follows.
By \cref{ET2.5H} we obtain
\begin{equation*}
\Lg\Breve\psi(z) + 2 \langle u, a\grad\Breve\psi\rangle
 +F_u(z) \;\ge\; -\Lambda_\beta
\end{equation*}
for all $u, z\in\Rd$, and we apply the It\^o--Krylov formula and Fatou's lemma
(using the fact that $\Breve\psi$ is bounded above)
with $u=U_t$ to obtain analogously to \cref{ET2.6H} that
\begin{equation}\label{ET2.6Hx}
\widehat\Exp^U_x\bigl[\Breve\psi(Z_t)\bigr]-\Breve\psi(x)
  + \widehat\Exp^U_x\biggl[\int_0^t F_{U_s}(Z_s)\, \D{s}\biggr]
\;\ge\; -t\Lambda_\beta\,.
\end{equation}
Dividing \cref{ET2.6Hx} by $t$ and letting $t\to\infty$,
we obtain $$\liminf_{t\to\infty}\;
\frac{1}{t}\;\widehat\Exp^U_x\biggl[\int_0^t F_{U_s}(Z_s)\, \D{s}\biggr]
\;\ge\; -\Lambda_\beta\,,$$
thus proving the claim.

\subsubsection{Strong duality}
The optimality result in \cref{T2.6} can be strengthened.
Consider the class of
\emph{infinitesimal ergodic occupation measures}, i.e.,
measures $\uppi\in\cP(\Rd\times\Rd)$ which satisfy
\begin{equation}\label{E-eom}
\int_{\Rd\times\Rd} \sA_u g(x)\,\uppi(\D{x},\D{u})
\;=\;0 \qquad \forall\,g\in\Cc_{\mathrm{c}}^{\infty}(\Rd)\,,
\end{equation}
with $\sA_u\df \Lg + \langle 2au,\grad\rangle$.
Disintegrate these as 
$\uppi(\D{x},\D{u}) = \eta_v(\D{x})\,v(\D{u}\,|\, x)$,
and denote this disintegration as $\uppi=\eta_v\circledast v$.
Let $\Hat{v}(x)=\int u\, v(\D{u}\,|\, x)$.
Since $\int\abs{u}^2\eta(\D{x})\,v(\D{u}\,|\, x)
\ge \int\abs{\Hat{v}(x)}^2\eta(\D{x})$,
and $\eta(\D x) \delta_{\Hat{v}(x)}(\D u)$ is also an ergodic
occupation measure, it is enough to consider the class
of infinitesimal ergodic occupation measures
$\uppi$ that correspond to a precise control $v$, i.e.,
a Borel measurable map from $\Rd$ to $\Rd$.
We denote this class by $\sM$.
Thus for $\uppi=\eta_v\circledast v\in\sM$,
\cref{E-eom} takes the form
$\int_{\Rd} \sA_v g(x)\,\eta_v(\D{x})=0$.
Note that $v$ is not necessarily locally bounded,
so this class of controls is, in general, larger than $\Ulbs$.

In \cref{T2.7} below we use the following simple assertions
which are stated as remarks.

\begin{remark}\label{R2.3}
If $\eta_v$ has density
$\rho_v\in\Lpl^{\nicefrac{d}{(d-1)}}(\Rd)$, and $v\in\Lp^{2}(\Rd;\eta_v)$,
then
\begin{equation*}
\int_\Rd \sA_v g(x)\,\eta_v(\D x)\;=\;0\qquad\forall\,g\in\Sobl^{2,d}(\Rd)\cap
\Cc_{\mathrm{c}}(\Rd)\,.
\end{equation*}
This can be proved as follows.
We mollify $g$ with a smooth mollifier family $\{\chi_r,\, r>0\}$, so that
\cref{E-eom} can applied to  the function $g*\chi_r$, where
`$*$' denotes convolution.
Then we separate terms, and applying the H\"older inequality on
$\babs{\int (\Lg g- \Lg(g*\chi_r))\rho_v}$, and using the
convergence of $\Lg(g*\chi_r)$ to $\Lg g$ in $\Lpl^d(\Rd)$, we deduce that
this term tends to $0$ as $r\searrow0$.
Similarly, we apply the H\"older inequality in the form
\begin{equation*}
\babss{\int \langle 2av,\grad (g-g*\chi_r)\rangle\,\rho_v}^2\;\le\;
\int \abs{2av}^2\rho_v \int \abs{\grad (g-g*\chi_r)}^2\rho_v\,.
\end{equation*}
Then the first integral on the right hand side is bounded, and
the second integral vanishes as $r\searrow0$ since
$g*\chi_r$ converges to $g$ uniformly on compact sets.
\end{remark}

\begin{remark}\label{R2.4}
Suppose that the drift $b$ in \cref{E2.1} has at most affine growth.
It is then well known that the map
$x\mapsto\Exp_x[\uptau(\sB^c)]$ is inf-compact
for any open ball $\sB$, provided of course that
\cref{E2.1} is positive recurrent.
This fact together with the stochastic representation in \cref{ET2.1A}
and Jensen's inequality, imply that if $f\in\cB_{\mathrm{o}}(\Rd)$,
$\Lambda_\beta<0$, and \cref{E2.1} is recurrent, then
the ground state $\Psi^*_\beta$ is inf-compact,
and this of course renders \cref{E2.1} positive recurrent.
An analogous argument using the ground state diffusion shows
that, if $b+a\grad\Psi^*_\beta$ has at most affine growth,
and $\Lambda_\beta>\max\{0,\lambda_c\}$ (see \cref{R2.2}),
then $\widetilde\Psi=(\Psi^*_\beta)^{-1}$
is inf-compact.
\end{remark}

The theorem that follows shows that there is no optimality gap
between the primal problem which consists of minimizing
$\int F_u(x) \uppi(\D x,\D u)$ subject to the constraint
\cref{E-eom}, and the dual problem which amounts to a maximization over
subsolutions of the HJB equation, as described in \cref{S1}.
This theorem is stated for  $f\in\cB_{\mathrm{o}}(\Rd)$ which is not
necessarily nonnegative as discussed in \cref{R2.2}.

\begin{theorem}\label{T2.7}
Assume that $f\in\cB_{\mathrm{o}}(\Rd)$, $\Lambda_\beta>\lambda_c$,
 and that one of the following conditions holds.
\begin{enumerate}[(i)]
\item
$\Lambda_\beta>0$, the coefficients $a$ and $b$ are bounded,
and $a$ is uniformly strictly elliptic.

\item
$\Lambda_\beta < 0$, \cref{E2.1} is recurrent, and $b$ has at most
affine growth.
\end{enumerate}
Then any $\uppi=\eta_v\circledast v\in\sM$,  such that $\int_\Rd F_v\,\D\eta_v<\infty$, 
satisfies
\begin{equation}\label{ET2.7A}
\int_\Rd F_v\,\D\eta_v
\;=\; - \Lambda_\beta
+ \int_\Rd \norm{v-\grad\psi^*_\beta}^2_a\,\D\eta_v\,.
\end{equation}
In addition, if $\uppi=\eta_v\circledast v\in\sM$ is optimal, i.e., if it satisfies
$\int_\Rd F_v\,\D\eta_v = -\Lambda_\beta$,
then $v=\grad\psi^*_\beta$ a.e.\ in $\Rd$ and $\eta_v=\mu^*_\beta$.
\end{theorem}

\begin{proof}
We first consider case (i).
Since $a$, $b$, and $f$ are bounded,
it follows that $\grad\psi^*_\beta$ is bounded by \cite[Lemma~3.3]{ari-anup}.
Then $-\psi^*_\beta$ is inf-compact by \cref{R2.4}.
Recall that $\sA_v=\Lg+2\langle  a\,v, \grad\rangle$.
We have
\begin{equation}\label{ET2.7B}
-\sA_v\psi^*_\beta -\norm{v-\grad\psi^*_\beta}^2_a
 +F_v \;=\; -\Lambda_\beta\,,
\end{equation}
Let $\chi$ be a convex $\Cc^2(\mathbb{R})$ function such that
$\chi(x)= x$ for $x\ge0$, $\chi(x) = -1$ for $x\le -1$,
and $\chi'$, $\chi''$ are positive on $(-1,0)$.
Define $\chi_R(x) \df -R + \chi(x+R)$, $R>0$.
Then we have from \cref{ET2.7B} that
\begin{equation}\label{ET2.7C}
-\sA_v\chi_R(\psi^*_\beta) + \chi_R''(\psi^*_\beta)\norm{\grad\psi^*_\beta}^2_a
-\chi_R'(\psi^*_\beta)\norm{v-\grad\psi^*_\beta}^2_a
 + \chi_R'(\psi^*_\beta) F_v \;=\; -\chi_R'(\psi^*_\beta)\Lambda_\beta\,.
\end{equation}
Since $\int \sA_v g\,\D\eta_v=$ for all $g\in\Cc^\infty_{\mathrm c}(\Rd)$,
an application of \cite[Theorem~2.1]{Bogachev-01} shows that $\eta_v$
has a density $\rho_v\in\Lpl^{\nicefrac{d}{(d-1)}}(\Rd)$.
Note that this does not require $a$ or $b$ to be bounded.
Therefore, since $\chi_R(\psi^*_\beta)+R+1$ has compact support,
we have $\int_\Rd \sA_v\chi_R(\psi^*_\beta)\,\eta_v(\D x)=0$ by \cref{R2.3}.
Thus letting $R\to\infty$ in \cref{ET2.7C}, using
monotone convergence, we obtain \cref{ET2.7A}. 

We next show uniqueness.
Let $\uppi=\eta_v\circledast v\in\sM$ be optimal, and
$\uppi_* = \eta_* \circledast v_*$ denote the ergodic occupation measure
corresponding to $v_*=\grad\psi^*_\beta$.
Here, $\eta_*=\mu^*_\beta$.
Let $\rho_*$ denote the density of $\eta_*$.
Define $\Bar\eta\df\frac{1}{2}(\eta_v+\eta_*)$ and
$\Bar{v} \df \zeta_v v + \zeta_* v_*$, with $\zeta_v$ and $\zeta_*$ given by
$\zeta_v\df\frac{\rho_v}{\rho_v+\rho_*}$ and $\zeta_*\df\frac{\rho_*}{\rho_v+\rho_*}$,
respectively.
It is straightforward to verify, using the fact that the drift is affine in the
control, that $\Bar\uppi=\Bar\eta\circledast\Bar{v}$ is in $\sM$.

By optimality, we have
\begin{align}\label{ET2.7D}
0 &\;\ge \; \int_\Rd F_{v}(x)\,\eta_v(\D{x}) + \int_\Rd F_{v_*}(x)\,\eta_*(\D{x})
-2 \int_\Rd F_{\Bar{v}}(x)\,\Bar\eta(\D{x}) \nonumber\\[5pt]
&\;=\;2\,\int_\Rd \Bigl(\zeta_v(x)\,\norm{v(x)}^2_{a(x)}
+ \zeta_*(x)\,\norm{v_*(x)}^2_{a(x)} -\bnorm{\zeta_v(x) v(x)
+ \zeta_*(x) v_*(x)}^2_{a(x)}\Bigr)\,\Bar\eta(\D{x}) \nonumber\\[5pt]
&\;=\;\int_{\Rd}\frac{\rho_v(x)\,\rho_*(x)}{\rho_v(x)+\rho_*(x)}\,
\norm{v(x)-v_*(x)}^2_{a(x)}\, \D{x}\,.
\end{align}
Since $\rho_*$ is strictly positive, \cref{ET2.7D} implies
that $\rho_v\,\abs{v-v_*}=0$ a.e.\ in $\Rd$, and thus
$v=v_*$ on the support of $\eta_v$.
It is clear that if $v$ is modified outside the support of $\eta_v$, then
the modified $\eta_v\circledast v$ is also an infinitesimal ergodic occupation
measure.
Therefore $\eta_v\circledast v_*\in\sM$.
The uniqueness of the invariant measure of the diffusion with
generator $\sA_{v_*}$ then implies that $\eta_v=\eta_*$,
which in turn implies that $v=\grad\psi^*_\beta$ a.e.\ in $\Rd$.

We now turn to case (ii).
By \cref{R2.4}, $\psi^*_\beta$ is inf-compact.
Also, as shown in case (i), $\eta_v$
has a density $\rho_v\in\Lpl^{\nicefrac{d}{(d-1)}}(\Rd)$.
We write \cref{ET2.6G} as
\begin{equation}\label{ET2.7E}
\sA_v\Breve\psi -\norm{v+\grad\Breve\psi}^2_a +F_v \;=\; -\Lambda_\beta\,,
\end{equation}
with $\Breve\psi = - \psi^*_\beta$.
Then we have from \cref{ET2.7E} that
\begin{equation}\label{ET2.7F}
\sA_v\chi_R(\Breve\psi) - \chi_R''(\Breve\psi) \norm{\grad\Breve\psi}^2_a
-\chi_R'(\Breve\psi)\norm{v+\grad\Breve\psi}^2_a
 + \chi_R'(\Breve\psi) F_v \;=\; -\chi_R'(\Breve\psi)\Lambda_\beta\,.
\end{equation}
Using the inequality
$\norm{v+\grad\Breve\psi}^2_a \le 2\norm{v}^2_a +2\norm{\grad\Breve\psi}^2_a$,
then integrating
\cref{ET2.7F} with respect to $\eta_v$, and rearranging
terms we obtain
\begin{equation}\label{ET2.7K}
\int \bigl(\chi_R''(\Breve\psi)+2\chi_R'(\Breve\psi)\bigr)\,
\norm{\grad\Breve\psi}^2_a\,\D\eta_v
\;\le\; \int\chi_R'(\Breve\psi) (F_v+2\beta f)\,\D\eta_v  -\Lambda_\beta
 \int\chi_R'(\Breve\psi)\,\D\eta_v\,.
\end{equation}
Thus letting $R\to\infty$ in \cref{ET2.7K}, using
monotone convergence, we obtain the energy inequality 
\begin{equation}\label{ET2.7L}
2\int \norm{\grad\Breve\psi}^2_a\,\D\eta_v
\;\le\; \int  (F_v+2\beta f)\,\D\eta_v- \Lambda_\beta\;<\;\infty\,.
\end{equation}
Then \cref{ET2.7A} follows by letting $R\to\infty$ in \cref{ET2.7K}, using
again monotone convergence and \cref{ET2.7L}.
Uniqueness follows as in case (i).
This completes the proof.
\qed\end{proof}

\begin{remark}
The proof of \cref{T2.7} provides a general recipe to prove the
lack of an optimality gap in ergodic control problems.
Note that the  model in \cite{Ichihara-13b} is such that $\grad\psi^*_\beta$ is bounded,
and $a$ is also bounded.
Therefore,
$$\int \chi_R''(\Breve\psi)\norm{\grad\Breve\psi}^2_a\,\D\eta_v\;
\xrightarrow[R\to\infty]{}\;0\,,$$
and
the proof of \cref{T2.7} goes through even for the more general Hamiltonian
$H(x,p)$ in \cite{Ichihara-13b}.
\end{remark}

\begin{remark}
If $\Lambda_\beta>\lambda_c$,
and under some nonanticipative control
$U$ the diffusion \cref{E-twist2} has
a unique weak solution, it was shown in the discussion following
\cref{R2.2} that $\sJ_x(U)\ge -\Lambda_\beta$, provided
$\Lambda_\beta\le0$ and \cref{E2.1} is recurrent.
The same conclusion can be drawn if $\Lambda_\beta>0$ and under the
hypotheses of \cref{T2.7}.
Define the set of \emph{mean empirical measures}
$\bigl\{\xi^U_{x,t}\,,\; t\ge0\}$
of \cref{E-twist2} under the control $U$ by
\begin{equation*}
\int_{\Rd\times\Rd} h(x,u)\,\xi^U_{x,t}(\D{x},\D{u})
\;=\; \frac{1}{t}\int_0^t \widehat\Exp_x^U \bigl[h(Z_t,U_t)\bigr]\,\D{t}
\qquad \forall h\in\Cc_b(\Rd\times\Rd)\,.
\end{equation*}
If $\Lambda_\beta>0$, then $F_u(x)-\Lambda_\beta$ is bounded away from zero
for all $x$ outside some compact set, and one can follow the
arguments in the proof of \cite[Lemma~3.4.6]{book} to show that
every limit point in
$\cP(\overline{\Rd\times\Rd})$ (the set of Borel probability measures
on the one-point compactification of $\Rd\times\Rd$)
of a sequence of mean empirical measures $\{\xi^{U_n}_{x,{t_n}}\,,\;n\in\NN\}$ as
$t_n\to\infty$ takes the form $\delta \uppi + (1-\delta)\uppi_\infty$, where
$\uppi$ is an infinitesimal ergodic occupation measure and
$\uppi_\infty(\{\infty\})=1$.
Using this property, one can show, by following the argument in the proof
of \cite[Theorem~3.4.7]{book}, that
if $\sJ_x(U)\le -\Lambda_\beta$, then the mean empirical measures
are necessarily tight in $\cP(\Rd\times\Rd)$ and $\delta=1$ in
this decomposition.
This of course implies that $\sJ_x(U)= -\Lambda_\beta$.
This argument establishes optimality over the largest possible class of
controls $U$.
\end{remark}

\subsubsection{Differentiability of \texorpdfstring{$\Lambda_\beta$}{Lb}}

Differentiability of the map $\beta\mapsto \Lambda_\beta$ for
all $\beta>\beta_c$ is established in \cite[Proposition~5.4]{Ichihara-15}
under the hypothesis that the coefficients $a$, $b$, and $f$ are
Lipschitz continuous and bounded in $\Rd$, but for a more general class
of Hamiltonians (see (A1)--(A3) in \cite{Ichihara-15}).
These assumptions are used to show that $\grad \psi^*$ is bounded in $\Rd$,
and this is utilized in the proofs.

In the next theorem we demonstrate this  differentiability result
for the model in this paper which assumes only measurable $b$ and $f$,
in which case it is not possible, in general, to obtain gradient estimates
and follow the approach in \cite{Ichihara-13b,Ichihara-15,Barles-16}.
The first assertion in this theorem should be compared to
\cite[Proposition~5.4]{Ichihara-15}.
Recall the definition
 $\widetilde\Psi_\epsilon=\frac{\Psi^*_{\beta-\epsilon}}{\Psi^*_\beta}$
after \cref{ET2.5N},
and let $\widetilde\psi_\epsilon=\log \widetilde\Psi_\epsilon$.

\begin{theorem}\label{T2.8}
Suppose $f\in\cB^+_{\mathrm{o}}(\Rd)$, and that $\beta>\beta_c$.
Then for all $\epsilon>0$ such that $\beta-\epsilon>\beta_c$, we have
\begin{equation}\label{ET2.8A}
\epsilon\,
\frac{\mu^*_\beta(f\,\widetilde\Psi_\epsilon)}{\mu^*_\beta(\widetilde\Psi_\epsilon)}
\;\le\; \Lambda_\beta-\Lambda_{\beta-\epsilon}\;=\;
\mu^*_\beta\bigl(\epsilon f -
\norm{\grad\widetilde\psi_\epsilon}^2_a\bigr)\,.
\end{equation}
In addition, we have 
\begin{equation}\label{ET2.8B}
\frac{\D\,\Lambda_{\beta}}{\D\beta} \;=\;\mu^*_\beta(f)\,.
\end{equation}
\end{theorem}

\begin{proof}
Fix some $\epsilon_1>0$ such that $\beta-2\varepsilon_1>\beta_c$,
and consider \cref{ET2.5O}.
As argued in the proof of \cref{T2.5}, the function
$\widetilde\Psi_\epsilon$ is bounded away from $0$ on $\Rd$
for all $\epsilon\in(0,\varepsilon_1]$.
We recall the notation $\widetilde\Exp^{\psi^*_\beta}[\,\cdot\,]
=\widetilde\Exp^{*}[\,\cdot\,]$.
Applying the It\^o--Krylov formula and Fatou's lemma to \cref{ET2.5O} we obtain
\begin{equation*}
\frac{1}{T}\;\widetilde\Exp_x^{*}\biggl[
\int_0^T \bigl(\Lambda_{\beta-\epsilon}
-\Lambda_\beta + \epsilon f(Y_t^*)\bigr)\,\widetilde\Psi_\epsilon(Y_t^*)\,\D t \biggr]
\;\ge\; 0\,,
\end{equation*}
from which the left hand side inequality of \cref{ET2.8A} follows by
an application of Birkhoff's ergodic theorem.
Also the analogous estimate to
\cref{ET2.5C} holds for $\widetilde\Psi_\epsilon$, which implies that
\begin{equation}\label{ET2.8C}
\lim_{t\to\infty}\; \frac{1}{t}\,
\widetilde\Exp_x^*\bigl[\widetilde\Psi_\epsilon(Y^*_t)\bigr]\;=\;0
\qquad\forall\,\epsilon\in(0,\varepsilon_1]\,.
\end{equation}
The second equality in \cref{ET2.8A} then follows by first using
the technique in the proof of \cref{T2.5} and \cref{ET2.8C} to establish
\cref{ET2.5E} for $\widetilde\psi_\epsilon$, $\epsilon\in(0,\varepsilon_1)$,
and then applying the It\^o--Krylov formula to
the log-transformed equation corresponding to \cref{ET2.5O} as in \cref{ET2.5I},
and taking limits at $t\to\infty$.

Using the convexity of $\beta\mapsto\Lambda_\beta$,
we write \cref{ET2.8A} as
\begin{equation}\label{ET2.8D}
\frac{\mu^*_\beta(f\,\widetilde\Psi_\epsilon)}{\mu^*_\beta(\widetilde\Psi_\epsilon)}
\;\le\; \frac{\Lambda_\beta-\Lambda_{\beta-\epsilon}}{\epsilon}
\;\le\; \frac{\Lambda_{\beta+\epsilon}-\Lambda_\beta}{\epsilon}
\;\le\; \mu^*_{\beta+\epsilon}(f)\,.
\end{equation}

Fix an open ball $\sB\subset\Rd$, such that
\begin{equation}\label{ET2.8E}
\Lambda_{\beta-2\epsilon_1} - \Lambda_{\beta-\epsilon}
+(2\epsilon_1-\epsilon) f(x)\;\le\;-\delta\;<\;0 
\qquad\forall\,\epsilon\in[-\epsilon_1,\epsilon_1]\,,\ \ \forall\,x\in\sB^c\,.
\end{equation}
This is clearly possible since $\epsilon\mapsto\Lambda_{\beta-\epsilon}$
is nonincreasing, $\Lambda_{\beta-2\epsilon_1}<\Lambda_{\beta-\epsilon_1}$,
 and $f$ vanishes at infinity.
Let $\uuptau=\uptau(\sB^c)$.
Since the ground state process corresponding to $\Lambda_{\beta-\epsilon}$
is exponentially ergodic for $\epsilon<\epsilon_1$ by \cref{T2.5},
we have
\begin{equation}\label{ET2.8F}
\Psi^*_{\beta-\epsilon}(x) \;=\;\Exp_x\Bigl[\E^{\int_0^{\uuptau}
[(\beta-\epsilon) f(X_s)-\Lambda_{\beta-\epsilon}]\, \D{s}}\,
\Psi^*_{\beta-\epsilon}(X_{\uuptau})\,\Ind_{\{\uuptau<\infty\}}\Bigr]\qquad
\forall\,\epsilon\in[-\epsilon_1,\epsilon_1]
\end{equation}
by \cref{L2.7}.
Since $\Psi^*_{\beta-\epsilon}$ and its inverse are bounded on $\sB$, uniformly
in $\epsilon\in[-\epsilon_1,2\epsilon_1]$, it follows from \cref{ET2.8E,ET2.8F}
that there exists $\kappa$ such that
$\Psi^{*}_{\beta-\epsilon}\le \kappa \Psi^*_{\beta-2\epsilon_1}$
for all $\epsilon\in[-\epsilon_1,\epsilon_1]$.
Therefore, since the collection $\bigl\{\Psi^*_{\beta-\epsilon}\,,\,
\epsilon\in[-\epsilon_1,\epsilon_1]\bigr\}$, is bounded in 
$\Ccl^{1,\alpha}(\sB)$, $\alpha>0$, we can use \cref{ET2.8F} and the dominated
convergence theorem to conclude that $\widetilde\Psi_\epsilon\to1$
as $\epsilon\searrow0$.
Thus, one more application of the dominated convergence theorem shows that
$\mu^*_\beta(f\,\widetilde\Psi_\epsilon)\to \mu^*_\beta(f)$ and
$\mu^*_\beta(\widetilde\Psi_\epsilon)\to 1$ as $\epsilon\searrow0$.
This shows that
\begin{equation}\label{ET2.8G}
\lim_{\epsilon\searrow0}\;
\frac{\mu^*_\beta(f\,\widetilde\Psi_\epsilon)}{\mu^*_\beta(\widetilde\Psi_\epsilon)}
\;=\; \mu^*_\beta(f)\,.
\end{equation}

We next study the term $\mu^*_{\beta+\epsilon}(f)$.
Let $\widetilde\Exp_x^{*,\epsilon}$ denote the expectation operator for the ground
state diffusion corresponding to $\Lambda_{\beta+\epsilon}$.
Since
\begin{equation*}
\widetilde\Lg^{\psi^*_{\beta+\epsilon}}\tfrac{\Psi^*_{\beta-2\epsilon_1}}
{\Psi^*_{\beta+\epsilon}}\;=\;
\bigl(\Lambda_{\beta-2\epsilon_1}
-\Lambda_{\beta+\epsilon} - (2\epsilon_1+\epsilon)f\bigr)\,
\tfrac{\Psi^*_{\beta-2\epsilon_1}}{\Psi^*_{\beta+\epsilon}}\,,
\end{equation*}
it follows by an estimate similar to \cref{ET2.8E} that
$\widetilde\Exp_x^{*,\epsilon} [\E^{\kappa\uuptau}]\;\le\;
\tfrac{\Psi^*_{\beta-2\epsilon_1}}{\Psi^*_{\beta+\epsilon}}(x)$
for all $x\in\sB^c$ (see also \cref{T3.1} in \cref{S3}).

We claim that
\begin{equation}\label{ET2.8H}
\inf_{\epsilon\in[0,\epsilon_1]}\,\mu^*_{\beta+\epsilon}(\sB)\;>\;0\,.
\end{equation}
Indeed, let $\Tilde\sB$ be a larger ball such that $\Bar\sB\subset\Tilde\sB$.
It suffices to exhibit the result for $\Tilde\sB$.
For some positive constants $\delta_i$, $i=1,2,3$, we have
\begin{equation*}
\sup_{\epsilon\in[0,\epsilon_1]}\;\sup_{x\in\partial\Tilde\sB}\;
\widetilde\Exp_x^{*,\epsilon} [\uuptau] \;\le\; \kappa^{-1}
\sup_{\epsilon\in[0,\epsilon_1]}\;\sup_{x\in\partial\Tilde\sB}\;
\tfrac{\Psi^*_{\beta-2\epsilon_1}}{\Psi^*_{\beta+\epsilon}}(x)
\;=:\; \delta_1 <\infty\,,
\end{equation*}
and also (see \cite[Theorem~2.6.1]{book})
\begin{equation*}
0\;<\; \delta_2 \;\le\;
\inf_{\epsilon\in[0,\epsilon_1]}\;\sup_{x\in\partial\sB}\;
\widetilde\Exp_x^{*,\epsilon} [\uptau(\Tilde\sB^c)] \;\le\;
\sup_{\epsilon\in[0,\epsilon_1]}\;\sup_{x\in\partial\sB}\;
\widetilde\Exp_x^{*,\epsilon} [\uptau(\Tilde\sB^c)]
\;\le\;  \delta_3 \;<\;\infty\,.
\end{equation*}
We use the inequality
$\mu^*_{\beta+\epsilon}(\Tilde\sB)\ge
\frac{\delta_2}{\delta_1+\delta_3}$,
which follows from the well-known characterization of invariant probability
measures due to Has$^{_{^{^\prime}}}\!$minski\u{\i} \cite[Theorem~2.6.9]{book},
and which establishes the claim.

It follows from \cref{ET2.8H} that the corresponding densities
$\eta^*_{\beta+\epsilon}$ are
locally bounded and also bounded away from $0$
uniformly in $\epsilon\in[0,\epsilon_1]$ by the Harnack inequality
(see proof of equation (3.2.6) in \cite{book}).
Therefore, standard pde estimates of the Fokker--Planck equation
show that this family of densities is locally H\"older equicontinuous
\cite[Theorem~8.24, p.~202]{GilTru}.
Given any $\theta\in(0,1)$ we may enlarge $\sB$ so that
$\mu^*_\beta(\sB)\ge 1-\theta$ and $\abs{f}\le\theta$ on $\sB^c$.
Let $\Bar\eta_\beta$ be the (uniform) limit of
$\eta^*_{\beta+\epsilon_n}$ on $\sB$ along some subsequence $\epsilon_n\searrow0$.
Since $\grad\psi^*_{\beta-\epsilon}$ is H\"older equicontinuous on $\sB$,
uniformly in $\epsilon\in[-\epsilon_1,\epsilon_1]$ as argued earlier,
it follows that $\Bar\eta_\beta$ is strictly positive on $\Bar\sB$.
It is straightforward to show then that $\Bar\eta_\beta$ is a positive solution of
the Fokker--Planck  equation for the (adjoint of the) operator
$\Lg + 2\langle a\grad\psi^*_\beta,\grad\rangle$.
By the uniqueness of the invariant probability measure we have
$\Bar\eta_\beta = C \eta^*_\beta$ for some positive constant $C$.
Since $\int_{\sB}\Bar\eta_\beta(x)\,\D x\le1$, we have
$C\le(1-\theta)^{-1}$.
Thus, since $\sup_{\epsilon\in[0,\epsilon_1]}\,
\norm{\eta^*_{\beta+\epsilon}}_\infty<\infty$, and $\abs{f}<\theta$ on $\sB^c$,
by Fatou's lemma we obtain
\begin{align*}
\limsup_{n\to\infty}\;\mu^*_{\beta+\epsilon_n}(f)
& \;\le\; \limsup_{n\to\infty}\; \int_\sB f(x) \eta^*_{\beta+\epsilon_n}(x)\,\D x
+ \theta  \\[5pt]
& \;\le\; \int_\sB f(x) \Bar\eta_\beta(x)\,\D x + \theta\\[5pt]
& \;\le\; (1-\theta)^{-1}\,\int_\sB f(x)\eta^*_\beta\,\D x + \theta\\[5pt]
&\;\le\; (1-\theta)^{-1}\,\mu^*_\beta(f)+ \theta\,.
\end{align*}
Since $\theta$ can be selected arbitrarily close to $0$,
we obtain from \cref{ET2.8D} that
$\lim_{\epsilon\searrow0}\,\frac{\Lambda_{\beta+\epsilon}-\Lambda_\beta}{\epsilon}
\le \mu^*_\beta(f)$.
Combining this with \cref{ET2.8F,ET2.8G}
we obtain \cref{ET2.8B}.
\qed\end{proof}

\section{Exponential ergodicity and strict monotonicity of principal eigenvalues}
\label{S3}

In this section we show that exponential ergodicity of \cref{E2.1}
is a sufficient condition for the strict monotonicity of the principal eigenvalue.
In \cite{Ichihara-11, Kaise-06}
exponential ergodicity is used to obtain results similar to \cref{T2.1}.
In these studies the coefficients
$a$, $b$, and $f$ are assumed to be $\Cc^2$, and this assumption seems hard
to waive as the technique used relies on a gradient estimate
(see \cite[Theorem~3.1]{Ichihara-11} and \cite[Lemma~2.4]{Kaise-06})
which is not available for less regular coefficients.
Our approach has allowed us to obtain the results in
\cref{S2} under much weaker hypotheses
on the coefficients.
Under some additional hypotheses, we show in this section that $\lamstr(f)=\sE(f)$.
Recall the definition of $\lambda''(f)$ in \cref{E-l''}.
It is straightforward to show that $\lambda''(f)\ge\sE(f)$.
We present an example where $\lamstr(f)<\sE(f)$, and therefore
also $\lamstr(f)<\lambda''(f)$.

\begin{example}\label{E3.1}
Let $\phi:\RR\to\RR_{+}$ be a smooth function which is strictly
positive on $[-1,1]$ and satisfies
$\phi(x) =\E^{-\frac{1}{2}\abs{x}}$ for $\abs{x}\ge1$.
Define
\begin{equation*}
f(x) \;\df\; -\frac{1}{\phi(x)}\,
\bigl(\phi''(x) + \sign(x)\, \phi'(x) - \phi(x)\bigr)\,.
\end{equation*}
Then $f(x) = \frac{5}{4}$ for $\abs{x}\ge1$, and
\begin{equation}\label{EE3.1A}
\phi''(x) +\sign(x)\, \phi'(x) + f(x)\phi(x) \;=\; \phi(x)\,.
\end{equation}

Consider the one-dimensional controlled diffusion
\begin{equation}\label{EE3.1B}
\D X_t \;=\; \sign(X_t)\,\D{t} + \sqrt{2}\, \D{W}_t\,.
\end{equation}
From \cref{EE3.1A} and \cref{L2.2}\,(ii) we have $\lamstr(f)\le 1$.
It is clear that \cref{EE3.1B}
is a transient process. Therefore, for any initial data $x$
\begin{equation*}
\sE_x(f)\;\ge\; \limsup_{T\to\infty}\; \frac{1}{T}\,
\Exp_x\biggl[\int_0^T f(X_t)\, \D{t}\biggr]\;=\;\frac{5}{4}\,.
\end{equation*}
Hence $\lamstr(f)<\sE(f)$.
\end{example}

\begin{remark}
\cref{E3.1} presents a
case where the conclusion of \cite[Theorem~1.9]{Berestycki-15}
fails to hold.
Since the operator $\Lg$ in this example is uniformly elliptic,
has bounded coefficients,
and $d=1$, the only aspect that makes it different from the class of operators in
part (i) of \cite[Theorem~1.9]{Berestycki-15}, is that it is not self-adjoint.
\end{remark}

Let us start by summarizing some equivalent characterizations
of exponential ergodicity.

\begin{theorem}\label{T3.1}
The following are equivalent.
\begin{enumerate}[(a)]
\item
For some ball $\sB_\circ$ there exists $\delta_\circ>0$ 
and $x_\circ\in \Bar{\sB}_\circ^c$ such
that $\Exp_{x_\circ}[\E^{\delta_\circ\, \uptau(\sB^c_\circ)}]<\infty$.

\item
For every ball $\sB$ there exists $\delta>0$ such
that $\Exp_{x}[\E^{\delta\, \uptau(\sB^c)}]<\infty\;$ for all $x\in\sB^c$.

\item
For every ball $\sB$, there exists
a positive function $\Lyap\in\Sobl^{2,d}(\Rd)$, with $\inf_\Rd\,\Lyap>0$, and positive
constants $\kappa_0$ and $\delta$
such that
\begin{equation}\label{ET3.1A}
\Lg\Lyap(x) \;\le\; \kappa_0\,\Ind_{\sB}(x) - \delta\Lyap(x)\qquad\forall\,x\in\Rd\,.
\end{equation}

\item
\Cref{E2.1} is recurrent, and
 $\lamstr(\Ind_{\sB^c})<1$ for every ball $\sB$.
\end{enumerate}
\end{theorem}

\begin{proof}
We first show that
(a)$\,\Rightarrow\,$(d).
It is clear that (a) implies that \cref{E2.1} is positive recurrent, and that
it is enough to prove that $\lamstr(\Ind_{\sB^c})<1$ for any $\sB\subset\sB_\circ$.
Let $f= \Ind_{\sB^c}$, and consider the Dirichlet eigensolutions 
$(\widehat\Psi_{n},\Hat\lambda_{n})$ in \cref{EL2.1A}.
It is easy to see that $\Hat\lambda_{n}<1$ for all $n$.
We claim that $\lamstr(f)<1$.
If not, then $\Hat\lambda_{n}\nearrow1$ as $n\to\infty$,
and $\widehat\Psi_{n}$ converges to some $\Psi\in\Sobl^{2,p}(\Rd)$, $p\ge d$,
which satisfies $\Lg\Psi= \Ind_{\sB}\Psi$ on $\Rd$ and $\Psi(0)=1$.
The same argument used in the proof of
\cref{L2.2} then shows that
$\Psi(x)=\Exp_x\bigl[\Psi(X_{\uptau(\sB^c_\circ)}\bigr]$.
Therefore, $\Psi$ attains a maximum on $\Bar\sB_\circ$, and by the strong maximum
principle it must be constant.
Thus $\Lg\Psi=0$ which contradicts the fact that $\Psi(0)=1$.

Next we show that (d)$\,\Rightarrow\,$(c).
If $\lamstr(f)<1$, for $f= \Ind_{\sB^c}$, then any limit point
$\Psi$ of the Dirichlet
eigenfunctions $\widehat\Psi_{n}$ as $n\to\infty$ satisfies
\begin{equation*}
\Lg\Psi\;=\; \Ind_{\Bar\sB}\Psi- \bigl(1-\lamstr(f)\bigr)\,\Psi
\;\le\; \Bigl(\sup_{\sB}\,\Psi\Bigr)\Ind_{\Bar\sB}
- \bigl(1-\lamstr(f)\bigr)\,\Psi\,.
\end{equation*}
Also by \cite[Lemma~2.1\,(c)]{ari-anup},
we have $\inf_\Rd\,\Psi = \min_{\Bar\sB}\,\Psi>0$.
Thus (c) holds with $\delta=1-\lamstr(f)$.

That (c)$\,\Rightarrow\,$(b) is well known, and can be shown by
a standard application of the It\^o--Krylov formula to \cref{ET3.1A}, by which we
obtain
\begin{align*}
\biggl(\inf_\Rd\;\Lyap\biggr)\,\Exp_{x}\Bigl[\E^{\delta
(\uptau(\sB^c)\wedge \uptau_R)}\Bigr]
- \Lyap(x)&\;\le\; \Exp_{x}\Bigl[\E^{\delta(\uptau(\sB^c)\wedge\uptau_R)}
\Lyap(X_{\uptau(\sB^c)\wedge\uptau_R})\Bigr] - \Lyap(x)\\
&\;\le\; \Exp_{x}\biggl[\int_{0}^{\uptau(\sB^c)\wedge\uptau_R}
\Bigl(\delta\E^{\delta t}\,\Lyap(X_t)
+ \E^{\delta t}\Lg\Lyap(X_t)\Bigr)\,\D{t}\biggr]\;\le\; 0\,.
\end{align*}
The result then follows by letting $R\to\infty$, and this completes the proof.
\qed\end{proof}

We introduce the following hypothesis.
\begin{itemize}
\item[\textbf{(H2)}]
There exists a lower-semicontinuous, inf-compact function
$\ell:\Rd\to[0, \infty)$ such that
$\sE(\ell)<\infty$, where $\sE(\cdot)$ is as defined in \cref{E-sE}.
\end{itemize}

\begin{lemma}
Under \textup{(H2)}, we have $\sE(\ell)=\lamstr(\ell)$,
and there exists a positive $V\in\Sobl^{2,p}(\Rd)$, $p\ge d$,
with $\inf_{\Rd} V>0$, and $V(0)=1$, satisfying
\begin{equation}\label{EL3.1A}
\Lg V + \ell\, V \;=\; \lamstr(\ell)\, V \quad \text{a.e. in\ } \Rd\,.
\end{equation}
In particular, the unique strong solution of \cref{E2.1} is exponentially
ergodic.
\end{lemma}

\begin{proof}
By \cref{E2.3} we have
\begin{equation}\label{EL3.1B}
\sE_x(\ell)\;\ge\;\limsup_{T\to\infty}\; \frac{1}{T}\,
\Exp_x\biggl[\int_0^T \ell(X_s)\, \D{s}\biggr]\,.
\end{equation}
Since $\sE(\ell)=\inf_{\Rd}\sE_x(\ell)$, \cref{EL3.1B} implies that
\begin{equation*}\limsup_{T\to\infty}\; \frac{1}{T}\,
\Exp_x\biggl[\int_0^T \ell(X_s)\, \D{s}\biggr]\;<\;\infty\end{equation*}
for some $x\in\Rd$.
The inf-compactness of $\ell$ then implies that
the unique strong solution of \cref{E2.1} is positive recurrent.
That $\sE(\ell)=\lamstr(\ell)$, and the existence
of a solution $V$ then follow by Theorem~1.4 and Lemma~2.1 in \cite{ari-anup},
respectively.
Exponential ergodicity then follows from \cref{EL3.1A}, using \cref{T3.1}.
\qed\end{proof}

An application of the It\^{o}--Krylov formula to \cref{EL3.1A},
followed by Fatou's lemma, shows that
\begin{equation}\label{E3.6}
\Exp_{x}\Bigl[\E^{\int_{0}^{\uuptau_r}
[\ell(X_{t})-\lamstr(\ell)]\,\D{t}}\, V (X_{\uuptau})\Bigr]\;\le\; V(x)
\qquad \forall\,x\in B^c_r\,,\ \forall\, r>0\,,
\end{equation}
where $\uuptau_r$, as defined earlier, denotes the first
hitting time of the ball $B_r$.

The next result shows that (H2) implies (P1).

\begin{theorem}\label{T3.2}
Assume \textup{(H2)}, and suppose that $f$ is a potential such that
$\ell-f$ is inf-compact.
Then for any continuous $h\in\Cc_{\mathrm{o}}^+(\Rd)$ we have
\begin{equation*}
\lamstr(f-h)\;<\;\lamstr(f)\;=\;\sE_x(f)\quad\forall\,x\in\Rd\,.
\end{equation*}
\end{theorem}

\begin{proof}
Let $h\in\Cc_{\mathrm{o}}^+(\Rd)$, and $\Tilde{f}\df f-h$.
It is easy to see that $\sE(f)$ and $\sE(\Tilde{f})$ are both finite.
It is shown in \cite{ari-anup, biswas-11} that the
Dirichlet eigensolutions $(\widehat\Psi_r,\Hat\lambda_r)$ in \cref{EL2.1A}
converge, along some subsequence as $r\to\infty$, to $\bigl(\Psi^*,\lamstr(f)\bigr)$
which satisfies
\begin{equation}\label{ET3.2A}
\Lg\Psi^* + f\,\Psi^* \;=\;\lamstr(f)\,\Psi^*
\quad\text{on\ }\Rd\,,\qquad\text{and} \quad \lamstr(f)\;\le\; \sE_x(f)
\quad\forall\,x\in\Rd\,.
\end{equation}
It is also clear that \cref{L2.2}\,(i) holds for $(\widehat\Psi_n, \Hat\lambda_n)$.
Now choose a bounded ball $\sB$ such that 
\begin{equation*}
\bigl(\abs{f(x)}+\abs{h(x)}\bigr ) + \sup_n\bigl(\abs{\Hat\lambda_n(f)}
+\abs{\Hat\lambda_n(\Tilde{f})}\bigr) + \lamstr(\ell)+ 1 \;<\; \ell(x)
\quad  \forall\,x\in\sB^c\,.
\end{equation*}
This is possible since $\ell-f$ is inf-compact.
In view of \cref{E3.6} we note that \cref{EL2.2F} holds with
$f-h-\lamstr(f-h)$ replaced by $\ell-\lamstr(\ell)$.
Thus with the above choice of $\sB$, we can justify the passing to
the limit in \cref{EL2.2H}, and therefore, we obtain
\begin{equation}\label{ET3.2B}
\Psi^*(x) \;=\; \Exp_{x}
\Bigl[\E^{\int_{0}^{\uuptau}
[f(X_{t})-\lamstr(f)]\,\D{t}}\, \Psi^* (X_{\uuptau})\Bigr]
\qquad\forall\, x\in\sB^c\,,
\end{equation}
with $\uuptau=\uptau(\sB^c)$.
Recall the definition of $\bigl(\Tilde\Psi^*,\lamstr(\Tilde{f})\bigr)$
in \cref{ET2.2B}.
A similar argument also gives
\begin{equation}\label{ET3.2C}
\Tilde\Psi^*(x) \;=\; \Exp_{x}
\Bigl[\E^{\int_{0}^{\uuptau}
[\Tilde{f}(X_{t})-\lambda\str(\Tilde{f})]\,\D{t}}\, \Tilde\Psi^* (X_{\uuptau})\Bigr]
\qquad\forall\, x\in\sB^c\,.
\end{equation}
In fact, the above relations hold for any bounded domain
$D\supset \sB$ with $\uuptau=\uptau(D^c)$.

Suppose that $\lambda\str(f)=\lamstr(\Tilde{f})$.
Then
\begin{equation*}
f(x)-h(x)-\lamstr(\Tilde{f})\;\le\; f(x) - \lambda\str(f)\,.
\end{equation*}
Thus if we multiply $\Psi^*$ with a suitable positive constant such that
$\Psi^*-\Tilde\Psi^*$
is nonnegative in $\sB$ and attains a minimum of $0$ in $\sB$,
it follows from \cref{ET3.2B,ET3.2C} that $\Psi^*-\Tilde\Psi^*$
is nonnegative in $\Rd$.
Since \cref{ET2.2E} holds,
and we conclude exactly as in the proof of
\cref{T2.2} that $\lamstr(\Tilde{f})<\lamstr(f)$.

Next we show that  $\lamstr(f)=\sE_x(f)$ for all $x\in\Rd$.
We have already established the strict monotonicity
of $\lamstr(f)$ at $f$, and therefore, \cref{T2.1} applies.
Hence for any continuous $g$ with compact support we
have from \cite[Theorem~1.3.10]{Kunita} that
\begin{align}\label{ET3.2D}
\Exp_x\Bigl[\E^{\int_{0}^{T}[f(X_{t})-\lamstr(f)]\,\D{t}}\,g(X_{T})\Bigr]
\;=\;\Psi^*(x)\, \widetilde\Exp^{\psi^*}_x\biggl[\frac{g(Y^*_T)}{\Psi^*(Y^*_T)}\biggr]
\;\xrightarrow[T\to\infty]{}\; \Psi^*(x)\,
\mu^*\Bigl(\frac{g}{\Psi^*}\Bigr) \;>\;0\,,
\end{align}
where $\mu^*$ denotes the invariant measure of the twisted process $Y^*$
satisfying \cref{EL2.5A}.
Let $\Tilde\sB$ be a ball such that $f(x)-\lamstr(f)<\ell(x)-\lamstr(\ell)$
for $x\in\Tilde\sB^c$.
Thus from \cref{EL3.1A} we obtain
\begin{equation}\label{ET3.2E}
\Lg V + \bigl(f-\lamstr(f)\bigr) V \;\le\; \kappa \Ind_{\Tilde\sB}\,,
\end{equation}
with $\kappa=\max_{\Tilde\sB}\bigl(\abs{f}+\abs{\ell}+\abs{\lamstr}
+\lamstr(\ell)\bigr)\,V$.
Applying the It\^o--Krylov formula
to \cref{ET3.2E} followed by Fatou's lemma we obtain
\begin{align*}
\Bigl(\min_{\Rd} V\Bigr)\,
\Exp_x\bigl[\E^{\int_{0}^{T}[f(X_{t})-\lamstr(f)]\,\D{t}}\bigr]
& \;\le\; \Exp_x\bigl[\E^{\int_{0}^{T}[f(X_{t})-\lamstr(f)]\,\D{t}}\, V(X_T)\bigr]\\
&\;\le\; \kappa\, \int_0^T \Exp_x\bigl[\E^{\int_{0}^{t}[f(X_{t})-\lamstr(f)]\,\D{t}}\,
\Ind_{\Tilde\sB}(X_T)\bigr]\,\D{t} + V(x)\\
&\;\le\; \kappa'\, T + V(x)\,,
\end{align*}
for some constant $\kappa'$, where in the last inequality we have used \cref{ET3.2D}.
Taking logarithms on both sides of the preceding inequality,
then dividing by $T$, and letting $T\to\infty$, we obtain
$\lamstr(f)\ge \sE_x(f)$ for all $x\in\Rd$.
Combining this with \cref{ET3.2A} results in equality.
\qed\end{proof}

\begin{remark}\label{R3.1}
Continuity of $h$ is superfluous in \cref{T3.2}.
The result holds if $h$ is a non-trivial, nonnegative measurable function,
vanishing at infinity.
\end{remark}

\begin{corollary}
Under the assumptions of \cref{T3.2}, for any potential $\Tilde{f}\lneqq f$, we have
$\lamstr(\Tilde{f})<\lamstr(f)$.
\end{corollary}

\begin{proof}
Note that for any cut-off function $\chi$ we have
$\lamstr(\Tilde{f})\le \lamstr(\chi\Tilde{f}+ (1-\chi)f)$.
Then the result follows from \cref{T3.2,R3.1}.
\qed\end{proof}

\begin{remark}
In \cref{T3.2} we can replace the assumption
that $f$ is bounded from below in $\Rd$ by the  hypothesis
that  $\ell-\abs{f}$ is inf-compact.
\end{remark}

Let us now discuss the exponential ergodicity and show that this implies (H2).

\begin{proposition}\label{P3.1}
Let $\ell:\Rd\to\RR$ be inf-compact,
and suppose $\phi\in\Sobl^{2,d}(\Rd)$ is bounded below in $\Rd$ and satisfies
\begin{equation}\label{P3.1A}
\Lg\phi +\langle \grad\phi, a\grad\phi\rangle \;=\; -\ell\,.
\end{equation}
 then $\sE_x(\ell)<\infty\;$ for all $x\in\Rd$.
\end{proposition}

\begin{proof}
Let $\Phi(x)=\exp(\phi(x))$.
Then $\inf_{\Rd}\Phi>0$, and \cref{P3.1A} gives
\begin{equation}\label{P3.1B}
\Lg\Phi + \ell\Phi\;=\;0\,.
\end{equation}
Now apply the It\^o--Krylov formula to \cref{P3.1B} followed by Fatou's lemma to obtain
\begin{equation*}
\Exp_x\Bigl[\E^{\int_0^T \ell(X_s)\, \D{s}}\, \Phi(X_T) \Bigr] \;\le\; \Phi(x)\,.
\end{equation*}
Taking logarithm on both sides, diving by $T$ and letting $T\to\infty$,
we obtain $\sE_x(\ell)<\infty$.
\qed\end{proof}

\begin{example}\label{EG3.1}
Let $a=\frac{1}{2}I$ and $b(x)= b_1(x) + B(x)$ where $B$ is bounded and 
\begin{equation*}
\langle b_1(x), x\rangle \;\le\; -\kappa \abs{x}^\alpha,
\quad \text{for some}\; \alpha\in(1, 2]\,.
\end{equation*}
Then we take $\phi(x)=\theta\,|x|^{\alpha}$ for $\abs{x}\ge 1\, ,\theta\in(0, 1)$.
It is easy to check that for a suitable choice of $\theta\in(0, 1)$,
\cref{P3.1A} holds for 
$\ell(x)\sim\abs{x}^{2\alpha-2}$.
\end{example}

\begin{remark}\label{R3.2}
\Cref{P3.1A} is a stronger condition than strict monotonicity
of $\lamstr(f)$ at $f$.
In fact, \cref{P3.1A} might not hold in many important situations.
For instance, if $a$ and $b$ are both bounded, and $a$ is uniformly elliptic,
then it is not possible to find inf-compact $\ell$ satisfying \cref{P3.1A}.
Otherwise, we can find a finite principal
eigenvalue for the operator $\Lg^{\ell}$, by a same method as in \cref{ET3.2A},
which would contradict \cite[Proposition~2.6]{Berestycki-15}.
\end{remark}

Even though \cref{P3.1A} does not hold for bounded $a$ and $b$,
strict monotonicity
of $\lamstr(f)$ at $f$ can be asserted under suitable hypotheses.
This is the subject of the following theorem.

\begin{theorem}\label{T3.3}
Let $\Lyap\in\Sobl^{2,d}(\Rd)$ such that
$\inf_{\Rd}\Lyap >0$, satisfying
\begin{equation}\label{ET3.3A}
\Lg\Lyap\;\le\; \kappa_0 \Ind_{\sK} -\gamma \Lyap\quad\text{on\ }\Rd\,,
\end{equation}
for some compact set $\sK$ and positive constants $\kappa_0$ and $\gamma$.
Let $f$ be a
nonnegative bounded measurable function with
$\limsup_{x\to\infty}\,f(x)< \gamma$.
Then for any $h\in\Cc_{\mathrm{o}}^+(\Rd)$,
we have $\lamstr(f-h)<\lamstr(f)=\sE_x(f)$ for all $x\in\Rd$.
\end{theorem}

\begin{proof}
Let $\Tilde{f}= f-h$.
Suppose $\lamstr(\Tilde{f})=\lamstr(f)$.
Applying an argument similar to \cref{ET3.2A} we can find $\Psi^*$ and $\Tilde\Psi^*$
that satisfy
\begin{align*}
\Lg\Psi^* + f\,\Psi^* &\;=\;\lamstr(f)\,\Psi^*,\\[2pt]
\Lg\Tilde\Psi^* + (f-h)\Tilde\Psi^* &\;=\;\lamstr(\Tilde{f})\Tilde\Psi^*.
\end{align*}
Let $\sK_0\supset\sK$ be any compact set such that
$f<\gamma$ on $\sK_0^c$.
If $\uuptau$ denotes the first hitting time to the compact set $\sK_0$,
then by an application of the It\^o--Krylov formula to \cref{ET3.3A} we obtain
\begin{equation*}
\Exp_x\bigl[\E^{\gamma\uuptau} \bigr] <\infty\,, \quad x\in\sK_0^c\,.
\end{equation*}
We next use the fact that if $\Lg$ corresponds to a recurrent diffusion and
$f$ is nonnegative then $\lamstr(f)\ge0$.
Indeed, in this case we have $\Lg \Psi^* \le \lamstr(f) \Psi^*$.
If $\lamstr(f)\le0$, this implies that $\Psi^*(X_t)$ is a nonnegative supermartingale
and since it is integrable, it converges a.s.
Since the process is recurrent, this implies that $\Psi^*$
must equal to a constant, which, in turn,
necessitates that $\lamstr(f)=0$ (and $f=0$).
Thus, since $\lamstr(f)\ge0$, an argument similar to the proof of \cref{L2.2}\,(ii)
shows that
\begin{align*}
\Psi^*(x) &\;=\;\; \Exp_{x}
\Bigl[\E^{\int_{0}^{\uuptau}
[f(X_{t})-\lamstr(f)]\,\D{t}}\, \Psi^* (X_{\uuptau})\Bigr]\,,\\[5pt]
\Tilde\Psi^*(x) &\;=\;\; \Exp_{x}
\Bigl[\E^{\int_{0}^{\uuptau}
[\Tilde{f}(X_{t})-\Tilde\lambda\str(\Tilde{f})]\,\D{t}}\,
\Tilde\Psi^* (X_{\uuptau})\Bigr]\,,
\end{align*}
for $x\in\sK_0^c$.
Therefore, applying the strong maximum principle as in \cref{T3.2},
we obtain $h\,\Tilde\Psi^*=0$ which is a contradiction since $h\neq 0$
and $\Tilde\Psi^*>0$. Thus 
we have $\lamstr(f-h)<\lamstr(f)$.
That $\lamstr(f)=\sE_x(f)$ for all $x\in\Rd$
follows by an argument similar to the one used in the proof \cref{T3.2}.
\qed\end{proof}

\begin{example}
Suppose $a=\frac{1}{2}I$, where $I$ denotes the identity matrix, and 
\begin{equation*}
\langle b(x), x\rangle\;\le\; -\abs{x}, \quad \text{outside a compact set\ } \sK_1\,.
\end{equation*}
With $\Lyap(x)=\exp(\abs{x})$ for $\abs{x}\ge 1$, we have 
\begin{equation*}
\Lg\Lyap\;=\; \biggl(\frac{d-1}{2\abs{x}}
+\frac{1}{2} - \frac{\langle b(x), x\rangle}{\abs{x}}\biggr)
\Lyap\;\le\; \biggl(\frac{d-1}{2\abs{x}}-\frac{1}{2}\biggr)\Lyap
\quad \text{for}\; \abs{x}\ge 1\,.
\end{equation*}
\end{example}

\section{Risk-sensitive control}\label{S4}
In this section we apply the results developed in the previous sections to
the risk-sensitive control problem.
As mentioned earlier, we establish the existence and
uniqueness of solutions to the risk-sensitive HJB equation,
and use this to completely characterize the optimal Markov controls
(see \cref{T4.1,T4.2}).
Another interesting result is the continuity of the controlled principal eigenvalue
with respect to the stationary Markov controls.
This is done in \cref{T4.3}. 
We first introduce the control problem.

\subsection{The controlled diffusion model}

Consider a
controlled diffusion process $X = \{X_{t},\,t\ge0\}$
which takes values in the $d$-dimensional Euclidean space $\RR^{d}$, and
is governed by the It\^o  equation
\begin{equation}\label{E4.1}
\D{X}_{t} \;=\;b(X_{t},U_{t})\,\D{t} + \upsigma(X_{t})\,\D{W}_{t}\,.
\end{equation}
All random processes in \cref{E4.1} live in a complete
probability space $(\Omega,\sF,\Prob)$.
The process $W$ is a $d$-dimensional standard Wiener process independent
of the initial condition $X_{0}$.
The control process $U$ takes values in a compact, metrizable set $\Act$, and
$U_{t}(\omega)$ is jointly measurable in
$(t,\omega)\in[0,\infty)\times\Omega$.
The set $\Uadm$ of \emph{admissible controls} consists of the
control processes $U$ that are \emph{non-anticipative}:
for $s < t$, $W_{t} - W_{s}$ is independent of
\begin{equation*}
\sF_{s} \;\df\;\text{the completion of~}
\cap_{y>s}\sigma\{X_{0},U_{r},W_{r},\;r\le y\}
\text{~relative to~}(\sF,\Prob)\,.
\end{equation*}

We impose the following standard assumptions on the drift $b$
and the diffusion matrix $\upsigma$ to guarantee existence
and uniqueness of solutions.
\begin{itemize}
\item[(B1)]
\emph{Local Lipschitz continuity:\/}
The functions
$b\colon\RR^{d}\times\Act\to\RR^{d}$ and 
$\upsigma\colon\RR^{d}\to\RR^{d\times d}$
are continuous, and satisfy
\begin{equation*}
\abs{b(x,u)-b(y, u)} + \norm{\upsigma(x) - \upsigma(y)}
\;\le\;C_{R}\,\abs{x-y}\qquad\forall\,x,y\in B_R\,,\ \forall\, u\in\Act\, .
\end{equation*}
for some constant $C_{R}>0$ depending on $R>0$.

\item[(B2)]
\emph{Affine growth condition:\/}
For some $C_0>0$, we have
\begin{equation*}
\sup_{u\in\Act}\; \langle b(x,u),x\rangle^{+} + \norm{\upsigma(x)}^{2}\;\le\;C_0
\bigl(1 + \abs{x}^{2}\bigr) \qquad \forall\, x\in\RR^{d},
\end{equation*}
\item[(B3)]
\emph{Nondegeneracy:\/}
Assumption (A3) in \cref{S1.1} holds.
\end{itemize}

It is well known that under (B1)--(B3), for any admissible control
there exists a unique solution of \cref{E4.1}
\cite[Theorem~2.2.4]{book}.
We define the family of operators $\cL_u\colon\Cc^{2}(\RR^{d})\mapsto\Cc(\RR^{d})$,
where $u\in\Act$ plays the role of a parameter, by
\begin{equation*}
\cL_u f(x) \;\df\;  a^{ij}(x)\,\partial_{ij} f(x)
+ b^{i}(x,u)\, \partial_{i} f(x)\,,\quad u\in\Act\,.
\end{equation*}

\paragraph{The risk-sensitive criterion}
Let $\mathfrak{C}$ denote the class of functions
$c(x,u)$ in $\Cc(\Rd\times\Act,\RR_+)$
that are locally Lipschitz in $x$ uniformly with respect to $u\in\Act$.
We let $c\in\mathfrak{C}$ denote the \emph{running cost} function,
and 
for any admissible control $U\in\Uadm$, we define
the risk-sensitive objective function $\sE^U_x(c)$ by
\begin{equation}\label{E-sEU}
\sE^U_x(c)\;\df\;\limsup_{T\to\infty}\;\frac{1}{T}\,
\log\Exp_x\Bigl[\E^{\int_0^T c(X_s, U_s)\, \D{s}}\Bigr]\,.
\end{equation}
We also define $\Lambda\str_x\df\inf_{U\in\Uadm}\,\sE^U_x(c)$.

\subsection{Relaxed controls}

We adopt the well-known \emph{relaxed control} framework \cite{book}.
According to this relaxation, a stationary Markov control is a measurable
map from $\Rd$ to $\cP(\Act)$, the latter denoting
the set of probability measures on $\Act$ under
the Prokhorov topology.
Let $\Usm$ denote the class of all such stationary Markov controls.
A control $v\in\Usm$ may be viewed as a kernel on $\cP(\Act)\times\Rd$,
which we write as $v(\D{u}\!\mid\! x)$.
We say that a control $v\in\Usm$ is precise if it is a measurable map
from $\Rd$ to $\Act$.
We extend the definition of $b$ and $c$ as follows.
For $v\in\Usm$ we let
\begin{equation*}
b_v(x)\;\df\;\int_{\Act} b(x,u)\,v(\D{u}\!\mid\! x)\,,
\quad \text{and}\quad c_v(x)\;\df\;\int_{\Act} c(x,u)\,v(\D{u}\!\mid\! x)
\quad \text{for}\; v\in\cP(\Act)\,.
\end{equation*}
It is easy to see from (B2) and Jensen's inequality that
\begin{equation*}
\sup_{v\in\Usm}\langle b_v(x), x\rangle ^+ \;\le\; \; C_0(1+\abs{x}^2)
\qquad\forall\, x\in\Rd\,.
\end{equation*}

For $v\in\Usm$, consider the relaxed diffusion
\begin{equation}\label{E-rsde}
\D{X}_{t} \;=\;b_v(X_{t})\,\D{t} + \upsigma(X_{t})\,\D{W}_{t}\,.
\end{equation}
It is well known that under $v\in\Usm$
\cref{E-rsde} has a unique strong solution \cite{Gyongy-96},
which is also a strong Markov process.
It also follows from the work in \cite{Bogachev-01} that under
$v\in\Usm$, the transition probabilities of $X$
have densities which are locally H\"older continuous.
Thus $\Lg_{v}$ defined by
\begin{equation*}
\Lg_{v} f(x) \;\df\; a^{ij}(x)\,\partial_{ij} f(x)
+ b^{i}_v(x)\, \partial_{i} f(x)\,,\quad v\in\Usm\,,
\end{equation*}
for $f\in\Cc^{2}(\RR^{d})$,
is the generator of a strongly-continuous
semigroup on $\Cc_{b}(\RR^{d})$, which is strong Feller.
We let $\Prob_{x}^{v}$ denote the probability measure and
$\Exp_{x}^{v}$ the expectation operator on the canonical space of the
process under the control $v\in\Usm$, conditioned on the
process $X$ starting from $x\in\RR^{d}$ at $t=0$.
We denote by $\Ussm$ the subset of $\Usm$ that consists
of \emph{stable controls}, i.e.,
under which the controlled process is positive recurrent,
and by $\mu_v$ the invariant probability measure of the process
under the control $v\in\Ussm$.

\begin{definition}
For $v\in\Usm$ and a locally bounded measurable
function $f\colon\Rd\to\RR$, we let $\lambda\str_v(f)$
denote the principal eigenvalue of the operator $\Lg_v^f\df\Lg_v+f$ on $\Rd$
(see \cref{D2.1}).

We also adapt the notation in \cref{E2.3} to the control setting,
and define
\begin{equation*}
\sE_x^v(f)\;\df\;\limsup_{T\to\infty}\, \frac{1}{T}\,
\log\Exp_x^v\Bigl[\E^{\int_0^T f(X_s)\, \D{s}}\Bigr]\,,
\quad \text{and}\quad \sE^v(f)\;\df\;\inf_{x\in\Rd}\;\sE^v_x(f)\,,
\qquad v\in\Usm\,.
\end{equation*}
We refer to $\sE^v(f)$ as the \textit{risk-sensitive} average of $f$
under the control $v$.

Recall the risk-sensitive objective function $\sE_x^U$ defined in
\cref{E-sEU} and the optimal value $\Lambda\str$.
We say that a stationary Markov control $v\in\Usm$ is optimal (for the
risk-sensitive criterion) if
$\sE^v_x(c_v)=\Lambda\str_x$ for all $x\in\Rd$, and we let
$\Usm^*$ denote the class of these controls.
\end{definition}

\subsection{Optimal Markov controls and the risk-sensitive HJB}

We start with the following assumption.

\begin{assumption}[uniform exponential ergodicity]\label{A4.1}
There exists an inf-compact function $\ell\in\Cc(\Rd)$ and a positive function
$\Lyap\in \Sobl^{2,d}(\Rd)$, satisfying
$\inf_\Rd\,\Lyap>0$, such that
\begin{equation}\label{EA4.1A}
\sup_{u\in\Act}\;\cL_u \Lyap \;\le\; \Bar{\kappa}\, \Ind_{\sK} - \ell\Lyap
\quad\text{a.e.\ on\ }\Rd\,,
\end{equation}
for some constant $\Bar{\kappa}$, and a compact set $\sK$.
\end{assumption}
 
It is easy to see that for $\Bar{\kappa}_\circ\df\frac{\Bar{\kappa}}{\min_{\Rd}\,\Lyap}$
we obtain from \cref{EA4.1A} that
\begin{equation*}
\sup_{u\in\Act}\;\cL_u \Lyap + (\ell-\Bar{\kappa}_\circ)\Lyap\;\le\; 0\,,
\end{equation*}
and therefore, applying the It\^{o}--Krylov formula, we have
$\sE^v_x(\ell)\le \Bar{\kappa}_\circ$
for any stationary Markov control $v\in\Usm$, and all $x\in\Rd$.
\begin{example}
Let $\upsigma$ be bounded and $b:\Rd\times\Act\to\Rd$ be such that
\begin{equation*}
\langle b(x,u)-b(0, u), x\rangle\; \le \;-\kappa\, \abs{x}^\alpha,
\quad \text{for some}\; \alpha\in(1,2]\,, \quad (x,u)\in\Rd\times\Act\,.
\end{equation*}
Then as seen in \cref{EG3.1}, $\Lyap(x)=\exp(\theta\,\abs{x}^{\alpha})$,
for $\abs{x}\ge 1$, satisfies \cref{EA4.1A}
for sufficiently small $\theta>0$, and $\ell(x)\sim \abs{x}^{2\alpha-2}$.
Note that $\alpha=2$ and $\upsigma=I$ is considered in \cite{Fleming-95}.
\end{example}

We introduce the class of running costs
$\cC_\ell$ defined by
\begin{equation*}
\cC_\ell\;\df\;\Bigl\{c\in\mathfrak{C}\;\colon\;
\ell(\cdot)-\max_{u\in\Act}\, c(\cdot, u)\; \text{is inf-compact}\Bigr\}\,.
\end{equation*}

The first important result of this section is the following.

\begin{theorem}\label{T4.1}
Suppose \cref{A4.1} holds, and $c\in\cC_\ell$.
Then $\Lambda\str=\Lambda\str_x$ does not depend on $x$,
and there exists a positive solution $V\in\Cc^2(\Rd)$ satisfying
\begin{equation}\label{ET4.1A}
\min_{u\in\Act}\;[\cL_u V + c(\cdot, u) V]\;=\;\Lambda\str V\quad\text{on\ }\Rd\,,
\quad\text{and\ }V(0)=1\,.
\end{equation}
In addition, if $\;\bUsm\subset\Usm$ denotes the class of Markov controls $v$
which satisfy
\begin{equation*}
\Lg_v V + c_v V \;=\; \min_{u\in\Act}\;[\cL_u V + c(\cdot, u) V]
\qquad\text{a.e.\ in\ }\Rd\,,
\end{equation*}
then the following hold.
\begin{enumerate}[(a)]
\item
$\bUsm\subset\Usm^*$, and it holds that $\lamstr_v(c_v)=\Lambda\str$
for all $v\in\bUsm$;

\item
$\Usm^*\subset\bUsm\,$;

\item
\Cref{ET4.1A} has a unique positive solution in $\Cc^2(\Rd)$
\textup{(}up to a multiplicative constant\/\textup{)}.
\end{enumerate}
\end{theorem}

\begin{proof}
Using a standard argument (see \cite{biswas-11a, biswas-11, ari-anup}) we can find
a pair
$(V, \Hat\lambda)\in\Cc^2(\Rd)\times\RR$, with $V>0$ on $\Rd$, and $V(0)=1$,
that satisfies
\begin{equation}\label{ET4.1B}
\min_{u\in\Act}\;[\cL_u V + c(\cdot, u) V]\;=\;\Hat\lambda V\,,\quad
\Hat\lambda\;\le\; \inf_{x\in\Rd}\;\Lambda\str_x\,.
\end{equation}
This is obtained as a limit of Dirichlet eigensolutions
$(\widehat{V}_n,\Hat\lambda_n)
\in\bigl(\Sobl^{2,p}(B_n)\cap\Cc(\Bar{B}_n)\bigr)\times\RR$,
for any $p>d$, satisfying
$\widehat{V}_n>0$ on $B_n$, $\widehat{V}_n=0$ on
$\partial B_n$, $\widehat{V}_n(0)=1$,
and
\begin{equation*}
\min_{u\in\Act}\;\bigl[\cL_u\widehat{V}_n(x) + c(x,u)\,\widehat{V}_n(x)\bigr]
\;=\; \Hat\lambda_n\,\widehat{V}_n(x)
\qquad\text{a.e.\ }x\in B_n\,.
\end{equation*}
For $v\in\bUsm$ we have
\begin{equation}\label{ET4.1C}
\Lg_v V + c_v V \;=\; a^{ij}\partial_{ij} V + \langle b_v,\grad V\rangle + c_v V
\;=\;\Hat\lambda V\quad\text{on\ }\Rd\,.
\end{equation}
By \cref{C2.1} we obtain $\Hat\lambda\ge\lamstr_v(c_v)$.
Also by \cref{T3.2} we have $\lamstr_v(c_v)=\sE^v_x(c_v)$ for all $x\in\Rd$.
Combining these estimates with \cref{ET4.1B} we obtain
\begin{equation*}
\Lambda\str_x \;\le\; \sE^v_x(c_v) \;=\; \lamstr_v(c_v)
\;\le\; \Hat\lambda \;\le\; \inf_{z\in\Rd}\;\Lambda\str_z\qquad
\forall\,x\in\Rd\,.
\end{equation*}
This of course shows that $\Hat\lambda=\lamstr_v(c_v)=\Lambda\str_x$ for all $x\in\Rd$,
and also proves part (a).

We continue with part (b).
By \cref{T3.2} we have
\begin{equation}\label{ET4.1D}
\lamstr_v(c_v-h)\;<\;\lamstr_v(c_v)\qquad\forall\,h\in\Cc_{\mathrm{o}}^+(\Rd)\,,
\quad\forall\,v\in\Usm\,.
\end{equation}
In turn, by \cref{L2.4} there exists a unique 
eigenfunction $\Psi_v\in\Sobl^{2,d}(\Rd)$
which is associated with the principal eigenvalue
$\lamstr_v(c_v)$ of the operator $\Lg_v^{c_v}=\Lg_v+c_v$.
Since $\Hat\lambda=\lamstr_v(c_v)$ for all
$v\in\bUsm$ by part (a), it follows by \cref{ET4.1C} that
\begin{equation}\label{ET4.1E}
V \;=\; \Psi_v\quad\forall\,v\in\bUsm\,.
\end{equation}
By \cref{ET4.1D} and \cref{L2.2}\,(ii), and since \cref{E-rsde} is recurrent,
we have
\begin{equation}\label{ET4.1F}
\Psi_v(x)\;=\;\Exp_x^v\Bigl[e^{\int_0^{\uuptau}[c_v(X_s)-\Lambda\str]\D{s}}\,
\Psi_v(X_{\uuptau})\Bigr]
 \qquad \forall\,x\in\sB^c\,,\quad\forall\,v\in\Usm^*\,,
\end{equation}
and all sufficiently large balls $\sB$ centered at $0$, where
$\uuptau=\uptau(\sB^c)$, as usual.

Since the Dirichlet eigenvalues satisfy
$\Hat\lambda_n<\Hat\lambda=\Lambda\str$ for all $n\in\NN$,
 the Dirichlet problem
\begin{equation}\label{ET4.1G}
\min_{u\in\Act}\;
\bigl[\cL_u\varphi_{n}(x)+ \bigl(c(x,u)-\Lambda\str\bigr)\,\varphi_{n}(x)\bigr]
\;=\; -\alpha_n\,\Ind_{\sB}(x)\qquad\text{a.e.\ }x\in B_n\,,\qquad
\varphi_{n}=0\text{\ \ on\ \ }\partial B_n\,,
\end{equation}
with $\alpha_n>0$,
has a unique solution $\varphi_{n}\in\Sobl^{2,p}(B_n)\cap\Cc(\Bar{B}_n)$,
for any $p\ge1$ \cite[Theorem~1.9]{Quaas-08a}
(see also \cite[Theorem~1.1\,(ii)]{Yoshimura-06}).
We choose $\alpha_n$ as follows: first select
$\Tilde\alpha_n>0$ such that the solution
$\varphi_n$ of \cref{ET4.1G} with $\alpha_n=\Tilde\alpha_n$
satisfies $\varphi_n(0)=1$, and then set $\alpha_n=\min(1,\Tilde\alpha_n)$.
Passing to the limit in \cref{ET4.1G} as $n\to\infty$ along a subsequence,
we obtain a nonnegative solution
$\Phi\in\Sobl^{2,p}(\Rd)$ of
\begin{equation}\label{ET4.1H}
\min_{u\in\Act}\;
\bigl[\cL_u\Phi(x)+ \bigl(c(x,u)-\Lambda\str\bigr)\,\Phi(x)\bigr]
\;=\; -\alpha\,\Ind_{\sB}(x)\,,\qquad x\in\Rd\,.
\end{equation}
It is evident from the construction that if $\alpha=0$ then
$\Phi(0)=1$.
On the other hand, if $\alpha>0$, then necessarily $\Phi$ is positive on $\Rd$.
Let $\Hat{v}\in\Usm$ be a selector from the minimizer of \cref{ET4.1H}.
If $\alpha>0$, then \cref{ET4.1H} implies that there exists
$h\in\Cc_{\mathrm{o}}^+(\Rd)$ such that
$\lamstr_{\Hat{v}}(c_{\Hat{v}}+h) \le \Lambda\str$.
Since $\lamstr_{\Hat{v}}(c_{\Hat{v}})=\sE^{\Hat{v}}_x(c_{\Hat{v}})$ for all $x\in\Rd$
by \cref{T3.2}, and $\sE^{\Hat{v}}_x(c_{\Hat{v}})\ge\Lambda\str$,
then, in view of \cref{C2.1},
this contradicts \cref{ET4.1D} and the convexity of $\lamstr_{\Hat{v}}$.
Therefore, we must have $\alpha=0$.
Let $\Bar{v}\in\Usm^*$.
Applying the It\^{o}--Krylov formula to \cref{ET4.1G} we obtain
\begin{equation*}
\varphi_n(x) \;\le\; \Exp_x^{\Bar{v}} \Bigl[\E^{\int_{0}^{\uuptau}
[c_{\Bar{v}}(X_s)-\Lambda\str]\,\D{s}}\,
\varphi_n(X_{\uuptau})\,\Ind_{\{\uuptau<T\wedge\uptau_n\}}\Bigr]
+\Exp_x^{\Bar{v}} \Bigl[\E^{\int_{0}^{T}[c_{\Bar{v}}(X_s)-\Lambda\str]\,\D{s}}\,
\varphi_n(X_{T})\,
\Ind_{\{T<\uuptau\wedge\uptau_n\}}\Bigr]
\qquad\forall\,T>0\,,
\end{equation*}
and for all $x\in B_n\setminus\sB$,
where $\uuptau=\uptau(\sB^c)$.
Using the argument in the proof of \cite[Lemma~2.11]{ari-anup},
we obtain 
\begin{equation}\label{ET4.1I}
\Phi(x) \;\le\;
\Exp_x^{\Bar{v}} \Bigl[\E^{\int_{0}^{\uuptau}[c_{\Bar{v}}(X_s)-\Lambda\str]\,\D{s}}\,
\Phi(X_{\uuptau})\Bigr]
\qquad\forall\,x\in \sB^c\,,\quad\forall\,\Bar{v}\in\Usm^*\,.
\end{equation}
Comparing \cref{ET4.1F} and \cref{ET4.1I}, it follows that, given any
$\Bar{v}\in\Usm^*$,
we can scale $\Psi_{\Bar{v}}$ by a positive constant so that it touches $\Phi$ from
above at some point in $\Bar\sB$.
However, $\Bar{v}$ satisfies
\begin{equation*}
\Lg_{\Bar{v}} \Phi + c_{\Bar{v}}\,\Phi\;\ge\; \Lambda\str\,\Phi
\qquad \text{a.e.\ in\ }\Rd
\end{equation*}
by \cref{ET4.1H}.
Thus we have
\begin{equation*}
\Lg_{\Bar{v}} (\Psi_{\Bar{v}}-\Phi)
- (c_{\Bar{v}}-\Lambda\str)^{-}\,(\Psi_{\Bar{v}}-\Phi)\;\le\; 0\,
\qquad \text{a.e.\ in\ }\Rd\,,
\end{equation*}
and it follows by the strong maximum principle that
$\Phi=\Psi_{\Bar{v}}$ for all $\Bar{v}\in\Usm^*$.
Since $\bUsm\subset\Usm^*$ by part (a), it then follows by
\cref{ET4.1E} that $V=\Psi_{\Bar{v}}$ for all $\Bar{v}\in\Usm^*$.
Thus we have
\begin{equation*}
\Lg_{\Bar{v}}V+c_{\Bar{v}}V
\;=\; \Lg_{\Bar{v}}\Psi_{\Bar{v}}+c_{\Bar{v}}\Psi_{\Bar{v}}
\;=\; \lamstr_{\Bar{v}}(c_{\Bar{v}})\Psi_{\Bar{v}}
\;=\; \Lambda\str V
\;=\; \min_{u\in\Act}\;[\cL_u V + c(\cdot, u) V]\,.
\end{equation*}
This proves the verification of optimality result in part (b).

Suppose now that $\Tilde{V}\in\Cc^2(\Rd)$ is a positive solution of
\begin{equation}\label{ET4.1J}
\min_{u\in\Act}\;[\cL_u \Tilde{V} + c(\cdot, u) \Tilde{V}]\;=\;\Lambda\str\, \Tilde{V}
\quad\text{on\ }\Rd\,.
\end{equation}
Let $\Tilde{v}\in\Usm$ be a selector from the minimizer of \cref{ET4.1J}.
We have
$\lamstr_{\Tilde{v}}(c_{\Tilde{v}})=\sE^{\Tilde{v}}_x(c_{\Tilde{v}})\ge\Lambda\str$
for all $x\in\Rd$
by \cref{T3.2} and the definition of $\Lambda\str$,
and $\lamstr_{\Tilde{v}}(c_{\Tilde{v}})\le \Lambda\str$ by \cref{C2.1}.
Thus $\sE^{\Tilde{v}}_x(c_{\Tilde{v}})=\Lambda\str$ for all $x\in\Rd$,
which implies that $\Tilde{v}\in\Usm^*$.
Then $\Tilde{V}=\Psi_{\Tilde{v}}$ by the uniqueness of the latter.
Therefore, $\Tilde{V}=\Psi_{\Tilde{v}}=V$ by part (b).
This completes the proof.
\qed\end{proof}

As mentioned in \cref{R3.2} the existence of an
inf-compact $\ell$ in \cref{A4.1} is not possible
when $a$ and $b$ are bounded. So we consider the following alternative assumption.

\begin{assumption}\label{A4.2}
There exists a function $\Lyap\in\Sobl^{2,d}(\Rd)$,
such that $\inf_{\Rd}\Lyap>0$,
a compact set $\sK$, and positive constants $\kappa_0$ and $\gamma$,
satisfying
\begin{align*}
\max_{u\in\Act}\;\cL_u \Lyap (x)&\;\le\;
\kappa_0 \Ind_{\sK}(x) -\gamma \Lyap(x),\quad x\in\Rd\,,\\[3pt]
\limsup_{\abs{x}\to\infty}\;\sup_{u\in\Act}\,c(x,u) &\;<\; \gamma\,.
\end{align*}
\end{assumption}

A similar assumption is used in \cite{biswas-11a} where the author has obtained only
the existence of the solution $V$ to the HJB, and an optimal control.
Also it is shown in \cite{biswas-11a} that there exists a constant
$\gamma_1$, depending on $\gamma$, such that if $\norm{c}_\infty<\gamma_1$,
then \cref{ET4.2A} below has a solution. 
We improve these results substantially by proving uniqueness of the solution
$V$, and verification of optimality.

\begin{theorem}\label{T4.2}
Under \cref{A4.2}, there exists a positive solution $V\in\Cc^2(\Rd)$ satisfying
\begin{equation}\label{ET4.2A}
\min_{u\in\Act}\;[\cL_u V + c(\cdot, u) V]\;=\;\Lambda\str V.
\end{equation}
Let $\;\bUsm\subset\Usm$ be as in \cref{T4.1}.
Then \textup{(}a\textup{)} and \textup{(}b\textup{)} of \cref{T4.1} hold,
and \cref{ET4.2A} has
a unique positive solution in $\Cc^2(\Rd)$ up to a multiplicative constant.
\end{theorem}

\begin{proof}
Part (a) follows exactly as in the proof of \cref{T4.1}.

By \cref{T3.3,L2.7} for any $v\in\Usm$ there exists a unique eigenpair
$(\Psi_v,\lambda^*_v)$ for $\Lg_v^{c_v}$.
In addition,
\begin{equation*}
\Psi_v(x)\;=\;\Exp_x^v\Bigl[\E^{\int_0^{\uuptau_r}[c_v(X_s)-\lambda^*_v]\, \D{s}}\,
\Psi_v(X_{\uuptau_r})\Bigr]\,,
\quad x\in \Bar{B}_r^c\,.
\end{equation*}
The rest follows as in \cref{T4.1}.
\qed\end{proof}

\subsection{Continuity results}

It is known from \cite{book} that the set of relaxed stationary Markov controls
$\Usm$ is compactly metrizable (see also \cite{Borkar-topology} for a detailed
construction of this topology).
In particular $v_n\to v$ in $\Usm$ if and only if
\begin{equation*}
\int_{\Rd} f(x)\int_{\Act} g(x,u)\, v_n(\D{u}| x)\, \D{x}
\;\xrightarrow[n\to\infty]{}\;
\int_{\Rd} f(x)\int_{\Act} g(x,u)\, v(\D{u}| x)\, \D{x}
\end{equation*}
for all $f\in L^1(\Rd)\cap L^2(\Rd)$ and $g\in\Cc_b(\Rd\times\Act)$.
For $v\in\Usm$ we denote by $(\Psi_v,\lamstr_v(f))$ the principal eigenpair of the
operator $\Lg_v^{f}$, i.e.,
\begin{equation*}
\Lg_v\Psi_v(x) + f(x)\,\Psi_v(x)\;=\; \lamstr_v(f)\,\Psi_v(x),\quad \Psi_v(x)>0,
\quad x\in\Rd\,.
\end{equation*}

When $f=c_v$, we occasionally
drop the dependence on $c_v$ and denote the eigenvalue
as $\lamstr_v=\lamstr_v(c_v)$. 
The next result concerns the continuity of $\lamstr_v$ with respect to
stationary Markov controls, and
extends the result in \cite[Proposition~9.2]{Berestycki-15}.
The continuity result in \cite[Proposition~9.2]{Berestycki-15} is established
with respect to the $L^\infty$ norm convergence of the
coefficients, whereas \cref{T4.3} that
follows asserts continuity under a much weaker topology.

\begin{theorem}\label{T4.3}
Assume one of the following.
\begin{enumerate}[(i)]
\item
\cref{A4.1} holds,
and $c\in\cC_{\beta\ell}$ for some $\beta\in(0,1)$.

\item
\cref{A4.2} holds.
\end{enumerate}
Then the map $v\mapsto \lamstr_v$ is continuous.
\end{theorem}

\begin{proof}
We demonstrate the result under (i).
For case (ii) the proof is analogous.
Let $v_n\to v$ in the topology of Markov controls.
Let $(\Psi_{n}, \lamstr_{n})$ be the principal eigenpair which satisfies
\begin{equation}\label{ET4.3A}
\Lg_{v_n}\Psi_{n}(x) + c_{v_n}(x)\,\Psi_n(x)\;=\; \lamstr_n\, \Psi_n(x), \quad x\in\Rd,
\quad \text{and}\quad
\lamstr_n=\sE^{v_n}(c_{v_n})\,,
\end{equation}
where the equality $\lamstr_n=\sE^{v_n}(c_{v_n})$ is a consequence of \cref{T3.2,P3.1}.
It is obvious that $\lamstr_n\ge 0$ for all $n$.

Since
$\ell(\cdot)-\max_{u\in\Act}c(\cdot, u)$ is inf-compact, we can find a constant
$\kappa_1$ such that
$\max_{u\in\Act}c(x,u)\le \kappa_1 + \ell(x)$.
Recall that 
$\sE^v(\ell)<\Bar{\kappa}_\circ$ for all $v\in\Usm$
(as shown in the paragraph after \cref{A4.1}),
and this implies that $\lamstr_n\le \kappa_1+\Bar{\kappa}_\circ$ for all $n$.
Thus $\{\lamstr_n\, \colon n\ge 1\}$ is bounded.
Therefore, passing to a subsequence we may assume that $\lamstr_n\to\lamstr$ as
$n\to\infty$.
To complete the proof we only need to show that $\lamstr=\lamstr_v$.
Since $\Psi_n(0)=1$ for all $n$, and the coefficients 
$b_{v_n}$, and $c_{v_n}$ are uniformly locally bounded,
applying Harnack's inequality and Sobolev's estimate we can find
$\Psi\in\Sobl^{2,p}(\Rd)$, $p\ge 1$, such that $\Psi_n\to\Psi$
weakly in $\Sobl^{2,p}(\Rd)$.
Therefore, by \cite[Lemma~2.4.3]{book} and \cref{ET4.3A}, we obtain
\begin{equation}\label{ET4.3B}
\Lg\Psi(x) + c_v(x)\,\Psi(x)\;=\; \lamstr\, \Psi(x), \quad x\in\Rd,\quad \Psi>0\,.
\end{equation}
By \cref{C2.1} we have $\lamstr\ge\lamstr_v$.

Let  $\sB\supset\sK$ be an open ball such that 
$\abs{c(x,u)-\lambda^*}\le \beta \ell(x)$ for all $(x,u)\in\sB^c\times\Act$,
and $R>0$ be large enough so that $\sB\subset B_R$.
Let $\uuptau=\uptau(\sB^c)$.
Applying the It\^{o}--Krylov
formula to \cref{ET4.3B}, we obtain 
\begin{equation}\label{ET4.3C}
\Psi(x)\;=\; \Exp^{v}_x
\Bigl[\E^{\int_0^{\uuptau\wedge\uptau_R\wedge T}[c_{v}(X_s)-\lamstr]\, \D{s}}\, 
\Psi(X_{\uuptau\wedge\uptau_R\wedge T})\Bigr]\,, \quad x\in\sB^c\cap B_R\,,
\end{equation}
for any $T>0$.
Since 
\begin{align}\label{ET4.3D}
\Exp^{v}_x\Bigl[\E^{\int_0^{\uuptau}[c_{v}(X_s)-\lamstr]\, \D{s}}\, \Bigr]
&\;\le\;
\Exp^{v}_x\Bigl[\E^{\int_0^{\uuptau}\beta\ell(X_s)\, \D{s}}\, \Bigr]
\nonumber\\[5pt]
&\;\le\; \biggl(\Exp^{v}_x\Bigl[\E^{\int_0^{\uuptau}\ell(X_s)\, \D{s}}\, \Bigr]
\biggr)^{\beta}
\;<\;\infty
\quad \text{for}\; x\in\sB^c,
\end{align}
and $\Psi$ in bounded in $\sB^c\cap B_R$, for every fixed $R$,
letting $T\to\infty$ in \cref{ET4.3C} we have
\begin{equation}\label{ET4.3E}
\Psi(x)\;=\;
\Exp^{v}_x\Bigl[\E^{\int_0^{\uuptau\wedge\uptau_R}[c_{v}(X_s)-\lamstr]\, \D{s}}\, 
\Psi(X_{\uuptau\wedge\uptau_R})\Bigr]\,, \quad x\in\sB^c\cap B_R\,.
\end{equation}

\Cref{ET4.3D} which also holds, possibly for a larger ball $\sB$,
if we replace $v$ and $\lamstr$ with $v_n$ and $\lamstr_n$, respectively,
shows that, for some constant $\Tilde\kappa$, we have
$\Psi_n(x) \le \Tilde\kappa \bigl(\Lyap(x)\bigr)^\beta$ for all $n\in\NN$,
and $x\in\sB^c$.
Therefore, $\Psi(x) \le \Tilde\kappa \bigl(\Lyap(x)\bigr)^\beta$
for all $x\in\sB^c$.

We write
\begin{equation}\label{ET4.3F}
\Exp_x^v\Bigl[\E^{\int_0^{\uuptau\wedge\uptau_R}\ell(X_s)\, \D{s}}\,\Bigr]
\;=\;
\Exp_x^v\Bigl[\E^{\int_0^{\uuptau}\ell(X_s)\, \D{s}}\,
\Ind_{\{\uuptau<\uptau_R\}}\Bigr]
+ \Exp_x^v\Bigl[\E^{\int_0^{\uptau_R}\ell(X_s)\, \D{s}}\,
\Ind_{\{\uptau_R<\uuptau\}}\Bigr]\,.
\end{equation}
The left hand side of \cref{ET4.3F} and the first term on the right
hand side both converge to
$\Exp_x^v\bigl[\E^{\int_0^{\uuptau}\ell(X_s)\, \D{s}}\bigr]$ as $R\to\infty$,
by monotone convergence.
Thus we have
\begin{equation}\label{ET4.3G}
\Exp_x^v\Bigl[\E^{\int_0^{\uptau_R}\ell(X_s)\, \D{s}}\,
\Ind_{\{\uptau_R<\uuptau\}}\Bigr]\;\xrightarrow[R\to\infty]{}\; 0\,.
\end{equation}

On the other hand \cref{A4.1} implies that
\begin{equation*}
\Exp_x\Bigl[\E^{\int_0^{\uptau_R}\ell(X_s)\, \D{s}}\,
\Lyap(X_{\uptau_R})\,\Ind_{\{\uptau_R<\uuptau\}}\Bigr]
\;\le\; \Lyap(x)
\qquad \forall\,x\in B_R\setminus\sB^c\,,\ \forall\,R>0\,.
\end{equation*}
We proceed as in the proof of \cref{T2.5}.
Let $\Gamma(R,m) \df \{x\in\partial B_R\colon \Psi(x)\ge m\}$
for $m\ge 1$.
Since $\Psi \le \Tilde\kappa \Lyap^\beta$ on $\sB^c$,
we have
$\Lyap^{\beta-1}\le \bigl(\tfrac{\Psi}{\Tilde\kappa}\bigr)^{1-\frac{1}{\beta}}$
on $\sB^c$, and, therefore,
\begin{equation}\label{ET4.3new}
\Psi\,\Ind_{\Gamma(R,m)}\;\le\; \Tilde\kappa \Lyap\,\Lyap^{\beta-1}\,\Ind_{\Gamma(R,m)}
\;\le\;\Tilde\kappa^{\frac{1}{\beta}}\,m^{1-\frac{1}{\beta}}\,\Lyap
\quad\text{on\ }\partial B_R\,.
\end{equation}
Thus, using \cref{ET4.3new}, we obtain
\begin{align}\label{ET4.3H}
\Exp^{v}_x\Bigl[\E^{\int_0^{\uptau_R}[c_{v}(X_s)-\lambda^*]\, \D{s}}\, 
\Psi(X_{\uptau_R})\,\Ind_{\{\uptau_R<\uuptau\}}\Bigr]
& \le\;
m\Exp^{v}_x\Bigl[\E^{\int_0^{\uptau_R}\ell(X_s)\, \D{s}}\, 
\Ind_{\{\uptau_R<\uuptau\}}\Bigr]\nonumber\\[3pt]
&\mspace{50mu}+
\Exp^{v}_x\Bigl[\E^{\int_0^{\uptau_R}\ell(X_s)\, \D{s}}\, 
\Psi(X_{\uptau_R})\,\Ind_{\Gamma(R,m)}(X_{\uptau_R})\,\Ind_{\{\uptau_R<\uuptau\}}\Bigr]
\nonumber\\[5pt]
&\le\;
m\Exp^{v}_x\Bigl[\E^{\int_0^{\uptau_R}\ell(X_s)\, \D{s}}\, 
\Ind_{\{\uptau_R<\uuptau\}}\Bigr]\nonumber\\[3pt]
&\mspace{50mu}+
\Tilde\kappa^{\frac{1}{\beta}}\,m^{1-\frac{1}{\beta}}
\,\Exp^{v}_x\Bigl[\E^{\int_0^{\uptau_R}\ell(X_s)\, \D{s}}\, 
\Lyap(X_{\uptau_R})\Ind_{\{\uptau_R<\uuptau\}}\Bigr]
\nonumber\\[5pt]
&\le\; m\Exp^{v}_x\Bigl[\E^{\int_0^{\uptau_R}\ell(X_s)\, \D{s}}\, 
\Ind_{\{\uptau_R<\uuptau\}}\Bigr]
+ \Tilde\kappa^{\frac{1}{\beta}}\,m^{1-\frac{1}{\beta}}\,\Lyap(x)\,,
\end{align}
and by first letting $R\to\infty$, using
\cref{ET4.3G}, and then $m\to\infty$, it follows
that the left hand side of \cref{ET4.3H} vanishes as $R\to\infty$.
Therefore,
letting $R\to\infty$ in \cref{ET4.3E}, we obtain
\begin{equation*}
\Psi(x)\;=\; \Exp^{v}_x\Bigl[e^{\int_0^{\uuptau}[c_{v}(X_s)-\lamstr]\, \D{s}}\, 
\Psi(X_{\uuptau})\Bigr]\,, \quad x\in\sB^c\,.
\end{equation*}
It then follows by \cref{C2.3} that $\lamstr=\lamstr_v$, and this completes
the proof.
\qed\end{proof}

\begin{remark}\label{R4.1}
Following the proof of \cref{T4.3} we can obtain the following continuity
result which should be compared with \cite[Proposition~9.2\,(ii)]{Berestycki-15}.
Consider a sequence of operators $\Lg_n^{f_n}$ with coefficients $(a_n, b_n, f_n)$,
where $b_n$, $f_n$ are locally bounded uniformly in $n$,
and $\inf_n(\inf_{\Rd} f_n)>-\infty$.
The coefficients $a_n$ and $b_n$ are
assumed to satisfy (A1)--(A3) uniformly in $n$.
Assume that $a_n\to a$ in $\Cc_{\mathrm{loc}}(\Rd)$, and
$b_n\to b$ and $f_n\to f$ weakly in $L^1_{\mathrm{loc}}(\Rd)$.
Moreover we suppose that one of the following hold.
\begin{enumerate}[(a)]
\item
There exists an inf-compact function $\ell\in\Cc(\Rd)$ and
$\Lyap\in \Sobl^{2,d}(\Rd)$, with $\inf_{\Rd}\Lyap>0$, such that
$\Lg_n \Lyap \;\le\; \Bar{\kappa}\, \Ind_{\sK} - \ell\Lyap$ 
a.e.\ on $\Rd$
for some constant $\Bar{\kappa}$, and a compact set $\sK$.
In addition, $\beta\ell-\sup_n f_n$ is inf-compact for some $\beta\in (0,1)$.

\item
The sequence $\Lg_n$ satisfies \cref{ET3.3A}
for all $n$, 
$\lim_{n\to\infty}\norm{f^-_n}_\infty=\norm{f^-}_\infty$, and 
$$\limsup_{\abs{x}\to\infty}\, \sup_n f_n(x)+\norm{f^-_n}_\infty\;<\; \gamma\,.$$
\end{enumerate}
Then the principal eigenvalue $\lamstr(f_n)$ converges to $\lamstr(f)$ as $n\to\infty$.
\end{remark}

As an application of \cref{T4.3} we have the following existence result for the
risk-sensitive control problem
under (Markovian) risk-sensitive type constraints.

\begin{theorem}
Assume one of the following.
\begin{enumerate}[(i)]
\item
\cref{A4.1} holds,
and $c, r_1, \dotsc, r_m\in\cC_{\beta\ell}$ for some $\beta\in(0,1)$.

\item
\cref{A4.2} holds, and $r_1, \dotsc, r_m\in\mathfrak{C}$ 
satisfy
\begin{equation*}
 \max_{i=1,\dotsc,m}\;
 \Bigl\{\limsup_{\abs{x}\to\infty}\,\max_{u\in\Act}\;r_i(x,u) \Bigr\}
 \;<\; \gamma\,.
 \end{equation*}
\end{enumerate}
In addition, suppose that $K_i$, $i=1,\dotsc,m$, are closed subsets of $\RR$,
and that there exists $\Hat{v}\in\Usm$ such that
$\sE^{\Hat{v}}(r_{i,\Hat{v}})\in K_i$ for all $i$,
where we use the usual notation $r_{i,v}(x)\df r_i\bigl(x,v(x)\bigr)$.

Then the following constrained
minimization problem admits an optimal control in $\Usm$
\begin{equation*}
\text{minimize over $v\in\Usm$}:\ \ \sE^v(c_v),
\quad \text{subject to\ \ }
\sE^v(r_{i,v})\in K_i\,,\ i=1,\dotsc, m\,.
\end{equation*}
\end{theorem}

\begin{proof}
Let $v_n\in\Usm$ be a sequence of controls
along which the constraints are met,
and $\sE^{v_n}(c_{v_n})$ converges to its infimum.
Since $\Usm$ is compact under the topology of Markov controls,
we may assume, without loss of generality,
that $v_n$ converges to some $\Bar{v}\in\Usm$ as $n\to\infty$.
By \cref{T4.3} we know that $v\mapsto \lamstr_v(c_v)$,
and $v\mapsto\lamstr_v(r_{i,v})$, $i=1,\dotsc,m$, are continuous maps,
and that
$\sE^v(c_v)=\lamstr_v(c_v)$, and 
$\sE^v(r_{i,v})=\lamstr_v(r_{i,v})$ for $i=1,\dotsc,m$.
It follows that the constraints are met at $\Bar{v}$.
Therefore,
$\Bar{v}$ is an optimal Markov control for the constrained problem.
\qed\end{proof}

Another application of \cref{T4.3} is a following characterization of
$\lamstr$ which provides a 
positive answer to \cite[Conjecture~1.8]{Berestycki-15} for a certain
class of $a, b$ and $f$.
In \cref{T4.5} below, we consider the uncontrolled generator $\Lg$ in
\cref{S3}. Let us introduce the following definition from \cite{Berestycki-15}
\begin{equation*}
\lambda^{\prime}(f)\;=\;\sup\;\bigl\{\lambda\, \colon\;
\exists\, \varphi\in\Sobl^{2,d}(\Rd)\cap L^\infty(\Rd), \varphi>0, \; \Lg\varphi
+ (f-\lambda)\varphi\;\ge\; 0\text{\ a.e. in\ }\Rd\bigr\}\,.
\end{equation*}
Recall the definition of $\lambda''$ in \cref{E-l''}.
From \cite[Theorem~1.7]{Berestycki-15},  under (A1)--(A2), we have
$\lamstr(f)\le\lambda^{\prime}(f)\le \lambda^{\prime\prime}(f)$
whenever $f$ is bounded above.
It is conjectured in 
\cite[Conjecture~1.8]{Berestycki-15} that for bounded $a$, $b$, and $f$,
one has $\lambda^{\prime}(f)= \lambda^{\prime\prime}(f)$.
It should be noted from \cref{E3.1} that $\lamstr(f)$
could be strictly smaller than $\lambda^{\prime\prime}(f)$.
The following result complements those in \cite[Theorems~1.7 and~1.9]{Berestycki-15}.

\begin{theorem}\label{T4.5}
For a potential $f$
the following are true.
\begin{enumerate}[(i)]
\item
Suppose that $\sE_x(f)<\infty$. Then under (A1)--(A3) we have
$$\lamstr(f)\,\le\,\lambda^{\prime}(f)\,\le\,\sE_x(f)\,\le\,
\lambda^{\prime\prime}(f)\,.$$

\item
Let $\Lg$, $\Lyap$ and $\gamma$ satisfy \cref{ET3.3A}, and suppose that
$\sup_{\Rd}(f+\norm{f^-}_\infty)<\gamma$.
Then
$\lamstr(f)\;=\;\lambda''(f)$.

\item
Let $\Lg$, $\Lyap$ and $\ell$ satisfy \cref{EA4.1A}, and suppose
that $\beta\ell-f$ is inf-compact
for some $\beta\in(0, 1)$.
Then $\lamstr(f)=\lambda^{\prime\prime}(f)$.
\end{enumerate}
\end{theorem}

\begin{proof}
We first show (i).
By \cite[Theorem~1.7\,(ii)]{Berestycki-15} we have $\lamstr(f)\le\lambda^{\prime}(f)$.
Let 
$\varphi\in\Sobl^{2,d}(\Rd)\cap L^\infty(\Rd)$, $\varphi>0$, be such that 
$$\Lg\varphi+ (f-\lambda)\varphi\;\ge\; 0\,.$$
Recall that $\uptau_n$ is the exit time from the open ball $B_n(0)$.
Therefore, applying the It\^{o}--Krylov formula,  we obtain
\begin{equation}\label{ET4.5A}
\varphi(x)\;\le\; \Exp_x\Bigl[e^{\int_0^{\uptau_n\wedge T} [f(X_s)-\lambda]\, \D{s}}\,
\varphi(X_{\uptau_n\wedge T})\Bigr]
\;\le\; \Bigl(\sup_{\Rd}\varphi\Bigr)\, \Exp_x\Bigl[e^{\int_0^{\uptau_n\wedge T}
[f(X_s)-\lambda]\, \D{s}}\Bigr]\;,
\quad T\ge 0\,.
\end{equation}
Since $\sE_x(f)$ is finite, letting $n\to\infty$ in \cref{ET4.5A},
taking logarithms on both sides, dividing by $T$ and then letting
$T\to\infty$ we obtain $\lambda\le\sE_x(f)$.
This implies $\lambda^\prime(f)\le\sE_x(f)$.
Now suppose 
$\varphi\in\Sobl^{2,d}(\Rd)$, with $\inf_\Rd\,\varphi>0$, satisfies
$$\Lg\varphi+ (f-\lambda)\varphi\;\le\; 0\,.$$
Repeating the analogous calculation as above, we obtain $\lambda\ge \sE_x(f)$,
which implies that $\sE_x(f)\le \lambda^{\prime\prime}(f)$.

Next we prove (ii). Since $\lamstr(f+c)=\lamstr(f) +c$ for any constant $c$,
we may replace $f$ by $f+\norm{f^-}_\infty$.
Therefore, $f$ is non-negative and $\norm{f}_\infty<\gamma$.
By (i) we have $\lamstr(f)\le\lambda''(f)$.
Let $\chi_n\colon\Rd\to[0, 1]$ be a cut-off function such that
$\chi_n(x)=1$ for $\abs{x}\le n$, and $\chi_n(x)=0$ for $\abs{x}\ge n+1$.
Define $f_n\df\chi_n\, f + (1-\chi_n)\norm{f}_\infty$.
Let $\bigl(\Psi^*_n,\lamstr(f_n)\bigr)$ denote the principal
eigenpair of $\Lg^{f_n}$.
By \cref{R4.1} we have $\lamstr(f_n)\to\lamstr(f)$ as $n\to\infty$.
Thus to complete the proof it is enough to show that $\inf_{\Rd}\Psi_n>0$,
which implies that $\lamstr(f_n)=\lambda''(f_n)\ge\lambda''(f)$ for all $n$, and thus 
$\lamstr(f)\ge\lambda''(f)$.
Note that $\lamstr(f_n)\le\sE(f_n)\le \norm{f}_\infty$ for all $n$. Now fix $n$
and let $\uuptau_n$ be the first hitting time to the ball $B_n$.
Then applying the It\^o--Krylov formula to
\begin{equation*}
\Lg\Psi^*_n + \bigl(f_n-\lamstr(f_n)\bigr)\, \Psi^*_n\,=\,0
\end{equation*} 
together with Fatou's lemma, we have
\begin{equation*}
\min_{z\in B_{n+1}}\Psi^*_n(z)\;\le\;
\Exp_x\Bigl[\E^{\int_0^{\uuptau_n} [f_n(X_s)-\lamstr(f_n)]\, \D{s}}\,
\Psi^*_n(X_{\uuptau_n})\Bigr]\;\le\; \Psi^*_n(x)
\end{equation*}
for all $x\in \overline{B}^c_{n+1}(0)$.
Hence $\inf_{\Rd}\Psi^*_n>0$ which completes the proof.

The proof of (iii) is completely analogous to the proof of part (ii).
Since $\beta\ell-f^+$ is inf-compact, we can find $g\colon\Rd\to\RR_+$,
such that $\lim_{\abs{x}\to\infty} g(x)=\infty$,
and $\beta\ell-f^+ - g$ is inf-compact.
We let $f_n=\chi_n f + (1-\chi_n)(g + f^+)$.
Note that 
$$\inf_n\,(\beta\ell-f_n)\;=\;\inf_{n\in\NN}\,
\bigl(\chi_n (\beta\ell-f)+(1-\chi_n)(\beta\ell-f^+-g)\bigr)$$
is inf-compact.
On the other hand, $f_n\ge f$ for all $n$.
The rest follows as part (ii).
\qed\end{proof}

\section*{Acknowledgements}The research of Ari Arapostathis
was supported in part by the Army Research Office through grant W911NF-17-1-001,
in part by the National Science Foundation through grant DMS-1715210,
and in part by the Office of Naval Research through grant N00014-16-1-2956.
The research of Anup Biswas was supported in part by an INSPIRE faculty fellowship,
and a DST-SERB grant EMR/2016/004810.

\def\polhk#1{\setbox0=\hbox{#1}{\ooalign{\hidewidth
  \lower1.5ex\hbox{`}\hidewidth\crcr\unhbox0}}}


\begin{thebibliography}{10}
\expandafter\ifx\csname url\endcsname\relax
  \def\url#1{\texttt{#1}}\fi
\expandafter\ifx\csname urlprefix\endcsname\relax\def\urlprefix{URL }\fi
\expandafter\ifx\csname href\endcsname\relax
  \def\href#1#2{#2} \def\path#1{#1}\fi

\bibitem{Krein-Rutman}
M.~G. Kre{\u\i}n, M.~A. Rutman, Linear operators leaving invariant a cone in a
  {B}anach space, Amer. Math. Soc. Translation 1950~(26) (1950) 128.

\bibitem{Pinsky}
R.~G. Pinsky, Positive harmonic functions and diffusion, Vol.~45 of Cambridge
  Studies in Advanced Mathematics, Cambridge University Press, Cambridge, 1995.

\bibitem{Berestycki-94}
H.~Berestycki, L.~Nirenberg, S.~R.~S. Varadhan, The principal eigenvalue and
  maximum principle for second-order elliptic operators in general domains,
  Comm. Pure Appl. Math. 47~(1) (1994) 47--92.
\newblock \href {http://dx.doi.org/10.1002/cpa.3160470105}
  {\path{doi:10.1002/cpa.3160470105}}.

\bibitem{Quaas-08a}
A.~Quaas, B.~Sirakov, Principal eigenvalues and the {D}irichlet problem for
  fully nonlinear elliptic operators, Adv. Math. 218~(1) (2008) 105--135.
\newblock \href {http://dx.doi.org/10.1016/j.aim.2007.12.002}
  {\path{doi:10.1016/j.aim.2007.12.002}}.

\bibitem{Berestycki-15}
H.~Berestycki, L.~Rossi, Generalizations and properties of the principal
  eigenvalue of elliptic operators in unbounded domains, Comm. Pure Appl. Math.
  68~(6) (2015) 1014--1065.
\newblock \href {http://dx.doi.org/10.1002/cpa.21536}
  {\path{doi:10.1002/cpa.21536}}.

\bibitem{Furusho-Ogura}
Y.~Furusho, Y.~Ogura, On the existence of bounded positive solutions of
  semilinear elliptic equations in exterior domains, Duke Math. J. 48~(3)
  (1981) 497--521.
\newblock \href {http://dx.doi.org/10.1215/S0012-7094-81-04828-6}
  {\path{doi:10.1215/S0012-7094-81-04828-6}}.

\bibitem{Pinchover-88}
Y.~Pinchover, On positive solutions of second-order elliptic equations,
  stability results, and classification, Duke Math. J. 57~(3) (1988) 955--980.
\newblock \href {http://dx.doi.org/10.1215/S0012-7094-88-05743-2}
  {\path{doi:10.1215/S0012-7094-88-05743-2}}.

\bibitem{Berestycki-06}
H.~Berestycki, L.~Rossi, On the principal eigenvalue of elliptic operators in
  {$\mathbb{R}^N$} and applications, J. Eur. Math. Soc. (JEMS) 8~(2) (2006)
  195--215.
\newblock \href {http://dx.doi.org/10.4171/JEMS/47}
  {\path{doi:10.4171/JEMS/47}}.

\bibitem{Kaise-06}
H.~Kaise, S.-J. Sheu, On the structure of solutions of ergodic type {B}ellman
  equation related to risk-sensitive control, Ann. Probab. 34~(1) (2006)
  284--320.
\newblock \href {http://dx.doi.org/10.1214/009117905000000431}
  {\path{doi:10.1214/009117905000000431}}.

\bibitem{ari-anup}
A.~{Arapostathis}, A.~{Biswas}, Infinite horizon risk-sensitive control of
  diffusions without any blanket stability assumptions, Stochastic Process.
  Appl. (in press).
\newblock \href {http://dx.doi.org/10.1016/j.spa.2017.08.001}
  {\path{doi:10.1016/j.spa.2017.08.001}}.

\bibitem{Engel-Nagel}
K.-J. Engel, R.~Nagel, One-parameter semigroups for linear evolution equations,
  Vol. 194 of Graduate Texts in Mathematics, Springer-Verlag, New York, 2000.

\bibitem{Donsker-75}
M.~D. Donsker, S.~R.~S. Varadhan, On a variational formula for the principal
  eigenvalue for operators with maximum principle, Proc. Nat. Acad. Sci. U.S.A.
  72 (1975) 780--783.

\bibitem{Donsker-76}
M.~D. Donsker, S.~R.~S. Varadhan, Asymptotic evaluation of certain {M}arkov
  process expectations for large time. {III}, Comm. Pure Appl. Math. 29~(4)
  (1976) 389--461.
\newblock \href {http://dx.doi.org/10.1002/cpa.3160290405}
  {\path{doi:10.1002/cpa.3160290405}}.

\bibitem{Kontoyiannis-02}
I.~Kontoyiannis, S.~P. Meyn, Spectral theory and limit theorems for
  geometrically ergodic {M}arkov processes, Ann. Appl. Probab. 13~(1) (2003)
  304--362.
\newblock \href {http://dx.doi.org/10.1214/aoap/1042765670}
  {\path{doi:10.1214/aoap/1042765670}}.

\bibitem{Wu-94}
L.~Wu, Feynman-{K}ac semigroups, ground state diffusions, and large deviations,
  J. Funct. Anal. 123~(1) (1994) 202--231.
\newblock \href {http://dx.doi.org/10.1006/jfan.1994.1087}
  {\path{doi:10.1006/jfan.1994.1087}}.

\bibitem{Ichihara-13b}
N.~Ichihara, Criticality of viscous {H}amilton-{J}acobi equations and
  stochastic ergodic control, J. Math. Pures Appl. (9) 100~(3) (2013) 368--390.
\newblock \href {http://dx.doi.org/10.1016/j.matpur.2013.01.005}
  {\path{doi:10.1016/j.matpur.2013.01.005}}.

\bibitem{Ichihara-11}
N.~Ichihara, Recurrence and transience of optimal feedback processes associated
  with {B}ellman equations of ergodic type, SIAM J. Control Optim. 49~(5)
  (2011) 1938--1960.
\newblock \href {http://dx.doi.org/10.1137/090772678}
  {\path{doi:10.1137/090772678}}.

\bibitem{Ichihara-15}
N.~Ichihara, The generalized principal eigenvalue for
  {H}amilton-{J}acobi-{B}ellman equations of ergodic type, Ann. Inst. H.
  Poincar\'e Anal. Non Lin\'eaire 32~(3) (2015) 623--650.
\newblock \href {http://dx.doi.org/10.1016/j.anihpc.2014.02.003}
  {\path{doi:10.1016/j.anihpc.2014.02.003}}.

\bibitem{Barles-16}
G.~Barles, J.~Meireles, On unbounded solutions of ergodic problems in
  {$\mathbb{R}^m$} for viscous {H}amilton-{J}acobi equations, Comm. Partial
  Differential Equations 41~(12) (2016) 1985--2003.
\newblock \href {http://dx.doi.org/10.1080/03605302.2016.1244208}
  {\path{doi:10.1080/03605302.2016.1244208}}.

\bibitem{Bhatt-Bor-96}
A.~G. Bhatt, V.~S. Borkar, Occupation measures for controlled {M}arkov
  processes: characterization and optimality, Ann. Probab. 24~(3) (1996)
  1531--1562.
\newblock \href {http://dx.doi.org/10.1214/aop/1065725192}
  {\path{doi:10.1214/aop/1065725192}}.

\bibitem{Stockbridge-90}
R.~H. Stockbridge, Time-average control of martingale problems: a linear
  programming formulation, Ann. Probab. 18~(1) (1990) 206--217.
\newblock \href {http://dx.doi.org/10.1214/aop/1176990945}
  {\path{doi:10.1214/aop/1176990945}}.

\bibitem{Kaise-04}
H.~Kaise, S.-J. Sheu, Evaluation of large time expectations for diffusion
  processes, preprint (2004).

\bibitem{Fleming-95}
W.~H. Fleming, W.~M. McEneaney, Risk-sensitive control on an infinite time
  horizon, SIAM J. Control Optim. 33~(6) (1995) 1881--1915.
\newblock \href {http://dx.doi.org/10.1137/S0363012993258720}
  {\path{doi:10.1137/S0363012993258720}}.

\bibitem{Nagai-96}
H.~Nagai, Bellman equations of risk-sensitive control, SIAM J. Control Optim.
  34~(1) (1996) 74--101.
\newblock \href {http://dx.doi.org/10.1137/S0363012993255302}
  {\path{doi:10.1137/S0363012993255302}}.

\bibitem{Menaldi-05}
J.-L. Menaldi, M.~Robin, Remarks on risk-sensitive control problems, Appl.
  Math. Optim. 52~(3) (2005) 297--310.
\newblock \href {http://dx.doi.org/10.1007/s00245-005-0829-y}
  {\path{doi:10.1007/s00245-005-0829-y}}.

\bibitem{biswas-11a}
A.~Biswas, An eigenvalue approach to the risk sensitive control problem in near
  monotone case, Systems Control Lett. 60~(3) (2011) 181--184.
\newblock \href {http://dx.doi.org/10.1016/j.sysconle.2010.12.002}
  {\path{doi:10.1016/j.sysconle.2010.12.002}}.

\bibitem{Biswas-10}
A.~Biswas, V.~S. Borkar, K.~Suresh~Kumar, Risk-sensitive control with near
  monotone cost, Appl. Math. Optim. 62~(2) (2010) 145--163.
\newblock \href {http://dx.doi.org/10.1007/s00245-009-9096-7}
  {\path{doi:10.1007/s00245-009-9096-7}}.

\bibitem{biswas-11}
A.~Biswas, Risk sensitive control of diffusions with small running cost, Appl.
  Math. Optim. 64~(1) (2011) 1--12.
\newblock \href {http://dx.doi.org/10.1007/s00245-010-9127-4}
  {\path{doi:10.1007/s00245-010-9127-4}}.

\bibitem{Basu-Ghosh}
A.~Basu, M.~K. Ghosh, Zero-sum risk-sensitive stochastic differential games,
  Math. Oper. Res. 37~(3) (2012) 437--449.
\newblock \href {http://dx.doi.org/10.1287/moor.1120.0542}
  {\path{doi:10.1287/moor.1120.0542}}.

\bibitem{BS-arxiv}
A.~{Biswas}, S.~{Saha}, \href{https://arxiv.org/abs/1704.02689}{Zero-sum
  stochastic differential game with risk-sensitive cost}, Appl. Math. Optim.
  (to appear).
\newline\urlprefix\url{https://arxiv.org/abs/1704.02689}

\bibitem{Ghosh-16}
M.~K. Ghosh, K.~S. Kumar, C.~Pal, Zero-sum risk-sensitive stochastic games for
  continuous time {M}arkov chains, Stoch. Anal. Appl. 34~(5) (2016) 835--851.
\newblock \href {http://dx.doi.org/10.1080/07362994.2016.1180995}
  {\path{doi:10.1080/07362994.2016.1180995}}.

\bibitem{Zhang-05}
X.~Zhang, Strong solutions of {SDES} with singular drift and {S}obolev
  diffusion coefficients, Stochastic Process. Appl. 115~(11) (2005) 1805--1818.
\newblock \href {http://dx.doi.org/10.1016/j.spa.2005.06.003}
  {\path{doi:10.1016/j.spa.2005.06.003}}.

\bibitem{Krylov}
N.~V. Krylov, Controlled diffusion processes, Vol.~14 of Applications of
  Mathematics, Springer-Verlag, New York, 1980.

\bibitem{book}
A.~Arapostathis, V.~S. Borkar, M.~K. Ghosh, Ergodic control of diffusion
  processes, Vol. 143 of Encyclopedia of Mathematics and its Applications,
  Cambridge University Press, Cambridge, 2012.

\bibitem{GilTru}
D.~Gilbarg, N.~S. Trudinger, Elliptic partial differential equations of second
  order, 2nd Edition, Vol. 224 of Grundlehren der Mathematischen
  Wissenschaften, Springer-Verlag, Berlin, 1983.

\bibitem{LiSh-I}
R.~S. Liptser, A.~N. Shiryayev, Statistics of random processes. {I},
  Springer-Verlag, New York, 1977, general theory, Translated by A. B. Aries,
  Applications of Mathematics, Vol. 5.

\bibitem{AA-Harnack}
A.~Arapostathis, M.~K. Ghosh, S.~I. Marcus, Harnack's inequality for
  cooperative weakly coupled elliptic systems, Comm. Partial Differential
  Equations 24~(9-10) (1999) 1555--1571.
\newblock \href {http://dx.doi.org/10.1080/03605309908821475}
  {\path{doi:10.1080/03605309908821475}}.

\bibitem{Hasminskii}
R.~Z. Has$^{_{^{\prime}}}\!$minski\u{\i}, Stochastic stability of differential
  equations, Sijthoff \& Noordhoff, The Netherlands, 1980.

\bibitem{ABG-arxiv}
A.~{Arapostathis}, A.~{Biswas}, D.~{Ganguly},
  \href{https://arxiv.org/abs/1708.09640}{Some {L}iouville-type results for
  eigenfunctions of elliptic operators}, ArXiv e-prints 1708.09640.
\newline\urlprefix\url{https://arxiv.org/abs/1708.09640}

\bibitem{Kunita}
H.~Kunita, Stochastic flows and stochastic differential equations, Vol.~24 of
  Cambridge Studies in Advanced Mathematics, Cambridge University Press,
  Cambridge, 1990.

\bibitem{Bogachev-01}
V.~I. Bogachev, N.~V. Krylov, M.~R{\"o}ckner, On regularity of transition
  probabilities and invariant measures of singular diffusions under minimal
  conditions, Comm. Partial Differential Equations 26~(11-12) (2001)
  2037--2080.
\newblock \href {http://dx.doi.org/10.1081/PDE-100107815}
  {\path{doi:10.1081/PDE-100107815}}.

\bibitem{Gyongy-96}
I.~Gy{\"o}ngy, N.~Krylov, Existence of strong solutions for {I}t\^o's
  stochastic equations via approximations, Probab. Theory Related Fields
  105~(2) (1996) 143--158.
\newblock \href {http://dx.doi.org/10.1007/BF01203833}
  {\path{doi:10.1007/BF01203833}}.

\bibitem{Yoshimura-06}
Y.~Yoshimura, A note on demi-eigenvalues for uniformly elliptic {I}saacs
  operators, Viscosity Solution Theory of Differential Equations and its
  Developments (2006) 106--114.

\bibitem{Borkar-topology}
V.~S. Borkar, A topology for {M}arkov controls, Appl. Math. Optim. 20~(1)
  (1989) 55--62.
\newblock \href {http://dx.doi.org/10.1007/BF01447645}
  {\path{doi:10.1007/BF01447645}}.

\end{thebibliography}
\end{document}